\documentclass[10pt]{amsart}
\usepackage[margin=1in]{geometry}
\usepackage[]{float}
\usepackage{empheq}
\usepackage{amsmath}
\usepackage{amsthm}
\usepackage{amssymb}
\usepackage{graphicx}
\usepackage{comment}
\usepackage[colorinlistoftodos]{todonotes}
\usepackage[colorlinks=true, allcolors=blue]{hyperref}
\newtheorem{theorem}{Theorem}[section]

\newtheorem{lemma}[theorem]{Lemma}
\newtheorem{definition}{Definition}[section]
\newtheorem{remark}{Remark}[section]
\DeclareMathSymbol{,}{\mathpunct}{operators}{"2C}
\DeclareMathSymbol{.}{\mathpunct}{operators}{"2E}
\numberwithin{equation}{section}

\begin{document}
\title{Instability and spectrum of the linearized two-phase fluids interface problem at shear flows}
\author[X. Liu]{Xiao Liu}
\address[X. Liu]{School of Mathematics, Georgia Institute of Technology, Atlanta, GA 30332}
\email{xliu458@gatech.edu}

\begin{abstract}
This paper is concerned with the 2-dim two-phase interface Euler equation linearized at a pair of monotone shear flows in both fluids. We extend the Howard's Semicircle Theorem and study the eigenvalue distribution of the linearized Euler system. Under certain conditions, there are exactly two eigenvalues for each fixed wave number $k\in \mathbb{R}$ in the whole complex plane. We provide sufficient conditions for spectral instability arising from some boundary values of the shear flow velocity. A typical mode is the ocean-air system in which the density ratio of the fluids is sufficiently small.  We give a complete picture of eigenvalue distribution for a certain class of shear flows in the ocean-air system. 
\end{abstract}
\maketitle
\begingroup
\allowdisplaybreaks
\section{Introduction}\label{S:introduction}
In this paper, We study two-phase interface system which is modeled by the two dimensional Euler problem. Both fluids are considered immiscible and incompressible (see, for example, \cite{janssen}). We use superscript $"+"$ to denote notations in the upper fluid and subscript $"-"$ for the lower one. At time $t\geq 0$, the above fluid occupies $\Omega_t^+$, and the fluid below inhabits the region $\Omega_t^-$. The fixed constants $h_\pm>0$ are the locations of flat lid of the above fluid  (see, for example, \cite{janssen}) and flat bed of the fluid below. We denote $S_t:=\partial \Omega_t^-\cap \partial\Omega_t^+=\{(t,x)| x_2=\eta(t,x_1)\}$ as the interface, which is considered as a graph of a smooth function $\eta$. In addition, we assume that 
\begin{align*}
    \Omega_t^+\cup \Omega_t^-\cup S_t=\mathbb{T}\times (-h_-, h_+),
\end{align*}
where $\mathbb{T}:=\mathbb{R}\setminus 2\pi \mathbb{Z}$. Suppose the densities of both fluids are constants $\rho^\pm>0$ for $x\in \Omega_t^\pm$. Let $v^\pm=v^\pm(t,x)\in \mathbb{R}^2$ be the velocity field and $p^\pm=p^\pm(t,x)\in \mathbb{R}$ be the pressure. Then $v^\pm$ and $p^\pm$ satisfy the two dimensional incompressible Euler system as following.
\begin{subequations}\label{E:Euler}
\begin{equation}\label{E:Euler-1}
\partial_t v^\pm+v^\pm\cdot \nabla v^\pm+\frac{1}{\rho^\pm}\nabla p^\pm+g\vec{e}_2=0 \qquad \text{in}\quad \Omega_t^\pm,
\end{equation}
\begin{equation}\label{E:Euler-2}
\nabla\cdot v^\pm=0 \qquad \text{in}\quad \Omega_t^\pm,
\end{equation}
where $g$ is the gravitational acceleration and $\vec{e}_2$ is the unit vector in $x_2$-axis. Since $v^\pm$ restricted to $S_t$ is a boundary velocity, the motion of the interface satisfies the kinematic boundary condition.
\begin{equation}\label{E:Euler-3}
\partial_t \eta=v^\pm\cdot (-\partial_{x_1}\eta, 1)^T  \qquad x\in S_t. 
\end{equation}
Considering the surface tension, we also impose the dynamic boundary condition
\begin{equation}\label{E:Euler-4}
p^+(t,x)-p^-(t,x)=\sigma \kappa \qquad x\in  S_t,
\end{equation}
where $\sigma > 0$ and $\kappa=\frac{\eta_{x_1x_1}}{(1+\eta_{x_1}^2)^{3/2}}$ is the mean curvature of $S_t$ at $x$ which corresponds to the surface tension.  We assume that the upper fluid has finite altitude. It is justified that the behavior of the flow does not strongly affect the dynamics near the lower fluid if it evanesces at high altitude (e.g. \cite{janssen}). The rigid flat lid of the upper fluid and bed of the lower fluid imply that
\begin{equation}\label{E:Euler-5}
v^\pm\cdot \vec{e}_2=0 \qquad  x_2=\pm h_\pm. 
\end{equation}
\end{subequations}
The local well-posedness theory for the interface problem (1.1) has been well studied. When surface tension is considered, the full nonlinear problem is locally well-posed (cf., \cite{ChengShkollerCoutand}, \cite{Shatah2}).
For both rotational and irrotational problem, the linearized system is ill-posed if surface tension is ignored\cite{Bea}.

It is well known that shear flows are a fundamental class of stationary solutions in the form of
\begin{equation}\label{E:Shear flow}
v_*^\pm(t,x)=(U^\pm (x_2), 0), \quad S_*=\mathbb{T}\times\{0\}.
\end{equation}

Our goal is to analyze this two-phase fluids interface system linearized at a pair of uniformly monotone shear flows. 
\begin{equation}\label{E:U monotone}
    U^\pm\in C^{l_0}, \; l_0\geq 6, \; (U^+)'\neq 0, \;\ x_2\in [0, h_+], \;\ \text{and} \;\ (U^-)'\neq 0, \;\ x_2\in [-h_-, 0]. 
\end{equation}

\begin{subsection}{Linearization}
In this subsection, we linearize the Euler system (\ref{E:Euler}) near shear flows (\ref{E:Shear flow}). The linearization was previously obtained in \cite{Oli}, but we adopt a slightly more direct approach as in \cite{LZ}.
We consider an one-parameter family of solutions $\big(S_t(\theta), v^\pm(\theta,t,x), p^\pm(\theta, t,x)\big)$ of system (\ref{E:Euler}) with 
\begin{align*}
    \big(S_t(0), v^\pm(0,t,x), \nabla p^\pm(0,t,x)\big)=(S_*, v_*^\pm, -g\rho^\pm\vec{e}_2).
\end{align*}
We abuse the notations and let  $v^\pm, p^\pm, \eta$ be the linearized solution in the following way.
\begin{align*}
    v^\pm(t,x)=\big(v_1^\pm(t,x), v_2^\pm(t,x)\big):=\lim_{\theta\rightarrow 0}\partial_\theta v^\pm(t, x),\\ p^\pm(t,x):=\lim_{\theta\rightarrow 0}\partial_\theta p^\pm(t, x), \;\ \eta(t, x_1):=\lim_{\theta \rightarrow 0 }\partial_\theta \eta(t, x_1).
\end{align*}
We differentiate (\ref{E:Euler-1}) with respect to $\theta$ and send $\theta$ to $0$,
\begin{equation}\label{E:d-Euler-theta}
    \partial_t v^\pm+v_*^\pm\cdot \nabla v^\pm+\nabla v_*^\pm\cdot v^\pm+\frac{1}{\rho^\pm}\nabla p^\pm=0.
\end{equation}
Considering the second component of this equation, we have
\begin{subequations}\label{E:LEuler}
\begin{equation}\label{E:LEuler-2}
     \partial_t v_2^\pm+U^\pm(x_2)\partial_{x_1}v_2^\pm+\frac{1}{\rho^\pm}\partial_{x_2}p^\pm=0, \qquad (x_1, \pm x_2)\in \mathbb{T}\times(0,  h_\pm).
\end{equation}
We apply the divergence operator to \eqref{E:d-Euler-theta} to obtain that
\begin{align*}
    \nabla\cdot(v_*^\pm\cdot \nabla v^\pm+v^\pm\cdot \nabla v_*^\pm)+\frac{1}{\rho^\pm}\Delta  p^\pm=0.
\end{align*}
Sending $\theta$ to 0 in the above equation, we obtain that
\begin{equation}\label{E:LEuler-1}
-\frac{1}{\rho^\pm}\Delta p^\pm=2(U^\pm)'(x_2)\partial_{x_1}v_2^\pm,\qquad  (x_1, \pm x_2)\in \mathbb{T}\times(0,  h_\pm).
\end{equation}

We differentiate (\ref{E:Euler-3}) in $\theta$ to obtain that
\begin{equation}\label{E:LEuler-3}
    \partial_t \eta= v_2^\pm-U^\pm(x_2)\partial_{x_1}\eta, \qquad \qquad \{x_2=0\}.
\end{equation}
Now we consider the variation of $S_t$. We notice that (\ref{E:Euler-4}) means that 
\begin{align*}
    p^+\big(\theta, t, x_1, \eta(\theta , t, x_1)\big)-p^-\big(\theta, t, x_1, \eta(\theta , t, x_1)\big)=\sigma \kappa\big(\theta, t, x_1, \eta(\theta, t, x_1)\big).
\end{align*}
Taking derivative with respect to $\theta $ and sending $\theta$ to 0, we obtain that
\begin{equation}\label{E:LEuler-4}
p^+(0)-p^-(0)=g\eta(\rho^+-\rho^-)+\sigma \partial_{x_1}^2\eta. 
\end{equation}
Linearizing (\ref{E:Euler-5}), we have
\begin{equation}\label{E:LEuler-5}
v_2^\pm(x_1, \pm h_\pm, t)=0. 
\end{equation}
We restrict (\ref{E:LEuler-3}) to $\pm x_2=h_\pm$ and obtain that 
\begin{equation}\label{E:LEuler-p-BC}
    \partial_{x_2} p^\pm =0, \qquad  \pm x_2=h_\pm. 
\end{equation}
\end{subequations}
The system \eqref{E:LEuler} form a two-phase fluids interface problem linearized at a pair of shear flows (\ref{E:Shear flow}). 
In the capillary gravity water wave problem ($\rho^+=0$),  the pressure $p$ can be covered by a boundary value problem of elliptic system (\ref{E:LEuler-1}), (\ref{E:LEuler-4}), and (\ref{E:LEuler-p-BC}). In this two-phase fluids interface problem($\rho^+>0$), the pressure can be expressed in terms of velocity. We refer the readers to Section 2 in \cite{SZ}. Hence, the linearized system (\ref{E:LEuler}) can be viewed as an evolutionary problem of two unknowns $(v_2, \eta)$. 
\end{subsection}
The goal of this paper is to study the instability of monotone shear flows.

\begin{subsection}{Background}
We notice that the variable coefficients in the linearized system (\ref{E:LEuler}) depend only on $x_2$. Hence, each Fourier mode in $x_1$ is decoupled from other modes. To study the linearized system, it is natural to seek eigenvalues and eigenfunctions in the form of
\begin{equation}\label{E:E-function form}
\big(v_2^\pm(t, x_1, x_2), \eta(t, x_1), p^\pm(t, x_1, x_2)\big)=\big(v_2^\pm(x_2), \eta, p^\pm(x_2)\big)e^{ik(x_1-ct)},\qquad  k\in \mathbb{Z}\setminus \{0\} 
\end{equation}
where the eigenvalues take the form  $\lambda=-ikc$ with the wave speed $c=c_R+ic_I\in \mathbb{C}$ and wave number $k$. The linear solution is spectrally unstable if a wave speed $c$, which appears in conjugate pairs, has a positive imaginary part and wave number $k>0$. Due to the symmetry of the spectrum, to find instability it suffices to consider the case of $k>0$ and seek solution with $c_I>0$. 

Due to the significance in mathematics and physics, the two-phase fluids interface problem linearized at shear flows has been studied by some mathematicians. The simplest situation is the classical Kelvin-Helmholtz model where the upper layer fluid is air and the ocean stays below. In this model, the ocean is at rest and the velocity of the air is uniform, i.e., $U^-\equiv 0$, $ U^+\equiv U_0$, for some constant $U_0$. Without surface tension, it is (linearly) unstable if $U_0\neq 0$. But with surface tension, the instability appears at very small wavelength when the wind speed becomes large. This is irrelevant with the practical phenomenon  \cite{Kelvin}. Later on, in a series of papers  \cite{Mil, miles592, miles593} written by Miles, the wind speed is assumed to be a shear flow. Miles considered  the wind generation process as the resonance phenomenon, which is the so-called quasi-laminar model. The atmosphere can be treated as a perturbation of vacuum when the density ratio is very small. If the water has infinite depth and the surface tension is ignored,  the dispersion relation is $c_k=\sqrt{g/k}$. Miles concluded that (linear) instability occurs if a critical layer exists where the unperturbed phase speed meets the wind speed at some $x_2$ level, i.e. $c_k = U^+(x_2)$, $ (U^+)'(x_2)\neq 0$, and $(U^+)''(x_2) < 0$.
B$\ddot{u}$hler-Shatah-Walsh-Zeng \cite{Oli} presented a rigorous derivation of the linearized evolution equations near a pair of shear flows. They also justified Miles's conclusion \cite{Mil}  that the existence of critical layer may enable a resonance between the shear flow in the air and capillary gravity waves in the ocean, and showed linear instability of the shear flow in the air above the stationary water. 

A related problem is the one fluid problem, which is the so-called water wave problem. There are lots of classical results on the spectra of the 2d linearized Euler equation around shear flows $(U(x_2), 0)$ in a fixed channel(cf., \cite{Dra}).  Rayleigh \cite{Ray} gave a necessary condition for linear instability that the shear velocity profile should have an inflection point. Howard \cite{How} showed that any unstable eigenvalue must have wave speed $c$ lie in the complex plane within an upper semicircle,
\begin{align*}
    (c_R-\frac{U_{max}+U_{min}}{2})^2+c_I^2\leq (\frac{U_{max}-U_{min}}{2})^2, \;\  c_I>0.
\end{align*}
For a class of shear flows, the rigorous bifurcation of unstable eigenvalues was proved, e.g., in \cite{FH98}, \cite{Lin03}. 
Lin showed that for a certain class of shear flows, the neutral limiting wave speed must be an inflection value of the velocity profile and proved the global bifurcation of unstable modes from neutral modes rigorously \cite{Lin03}. 
There are also some results regarding the linearized free boundary problem of gravity waves at shear flows. Yih \cite{Yih} showed that if the shear velocity profile $U$ is monotone without inflection value, there are no singular neutral modes with $c$ in the interior of the range of $U$. He also proved the existence of non-singular neutral modes with real wave speed outside the range of $U$ and extended the Semicircle Theorem to the linearized gravity water wave near shear flows.  For a certain class of shear flows, Hur and Lin \cite{Hur, HL13}  proved that the linear instability may occur only when the wave speed $c$ is near the value of $U$ at bottom, critical values, or inflection values of $U$. They also gave an open set of wave numbers where linear instability exists. M. Renardy and Y. Renardy \cite{Ren} found examples showing that instability can arise from these three locations through numerical methods. In the author's recent work with Zeng \cite{LZ}, we considered capillary gravity water waves linearized at monotone shear flows and proved that in contrast to the fixed channel flow, there are exactly two non-singular modes (i.e. the corresponding eigenvalue problem has nontrivial solution when $c$ is outside the range of $U$) for each wave number $k$ under certain conditions  which are real for all high wave numbers. We also obtained the bifurcation of unstable eigenvalues from inflection values of $U$ and value of $U$ at bottom. Moreover, we gave a complete picture of eigenvalue distribution for a certain class of monotone shear flows. 
\end{subsection}

\begin{subsection}{Main Results.} The goal of this paper is to study  the eigenvalue distribution of the linear system \eqref{E:LEuler} and seek linear instability. Particularly, we consider the upper layer fluid is lighter than the bottom fluid and the perturbed free surface $\eta\neq 0$. If $\eta=0$, system \eqref{E:E-problem} is equivalent to the system which consists of two Rayleigh equations with fixed boundary without interface motion. In that case of the Euler equation linearized at  monotone shear flows, the only possible locations of $c$ where instability arises are inflection values of $U^+$ or $U^-$ (e.g. \cite{Lin03}).  We first extend the Howard's Semicircle theorem to the two-phase fluid interface problem under a mild condition which holds if the fluid above is no heavier than the one below, i.e. $\rho^+\leq \rho^-$. Throughout this paper, we let 
\begin{equation}\label{E:ab}
a:= \min\{U^-\big([-h_-, 0]\big) \cup U^+\big([0, h_+]\big)\}, \;\ b:= \max\{U^-\big([-h_-, 0]\big)\cup U^+\big([0, h_+]\big)\}. 
\end{equation}

\begin{theorem}\label{T:semicircle}
Assume that $\eta\neq 0$ and the following holds.
\begin{equation}\label{E:semicircle assumption}
  -g(\rho^+-\rho^-)+\sigma k^2\geq 0.
\end{equation}
Then any unstable wave speed $c=c_R+ic_I\in \mathbb{C}$ with $c_I>0$ must stay in the following upper semicircle
\begin{equation}\label{E:semicircle}
(c_R-\frac{a+b}{2})^2+c_I^2\leq (\frac{b-a}{2})^2, \qquad c_I>0.
\end{equation}   
\end{theorem}

In seeking solutions in the form of \eqref{E:E-function form}, we treat the wave number $k\in \mathbb{R}$ as a parameter. The next theorem shows that when $|k|$ is large, there are exactly two eigenvalues  in the whole complex plane with two branches of wave speed $c^\pm(k)\in \mathbb{R}$. For a pair of monotone shear flows without inflection points, we discuss two cases, depending on whether the intersection of the range of $U^-$ and the one of $U^+$ is empty. For each case, we prove that under a certain condition, there is no singular mode and both $c^\pm(k)$ can be extended to be even and analytic functions for all $k\in \mathbb{R}$. Basically, singular modes correspond to embedding eigenvalues with wave speed $c\in U^-\big([-h_-, 0]\big)\cup U^+\big([0, h_+]\big)$. One can see the precise definition in Definition \ref{D:mode}.

\begin{theorem}\label{T:two branches}
Suppose $U^\pm\in C^3$, $(U^\pm)'\neq 0$, $\eta\neq 0$, and $0< \rho^+< \rho^-$. Let $a, b$ be defined in \eqref{E:ab}. Then the following hold.
\begin{enumerate}
    \item There exists $k_0>0$ such that for any $|k|>k_0$, there are no singular modes and exist exactly two non-singular modes with $c^\pm(k)\in \mathbb{C}$. Moreover, $c^+(k)>b$ and $c^-(k)<a$.
    \item $\lim_{|k|\rightarrow \infty} c^\pm(k)/ \sqrt{\frac{\sigma |k|}{\rho^-+\rho^+}}=\pm 1$.
    \item Assume that $(U^\pm)''\neq 0$ and 
    \begin{align*}
        \min_{c\in \{U^-(0), U^-(-h_-)\}}\int_{-h_-}^0 \frac{1}{\big(U^-(x_2)-c\big)^2}dx_2, \; \min_{c\in \{U^+(0), U^+(h_+)\}}\frac{\rho^-}{\rho^+}\int_0^{h_+} \frac{1}{\big(U^+(x_2)-c\big)^2}dx_2>\frac{1}{g}. 
    \end{align*}
    If one the following holds, 
\begin{enumerate}
    \item  $U^-\big((-h_-, 0)\big)\cap U^+\big((0, h_+)\big)=\emptyset$ and \begin{align*}  \max_{c\in \{U^\pm(0), U^\pm(\pm h_\pm)\}} \Big\{\rho^+\int_0^{h_+} \big(U^+(x_2)-c\big)^2dx_2+\rho^-\int_{-h_-}^0\big(U^-(x_2)-c \big)^2dx_2\Big\}<\sigma,
    \end{align*}
    \item For all $c\in U^-\big((-h_-, 0)\big)\cap U^+\big((0, h_+)\big)\neq \emptyset$, $(U^+)''(U^-)''>0$, and \begin{align*}  \max_{c\in \{a, b\}}\Big\{\rho^+\int_0^{h_+} \big(U^+(x_2)-c\big)^2dx_2+\rho^-\int_{-h_-}^0\big(U^-(x_2)-c \big)^2dx_2\Big\}<\sigma,
    \end{align*}
\end{enumerate}
then the two branches $c^\pm(k)$ can be extended to be even and analytic functions for all $k\in \mathbb{R}$. $c^+(k)>b$ and $c^-(k)<a$ for all $k\in \mathbb{R}$. Moreover,  $c^\pm(k)$ correspond to the only singular and non-singular modes of the linearized Euler system \eqref{E:LEuler}. 
\end{enumerate}
\end{theorem}
\begin{remark}
When $\rho^+=0$ and $(U^-)'\neq 0$, then results in Theorem \ref{T:two branches}(1) and (2) coincide the ones of capillary gravity water wave linearized at monotone shear flows which is obtained in \cite{LZ}.
\end{remark}

We notice that $c=a$ is on the boundary of the domain where the Rayleigh equation keeps regularity. Based on Theorem \ref{T:semicircle},  it might correspond to an isolated singular mode or a neutral limiting mode which is defined as the limit of a sequence of unstable modes (see Definition \ref{D:mode}). In the following theorem, we consider the case where the lighter fluid stays on top of the heavier fluid. For the case of monotone shear flows which have no inflection value, we study the behavior of branch $c^-(k)$ which is obtained in Theorem \ref{T:two branches}(1) as $|k|$ tends to $0$ from infinity.  We show that there exists $g_*\geq 0$ such that the value of $g$ relative to $g_*$ tells us when the instability happens near $c=a$, i.e. the minimal value of $U^+\big([0, h_+]\big)\cup U^-\big([-h_-, 0]\big)$. More precisely, if $g\geq g_*$, there is no instability when $c$ is near $a$. If $g<g_*$, under an additional assumption that $(U^\pm)''>0$, instability arises from $c=a$ under some conditions.  In this case, if $(U^\pm)''<0$, there is no instability near $c=a$. In the capillary gravity water wave problem linearized at monotone shear flows, according to Theorem 1.1 in \cite{LZ}, there exists $g_\#$ which provides a sharp condition for linear instability arising from the value of velocity profile at the bottom. More precisely, when $g\in (0, g_\#)$, an increasing convex shear flow  is spectrally unstable and an increasing concave shear flow is spectrally stable. When $g\geq g_\#$, monotone shear flow is always spectrally stable if the basic velocity profile has no inflection value. Our result for $\rho^+=0$ coincides this result and in this case $g_*=g_\#$. 

\begin{theorem}\label{T:instability fixed e}
Assume that $U^\pm\in C^6$, $(U^\pm)'\neq 0, (U^\pm)''\neq 0$, $\eta\neq 0$, and $0< \rho^+< \rho^-$. Let $c^-(k)$ be obtained in Theorem \ref{T:two branches}(1) and $a=\min U^-([-h_-, 0])\cup U^+([0, h_+])$.  Then, there exists $g_*> 0$
such that the following hold. \begin{enumerate}
    \item If $g>g_*$, then $c^-(k)$ can be extended as an even and analytic function such that for all $k\in \mathbb{R}$, $c^-(k)<a$. There is no singular mode at $c=a$ for all $k\in \mathbb{R}$.
    \item If $g=g_*$, there exists a unique $k_*>0$ such that $c^-(k)$ can be extended to be an even and $C^{1, \alpha}$ function for all $k\in \mathbb{R}$ (for any $\alpha\in [0, 1)$) such that it is analytic everywhere except $k=\pm k_*$ where  $c^-(\pm k_*)=a$ for all $k\in \mathbb{R}$.
    \item  If one of the following assumption is further satisfied,
    \begin{enumerate}
        \item $a=U^+(0)\notin U^-([-h_-, 0])$ and 
        \begin{align*}
            \frac{\rho^-}{g(\rho^--\rho^+)}<\int_{-h_-}^0 \frac{1}{\big( U^-(x_2)-U^+(0)\big)^2}dx_2;
        \end{align*}
        \item $a=U^-(-h_-)$ or $U^-(0)$;
        \item $a=U^+(h_+)\notin U^-([-h_-, 0])$ and 
        \begin{align*}
            \frac{\rho^-}{g(\rho^--\rho^+)}<\int_{-h_-}^0 \frac{1}{\big( U^-(x_2)-U^+(h_+)\big)^2}dx_2,
        \end{align*}
    \end{enumerate}
    then for $g\in (0, g_*)$, there exist $k_*^+>k_*^->0$ and $\delta>0$ such that we have the following. 
    \begin{enumerate}
        \item [(i)] Assume that $(U^\pm)''>0$. Then $c^-(k)$ can be extended to be an even and $C^{1, \alpha}$ function for all $|k|\geq k_*^+-\delta$ and analytic except at $k=\pm k_*^+$ and it satisfies that
        \begin{equation}\label{E:a-bifurcation-c-}
            c^-(\pm k_*^+)=a,\quad  c^-(k)<a, \; \forall |k|>k_*^+,
        \end{equation}
        \begin{equation}\label{E:cI-a}
             c_I^-(k)>0,\; c_R^-(k)>a, \;\forall |k|\in [k_*^+-\delta, k_*^+).
        \end{equation}
        And there exists a $C^{1, \alpha}$ function $\mathcal{C}(k)$ on $[0, k_*^-+\delta]$ and analytic except at $k=\pm k_*^-$ such that 
        \begin{equation}\label{E:a-bifurcation-tilde-c}
            \mathcal{C}(\pm k_*^-)=a, \quad \mathcal{C}(k)<a, \; \forall |k|\in [0, k_*^-), 
        \end{equation}
        \begin{equation}\label{E:tilde-C-I}
         \mathcal{C}_I(k)>0, \; \mathcal{C}_R(k)>a, \; \forall |k|\in (k_*^-, k_*^-+\delta].
        \end{equation}
        Moreover, there exists $\gamma>0$ such that for $c_R<a+\gamma$, all singular and non-singular modes with $|k|\geq k_*^+-\delta$ are $c^-(k)$ and the ones with $|k|\in [0, k_*^-+\delta]$ are $\mathcal{C}(k)$. 
    \item[(ii)] Assume that $(U^\pm)''<0$. Then $c^-(k)$ can be extended to be an even and $C^{1, \alpha}$ function for all $|k|\geq k_*^+$, analytic except at $k=\pm k_*^+$, and \eqref{E:a-bifurcation-c-} holds.  And there exists a $C^{1, \alpha}$ function $\Tilde{C}(k)$ on $[0, k_*^-]$ and analytic except at $k=\pm k_*^-$ such that \eqref{E:a-bifurcation-tilde-c} holds. In addition, for $c\leq a$, all singular and non-singular modes with $|k|\geq k_*^+$ are $c^-(k)$ and the ones with $|k|\in [0, k_*^-]$ are $\Tilde{C}(k)$. 
    \end{enumerate}
\end{enumerate}
\end{theorem}
\begin{remark}1). In Theorem \ref{T:instability fixed e}(3), $c^-(k)$ may not be extended to be a $C^{1, \alpha}$ function for all $k\in \mathbb{R}$ since there might exist an isolated neutral limiting mode (see Definition \ref{D:mode}) with $c\neq a$ (see Lemma \ref{L:nlm-location-fixed-ep}), which may break the continuation of the bifurcation curve.  2). Let $\rho^+=0$. The result in Theorem \ref{T:instability fixed e}(1), (2), and (3) for $a=U^-(-h_-)$ coincides the one of capillary gravity water wave linearized at monotonically increasing and convex shear flows obtained in \cite{LZ}. 
\end{remark}

When $\epsilon:=\frac{\rho^+}{\rho^-}\ll 1$, we consider the two-phase fluid interface problem as a perturbation of capillary gravity water wave linearized at monotone shear flows. 
%From above result, we know that the left branch $c^-(k)$ is real and on the left of $U^\pm$ for large $|k|$. As $k$ getting small from infinity, we study how $c^-(k)$ changes for each small nonzero $\epsilon$. We introduce $g_*$ and show that how big $g$ is relative to $g_*$ tells us when the instability arises from the minimal value of $U^+$ and $U^-$. When $g>g_*$, $c^-(k)$ can not enter the range of $U^\pm$. When $g=g_*$, $c^-(k)$ can touch $a$ which the minimum value of $U^\pm$ but no linear instability occurs. For all $k\in \mathbb{R}$, it still stays on the left. But when $0<g<g_*$, $c^-(k)$ can enter the range of $U^\pm$ and linear instability happens.  In the statement (3) of the following theorem, 
In the following, we discuss a special case where the speed of the monotonically increasing shear flow in lower fluid is strictly slower than the one in upper fluid. We prove that if the one fluid free boundary problem  has a neutral mode with wave speed $c>U^-(0)$ in the range of $U^+$, which means that there exists a critical layer location $x_2\in [0, h_+]$,  then small $\epsilon$ leads to unstable mode near such wave speed in the ocean-air system. This is due to the resonance between shear flows in the air and ocean. Such situation can be viewed as an extension of the critical layer phenomenon studied by Miles \cite{Mil, miles592, miles593} and B$\ddot{u}$hler-Shatah-Walsh-Zeng \cite{Oli}. We know that in the one fluid problem, the right branch of wave speed $c^+_0(k)$ is real and larger than $U^-(0)$ for all $k\in \mathbb{R}$. The eigenvalue distribution of the ocean-air system depend on the location of the neutral mode with zero wave number (i.e. $c_0$ defined in the below)  in the one fluid free boundary problem and the sign of $(U^+)''$.  If $(U^+)''<0$, then the existence of critical layer in the above sense in the range of $U^+$ leads to instability.  We prove that if $(U^+)''>0$ and $c^+(0)\in \big(U^+(0), U^+(h_+)]$, then then for small $\epsilon>0$,  as $|k|$ becomes small and tends to $0$, the right branch $c^+(k)$ of the two-phase interface problem disappears when it reaches $U^+(h_+)$. If $(U^+)''>0$ and 
$c_0^+(0)\in \big(U^-(0), U^+(0)\big)$, then  $c^+(k)$ of  disappears when it reaches $U^+(h_+)$ and reappears at $U^+(0)$. Meanwhile, we have more precise result of the other branch $c^-(k)$ compared with the one in Theorem \ref{T:instability fixed e}, which depends on the relative value of $g$ and $g_\#$(mentioned before Theorem \ref{T:instability fixed e}). Under a certain condition at $c=U^-(-h_-)$, if $g>g_\#$, for small $\epsilon$, $c^-(k)$ can still be extended for all $k\in \mathbb{R}$ and $c^-(k)<U^-(h_-)$. When $g\in (0, g_\#]$ (with an additional assumption when $g=g_\#$), if $(U^-)''>0$, instability occurs when $c^-_R(k)$ enters and then leaves the range of $U^-$ through $U^-(-h_-)$ as $|k|$ tends to $0$ from infinity for $0<\epsilon\ll 1$. If $(U^-)''<0$, $c^-(k)$ disappears when it touches $U^-(-h_-)$ and then reappears for small $|k|$.  

\begin{theorem}\label{T:e-small-g*}Assume that $U^\pm$ satisfies \eqref{E:U monotone}, $(U^\pm)'>0$, $U^-(0)<U^+(0)$,  $(U^\pm)''\neq 0$, and $\eta\neq 0$. 
\begin{enumerate}
    \item Assume that 
\begin{equation}\label{A:u+h+-bifurcation}
    \int_{-h_-}^0\big(U^-(x_2)-U^+(h_+)\big)^2dx_2<\frac{\sigma}{\rho^-}.
\end{equation}
Let $c_0>U^-(0)$ be the unique solution to 
\begin{equation}\label{A:c-solution-g}
    \int_{-h_-}^0 \frac{1}{\big(U^-(x_2)-c\big)^2}dx_2 =\frac{1}{g}.
\end{equation}
Let $c^+(k)$ be obtained in Theorem \ref{T:two branches}. Then the following hold. 
\begin{enumerate}
    \item If $c_0>U^+(h_+)$, there exists $\epsilon_0>0$ such that for any $\frac{\rho^+}{\rho^-}\in (0, \epsilon_0]$,  $c^+(k)$ can be extended to be an even and analytic function for all $k\in \mathbb{R}$  and $c^+(k)>U^+(h_+)$. It is the only singular and non-singular mode for $c_R>U^-(0)$.
    \item If $c_0\in (U^+(0), U^+(h_+)]$, there exists $\epsilon_1>0$ such that for any $\frac{\rho^+}{\rho^-}\in (0, \epsilon_1]$, there exists a unique $k_*> 0$ such that  $c^+(\pm k_*)=U^+(h_+)$ and the following holds.
    \begin{enumerate}
        \item If $(U^+)''<0$, then $c^+(k)$ can be extended to be an even and  $C^{1, \alpha}$ (for all $\alpha\in [0, 1)$) function for all $k\in \mathbb{R}$ and analytic in $k$ except at $\pm k_*$. $c^+(k)$ satisfies that
    \begin{align*}
     c^+(k)>U^+(h_+), \; \forall |k|\in (k_*, \infty), \quad c_I^+(k)>0, c_R^+(k)\in \big(U^+(0), U^+(h_+)\big), \; \forall |k|\in [0, k_*].
    \end{align*}
    \item If $(U^+)''>0$, $c^+(k)$ can be extended to be a real valued $C^{1, \alpha}$ function (for all $\alpha\in [0, 1)$) for $|k|\geq k_*$ and analytic in $k$ if $|k|>k_*$. 
    \end{enumerate}
    Moreover, $c^+(k)$ is the only singular and non-singular mode for $c_R>U^+(0)$. 
    \item Suppose $c_0\in (U^-(0), U^+(0)]$. Then there exists $\epsilon_2>0$ such that for any $\epsilon\in (0, \epsilon_2]$, there exist $k_*^+>k_*^-> 0$ such that $ c^+(k_*^+)=U^+(h_+),   c^+(k_*^-)=U^-(0)$. 
    \begin{enumerate}
        \item If $(U^+)''<0$, then $c^+(k)$ can be extended as an even $C^{1, \alpha}$ function ($\forall\alpha\in[0, 1)$) for all $k\in \mathbb{R}$ and analytic in $k$ except at $k=\pm k_*^\pm$ such that 
    \begin{align*}
        c_I^+(k)>0, \quad  c_R^+(k)\in \big(U^+(0), & U^+(h_+)\big), \; \forall |k|\in (k_*^-, k_*^+),\\  c^+(k)>U^+(h_+), \; \forall |k|\in (k_*^+, \infty),  \quad c^+(k)&\in \big( U^-(0), U^+(0)\big), \; \forall |k|\in [0,  k_*^-). 
    \end{align*}
    \item If $(U^+)''>0$, then $c^+(k)$ can be extended as an even $C^{1, \alpha}$ real valued function ($\forall \alpha\in [0, 1)$) for $|k|\geq k_*^+$ and analytic in $k$ if $k\neq \pm k_*^+$. And $c^+(k)>U^+(h_+)$ for all $|k|>k_*^+$. In addition, there exists an even real function $\mathcal{C}(k)$ which is $C^{1, \alpha}$ in $k$ for all $|k|\leq k_*^-$ and analytic in $k$ except at $k=\pm k_*^-$ such that $\mathcal{C}(\pm k_*^-)=U^+(0)$. And $\mathcal{C}(k)\in \big( U^-(0), U^+(0)\big)$ for all $|k|<k_*^-$. 
    \end{enumerate}
    Moreover, for $c_R>U^-(0)$, all the singular and non-singular modes are $c^+(k)$ and $\mathcal{C}(k)$. 
\end{enumerate}
\item There  exists $g_\#>0$ such that the following hold. 
\begin{enumerate}
    \item If $g>g_\#$, there exists $\epsilon_3>0$ such that for any $\frac{\rho^+}{\rho^-}\in (0, \epsilon_3]$, $c^-(k)$ can be extended to be an even and analytic function for all $k\in \mathbb{R}$ and $c^-(k)<U^-(-h_-)$. Moreover, it is the only singular and non-singular mode for $c_R\leq U^-(0)$. 
    \item If $g\in (0, g_\#)$, 
    there exists $\epsilon_4>0$ such that for any $\epsilon\in (0, \epsilon_4]$, there exist $k_\#^+>k_\#^->0$ such that $c^-(\pm k_\#^\pm)=U^-(-h_-)$. In addition, 
    \begin{enumerate}
        \item If $(U^-)''>0$, then $c^-(k)$ can be extended as an even $C^{1, \alpha}$ function (for any $\alpha\in [0, 1)$) for all $k\in \mathbb{R}$ and analytic in $k$ except at $k=\pm k_\#^\pm$ such that 
        \begin{align*}
            c^-(k)<U^-(-h_-), \; \forall |k|\in (k_\#^+, \infty)\cup [0, k_\#^-), \quad c_I^-(k)>0, \; c_R^-(k)>U^-(-h_-), \; \forall |k|\in (k_\#^-, k_\#^+). 
        \end{align*}
        \item If $(U^-)''<0$, then $c^-(k)$ can be extended as an even $C^{1, \alpha}$ real valued function (for any $\alpha\in [0, 1)$) for $|k|\geq k_\#^+$, analytic in $k$ except at $k=\pm k_\#^+$ and $c^-(k)<U^-(-h_-)$ for all $|k|>k_\#^+$. There exists a $C^{1, \alpha}$ function $\Tilde{C}(k)$ on $[0, k_\#^-]$ and analytic except at $k=\pm k_\#^-$ such that $\Tilde{C}(k)<U^-(-h_-)$ for all $|k|\in [0, k_\#^-)$.
    \end{enumerate}
    Moreover, for $c_R\leq U^-(0)$, all singular and non-singular modes are $c^-(k)$ and $\Tilde{C}(k)$. 
\end{enumerate}
\end{enumerate}
\end{theorem}
\begin{remark}
1). If there exists $\gamma>0$ such that $\gamma<\theta\min \{U^+(h_+)-U^+(0), U^+(0)-U^-(0)\}$ for any $\theta\in (0, 1)$ and $c_0\in [U^-(0)+\gamma, U^+(0)]\cup \big(U^+(0)+\gamma, U^+(h_+)]\cup [U^+(h_+)+\gamma, \infty\big)$ , then $\epsilon_0$ in (1), $\epsilon_1$ in (2), and $\epsilon_2$ in (3) can be the same and independent of location of $c_0$.
\end{remark}

\end{subsection}

\begin{subsection}{Outline of the proof.} 
The proof of Theorem \ref{T:semicircle} is presented in Section \ref{S:preliminary}. To prove Theorem \ref{T:two branches}, we first generalize some existing results we obtained in \cite{LZ} on the fundamental solutions $y^\pm(c, k, x_2)$ to classical Rayleigh equation and crucial quantities $Y^\pm(c, k):=
\frac{(y^\pm)'(0)}{y^\pm(0)}(c, k)$. The free boundary condition on the interface \eqref{E:E-problem-f} which gives the relation between wave speed $c$ and wave number $k$ can be written in terms of $Y^\pm(c, k)$ which are well defined if both $(U^+)''\neq 0$ and $(U^-)''\neq 0$ or $|k|$ is large. This allows us to analyze the function $F(c, k, \epsilon)$ which is defined in \eqref{E:F-definition} based on the properties of $Y^\pm(c, k)$. The zeros of $F(c, k, \epsilon)$ correspond to the eigenvalues of the linearized system \eqref{E:LEuler}. To discuss the eigenvalue distribution, we treat the wave number $k$ as a parameter and start with large $|k|$. We apply the Contraction Mapping Theorem on the region outside the range of $U^\pm$ to prove statement (1) and (2)(See Lemma \ref{L:c-large k}). Finally, after ruling out all the possible singular modes under certain conditions, we  prove that without singular modes, the two branches wave speed $c^\pm(k)$ can be extended for all $k\in \mathbb{R}$ in statement (3). This follows directly from Lemma \ref{L:no singular mode}. 

To prove Theorem \ref{T:instability fixed e}, we study the properties of $Y^\pm(c, k)$ and explicit representation of $y^+(c, 0)$(resp. $y^-(c, 0)$) when $c\notin U^+\big((0, h_+)\big)$(resp. $c\notin U^-\big(-h_-, 0\big)$). Thanks to the monotonicity of $\partial_k Y^\pm$ and concavity of $F(c, k, \epsilon)$ for $k\geq 0$, we reduce the analysis to the case of $k=0$. And then we apply a bifurcation analysis near $c=a$ to study the behavior of $c^-(k)$ locally. Finally, by applying a  complex continuation argument of a holomorphic function, we extend $c^-(k)$ as $|k|$ becomes small.  The proof of Theorem \ref{T:instability fixed e} follows from a series of Lemmas  in Section \ref{S:distribution of eigenvalues}.

To prove Theorem \ref{T:e-small-g*}, we consider the interface problem as a perturbation of the one fluid free boundary problem and study a special case where both speed of wind and the one of ocean are monotonically increasing and have no inflection value, and at the interface, the speed of ocean is strictly slower than the one of air. After overcoming the degeneracy at $c=U^+(0)$ and $c=U^+(h_+)$, we prove the instability arising from these locations under a certain condition by applying the Implicit Function Theorem and Mean Value Theorem. Based on the eigenvalue distribution of the capillary gravity water wave linearized at monotone shear flows(see \cite{LZ}), we apply complex continuation argument of a holomorphic function and give a complete picture of eigenvalue distribution.  The proof of Theorem \ref{T:e-small-g*} is completed in Section \ref{S:Ocean-Air Model}. 
\end{subsection}

\begin{section}{Preliminary}\label{S:preliminary}
To study the linear system \eqref{E:LEuler}, we seek eigenvalues and eigenfunctions in the form of \eqref{E:E-function form} and consider the corresponding eigenvalue problem. Using \eqref{E:E-function form} in \eqref{E:LEuler-1}, \eqref{E:LEuler-2} and \eqref{E:LEuler-4}, we have 
\begin{subequations}
\begin{equation}\label{E:p-computation}
    \frac{1}{\rho^\pm}(k^2-\partial_{x_2}^2)p^\pm (x_2)=2ik(U^\pm)'(x_2)v_2^\pm(x_2),
\end{equation} 
\begin{equation}\label{E:v2-computation}
    ik\big(U^\pm(x_2)-c\big)v_2^\pm(x_2)+\frac{1}{\rho^\pm}(p^\pm)'(x_2)=0,
\end{equation}
\begin{equation}\label{E:boundary-computation}
    p^+(0)-p^-(0)=g\eta(\rho^+-\rho^-)-\sigma \eta k^2.
\end{equation}
\end{subequations}
We compute $\partial_{x_2}\eqref{E:p-computation}-(k^2-\partial_{x_2}^2)\eqref{E:v2-computation}$ to obtain the following classical Rayleigh equations.
\begin{subequations}\label{E:E-problem}
\begin{equation}\label{E:E-problem-1}
-(v_2^\pm)''(x_2)+\big(k^2+\frac{(U^\pm)''(x_2)}{U^\pm(x_2)-c}\big)v_2^\pm(x_2)=0, \qquad \pm x_2\in [0,  h_\pm]. 
\end{equation}
From \eqref{E:LEuler-5} and \eqref{E:LEuler-3}, we obtain that
\begin{equation}\label{E:E-problem-2}
v_2^\pm (\pm h_\pm)=0
\end{equation}  
\begin{equation}\label{E:E-problem-3}
ik\eta\big(U^\pm(0)-c\big)=v_2^\pm(0). 
\end{equation}
We evaluate $\partial_{x_2}\eqref{E:v2-computation}+\eqref{E:p-computation}$ at $x_2=0$. Combining with the last equation, we obtain the boundary condition at the interface. 
\begin{equation}\label{E:E-problem-4}
\begin{split}
ik\eta\big(g(\rho^+-\rho^-)-\sigma k^2\big)= &\rho^-(U^-)'(0)v_2^-(0)-\rho^+ (U^+)'(0)v^+_2(0)+\rho^+ (v_2^+)'(0)\big(U^+(0)-c\big)\\
&-\rho^-(v_2^-)'(0)\big(U^-(0)-c\big). 
\end{split}
\end{equation}
\end{subequations}
\begin{remark}
To seek unstable solutions, we consider $k\neq 0$ in \eqref{E:E-function form} because zeroth mode corresponds to variation towards nearby shear flows. But we may still study the solution of system \eqref{E:E-problem} with $k=0$ since it is helpful to obtain a picture of the eigenvalue distribution for all integer $k$.  
\end{remark}

We are interested in looking for linearized unstable solutions with $\eta\neq 0$ and $k\neq 0$. Without loss of generality, we normalize it by taking
\begin{align*}
ik\eta=1.
\end{align*}
Meanwhile, we let $\epsilon:=\frac{\rho^+}{\rho^-}$ be the density ratio. Then we can rewrite \eqref{E:E-problem-3} and \eqref{E:E-problem-4} in the following.
\begin{subequations}\label{E:E-problem-n}
\begin{equation}\label{E:E-problem-n3}
v_2^\pm(0)=U^\pm(0)-c,
\end{equation}
\begin{equation}\label{E:E-problem-n4}
  \begin{split}
      g(\epsilon-1)-\frac{\sigma}{\rho^-} k^2= &(U^-)'(0)v_2^-(0)-\epsilon (U^+)'(0)v^+_2(0)+ \epsilon (v_2^+)'(0)\big(U^+(0)-c\big)\\
&-(v_2^-)'(0)\big(U^-(0)-c\big).
  \end{split}  
\end{equation}
\end{subequations}
We summarize the above calculation in the following.

\begin{lemma}\label{L:E-value form}
For $k\in \mathbb{Z}\setminus\{0\}$,  $-ikc$ with $c\notin U^-([-h_-, 0])\cup U^+([0, h_+])$ is an eigenvalue of the linearized Euler system \eqref{E:LEuler} at the shear flow $v_*^\pm=(U^\pm(x_2), 0)$ if  there exists a non-trivial solution of \eqref{E:E-problem-1}, \eqref{E:E-problem-2}, and \eqref{E:E-problem-n}.
\end{lemma}
\begin{remark}\label{R:symmetry-ikc}
The solutions to system \eqref{E:E-problem-1}, \eqref{E:E-problem-2}, and \eqref{E:E-problem-n} are even in $k$. In addition, the complex conjugate of solutions are still solutions and $c$ is replaced by $\Bar{c}$. 
\end{remark}

Using the symmetry in Remark \ref{R:symmetry-ikc}, we notice that the existence of a solution to system \eqref{E:E-problem-1}, \eqref{E:E-problem-2}, and \eqref{E:E-problem-n} with $c_I>0$ and $k>0$ implies exponential linear instability of Euler system. In the rest of paper, $k$ is assumed non-negative mostly.  To study instability, we introduce some concepts as following. 

\begin{definition} \label{D:mode}
A pair $(c, k)$ with  real $c$, positive $k$ is said to be a neutral mode if there exists a pair of  non-trivial solutions to \eqref{E:E-problem-1}, \eqref{E:E-problem-2}, and \eqref{E:E-problem-n}. It is called an unstable mode if $c_I>0$. We call it a neutral limiting mode if it is the limit of a sequence of unstable modes $\{(c_n, k_n)\}$. Furthermore, if $c\in U^+\big([0, h_+]\big)\cup U^-(\big[-h_-, 0]\big)$, it is called a singular mode. Otherwise, it is a non-singular mode.
\end{definition}

In the following, we prove Theorem \ref{T:semicircle} which shows that any unstable mode must have wave speed in an upper semicircle under a mild condition. 

\noindent \textit{Proof of Theorem \ref{T:semicircle}}.
Suppose $c_I\neq 0$. Let $\psi^\pm(x_2)$ be defined by $v_2^\pm(x_2)=\big(c-U^\pm(x_2)\big)\psi^\pm(x_2)$. Then $\psi^\pm(x_2)$ satisfy
\begin{equation}\label{E:psi}
\Big(\big(U^\pm(x_2)-c\big)^2(\psi^\pm) '(x_2)\Big)'(x_2)-k^2\big(U^\pm(x_2)-c\big)^2\psi^\pm(x_2)=0,
\end{equation}
along with
\begin{equation}\label{E:bottom psi}
\begin{split}
\psi^\pm(\pm h_\pm)&=0, \qquad \psi^\pm(0)=-1,
\end{split}
\end{equation}
\begin{equation}\label{E:free boundary psi}
\begin{split}
-g(\rho^+-\rho^-)+\sigma k^2=\rho^+\big(U^+&(0)-c\big)^2(\psi^+)'(0)-\rho^-\big(U^-(0)-c\big)^2(\psi^-)'(0).
\end{split}
\end{equation}
We multiply \eqref{E:psi} by $\overline{\psi^\pm}$ and integrate from $0$ to $\pm h_\pm$ respectively and obtain that
\begin{subequations}\label{E:int}
\begin{equation}\label{E:int+}
-\big(U^-(0)-c\big)^2(\psi^-)'(0)=\int_{-h_-}^0 (U^--c)^2(k^2|\psi^-|^2+|(\psi^-)'|^2)dx_2,
\end{equation} 
\begin{equation}\label{E:int-}
\big(U^+(0)-c\big)^2(\psi^+)'(0)=\int_0^{h_+}(U^+-c)^2 (k^2|\psi^+|^2+|(\psi^+)'|^2)dx_2. 
\end{equation}
\end{subequations}
Let $Q^\pm:= k^2|\psi^\pm|^2+|(\psi^\pm)'|^2$. Then $Q^\pm >0$. We consider imaginary part of \eqref{E:int+} and \eqref{E:int-}.
\begin{subequations}\label{E:int-Im}
\begin{equation}\label{E:int-Im-}
Im\Big(-\big(U^-(0)-c\big)^2(\psi^-)'(0)\Big)=-2c_I\int_{-h_-}^0 (U^--c_R)Q^-dx_2, 
\end{equation} 
\begin{equation}\label{E:int-Im+}
Im \Big(\big(U^+(0)-c\big)^2(\psi^+)'(0)\Big)=-2c_I\int_0^{h_+} (U^+-c_R)Q^+dx_2. 
\end{equation}
\end{subequations}
We compute $\rho^- \eqref{E:int-Im-}+\rho^+ \eqref{E:int-Im+}$, use equation \eqref{E:free boundary psi}, and obtain that
\begin{equation}\label{E:U<->cR}
\int_0^{h_+}\rho^+ U^+ Q^+dx_2+\int_{-h_-}^0 \rho^- U^- Q^-dx_2=c_R(\int_0^{h_+}\rho^+ Q^+dx_2+\int_{-h_-}^0 \rho^- Q^-dx_2),
\end{equation}
Similarly, we consider the real part of  \eqref{E:int+} and \eqref{E:int-}.
\begin{subequations}\label{E:int-Re}
\begin{equation}\label{E:int-Re-}
Re\Big(-\big(U^-(0)-c\big)^2(\psi^-)'(0)\Big)=\int_{-h_-}^0 \big((U^--c_R)^2-c_I^2\big)Q^-dx_2, 
\end{equation} 
\begin{equation}\label{E:int-Re+}
Re \Big(\big(U^+(0)-c\big)^2(\psi^+)'(0)\Big)=\int_0^{h_+} \big((U^+-c_R)^2-c_I^2\big)Q^+dx_2. 
\end{equation}
\end{subequations}
We compute $\rho^+ \eqref{E:int-Re+}+\rho^- \eqref{E:int-Re-}$ and use (\ref{E:free boundary psi}) to obtain that
\begin{align*}\label{E:U<->c}
\int_0^{h_+}U^{+2}\rho^+Q^+dx_2+\int_{-h_-}^0 U^{-2}\rho^-Q^-dx_2=&(c_R^2+c_I^2) (\int_0^{h_+}\rho^+ Q^+dx_2+\int_{-h_-}^0 \rho^- Q^-dx_2)\\
&-g(\rho^+-\rho^-)+\sigma k^2, \tag{2.11}
\end{align*}
where we used equation \eqref{E:U<->cR} and real part of (\ref{E:free boundary psi}). Now we consider the inequality
\begin{align*}
\int_0^{h_+}\rho^+ Q^+(U^+-a)(U^+-b)dx_2+\int_{-h_-}^0 \rho^- Q^-(U^--a)(U^--b)dx_2\leq 0,
\end{align*}
where $a$ and $b$ are defined in (\ref{E:ab}). Substituting \eqref{E:U<->cR} and \eqref{E:U<->c}, we have
\begin{align*}
\big(c_R^2+c_I^2-(a+b)c_R+ab\big) (\int_0^{h_+}\rho^+ Q^+dx_2+\int_{-h_-}^0 \rho^- Q^-dx_2)\leq -\big(g(\rho^--\rho^+)+\sigma k^2\big)\leq 0.
\end{align*}
Since $Q^\pm\geq 0$ and  can not be both identically zero, $(c_R-\frac{a+b}{2})^2+c_I^2\leq (\frac{a-b}{2})^2$.

\hfill $\square$

To study the instability, we will address the possible locations of the limit of a sequence of unstable solutions to \eqref{E:E-problem-1}, \eqref{E:E-problem-2}, and \eqref{E:E-problem-n}, i.e. neutral limiting mode.  In fact, not every neutral mode serves as a neutral limiting mode. First, we  consider and study the fundamental solutions of classical Rayleigh equations. Some results about the fundamental solutions were obtained in \cite{LZ}. Let $y^+(c, k, x_2)$ be the solution to  \eqref{E:E-problem-1} on $[0, h_+]$ under the following boundary condition.
\begin{subequations}\label{E:fundamental solution}
    \begin{equation}\label{E:y+}
y^+(h_+)=0, \quad (y^+)'(h_+)=1.
\end{equation}
Suppose that $y^-(c, k, x_2)$ be the solution to \eqref{E:E-problem-1} on $[-h_-, 0]$ satisfying
\begin{equation}\label{E:y-}
    y^-(-h_-)=0, \quad (y^-)'(-h_-)=1. 
\end{equation}
\end{subequations}
When $c\in U^+\big([0, h^+]\big)$ (or $U^-\big([-h_-, 0]\big)$), $y^+(c, k, x_2)$ (resp. $y^-(c, k, x_2)$) is defined as $\lim_{\theta\rightarrow 0^+}y^+(c+i\theta, k, x_2)$ (resp. $\lim_{\theta\rightarrow 0^+} y^-(c+i\theta, k, x_2)$). It satisfies the following condition for $ U^+(x_2)\neq  c$ (resp. $U^-(x_2)\neq c$).
\begin{align*}
    \lim_{\theta\rightarrow 0+}\big( (y_R^+)'(c, k, x_c+\theta)-(y_R^+)'(c, k, x_c-\theta)\big) =0, \\
    \lim_{\theta\rightarrow 0+}(y_I^+)'(c, k, x_c-\theta)=-\pi \frac{(U^+)''(x_c)}{|(U^+)'(x_c)|}y^+(x_c), 
\end{align*}
where $ x_c=(U^+)^{-1}(c)$ is the preimage of $c\in U^+([0, h_+])$  (resp. $\lim_{\theta\rightarrow 0+}\big((y_R^-)'(c, k, x_c+\theta)-(y_R^-)'(c, k, x_c-\theta)\big)=0,\;
    \lim_{\theta\rightarrow 0+}(y_I^-)'(c, k, x_c+\theta)=\pi \frac{(U^-)''(x_c)}{|(U^-)'(x_c)|}y^-(x_c)$, $x_c=(U^-)^{-1}(c)$). 
Moreover, because of the singularity of \eqref{E:E-problem-1}, $(y_R^\pm)'(x_2)$ has a logarithmic singularity $log|x_2-x_c|$ near $x_c$.

We also define the following crucial quantities which are related to Reynold stress. We let
\begin{subequations}\label{E:Y-definition}
\begin{equation}\label{E:Y-definition-}
\begin{split}
    Y^-(c, k)=Y_R^-(c, k)+iY_I^-(c, k):=\frac{(y^-)'(0)}{y^-(0)}(c, k), \quad c\in \mathbb{C}\setminus U^-\big([-h_-, 0]\big), \\
    Y^-(c, k)=\lim_{\beta\rightarrow 0+}Y^-(c+i\beta, k) \qquad c\in U^-\big([-h_-, 0)\big),
\end{split}
\end{equation}
\begin{equation}\label{E:Y-definition+}
\begin{split}
    Y^+(c, k)=Y_R^+(c, k)+iY_I^+(c, k):=\frac{(y^+)'(0)}{y^+(0)}(c, k), \quad c\in \mathbb{C}\setminus U^+\big([0, h_+]\big), \\
    Y^+(c, k)=\lim_{\beta\rightarrow 0+}Y^+(c+i\beta, k) \qquad c\in U^+\big((0, h_+]\big),
\end{split}
\end{equation}
\end{subequations}

We should notice that $Y^-$ (or $Y^+$) is not well defined at $c=U^-(0)$ (resp. $c=U^+(0)$) because of the singularity of \eqref{E:E-problem-1} at $x_2=0$. The domain of $Y^\pm$ is given by 
\begin{equation}\label{E:Y-domain}
    D(Y^\pm)=\{(c, k)\in \mathbb{C}\times \mathbb{R} \mid c\neq U^\pm(0), y^\pm(c, k, 0)\neq 0\}.
\end{equation}
In $D(Y^\pm)$, we can simplify the free boundary condition of our eigenvalue problem \eqref{E:E-problem-n4}. 
\begin{equation}\label{E:E-problem-f}
    g(1-\epsilon)+\frac{\sigma}{\rho^-}k^2=\epsilon\Big((U^+)'(0)\big(U^+(0)-c\big)-Y^+\big(U^+(0)-c)^2\Big)-(U^-)'(0)\big(U^-(0)-c\big)+Y^-\big(U^-(0)-c\big)^2. 
\end{equation}
This allows us to focus on the system consisting of \eqref{E:E-problem-1}, \eqref{E:fundamental solution}, and \eqref{E:E-problem-f}.  We summarize the analysis above in the following lemma. 
\begin{lemma}\label{L:E-value-n}
The linearization of the interface problem \eqref{E:Euler} at the shear flow \eqref{E:Shear flow} has an eigenvalue $-ikc$ if the solution $y^\pm(c, k, x_2)$ of \eqref{E:E-problem-1} satisfy \eqref{E:fundamental solution} and the pair $(c, k)$ satisfies \eqref{E:E-problem-f}.
\end{lemma}

For convenience,  we introduce notations $x_c^\pm$ as following.
\begin{definition}\label{D:xc}
If $c\in U^-\big([-h_-, 0] \big)$, let $x_c^-$ be $c=U^-(x_c^-)$. If $c\in U^+\big([0, h_+] \big)$, define $x_c^+$ as $c=U^+(x_c^+)$.  
\end{definition}
\begin{remark}
It is possible that $c\in U^-\big([-h_-, 0] \big)\cap U^+\big([0, h_+] \big)\neq \emptyset$. In this case, $c=U^+(x_c^+)=U^-(x_c^-)$, but $x_c^+$ is not equal to $x_c^-$.
\end{remark}

We shall use the following results about $y^-$ and $Y^-$ which were obtained in \cite{LZ} and the analogous properties of $y^+$ and $Y^+$. 
\begin{lemma}\label{L:cite-yY-}(\cite{LZ}, %Lemma 3.13(a), $y_-$$ is locally $C^\alpha$ in both $k$ and $c$ for any $\alpha\in [0, 1)$
Lemma 3.9, Lemma 3.19(1),(3), Lemma 3.20(3), Lemma 3.22, Lemma 3.24)
Assume that $U^-\in C^{l_0}, l_0\geq 6$ and $(U^-)'\neq 0$. Let $x_c^-$ be defined in Definition \ref{D:xc}. Then $Y^-$ is analytic in both $(c, k)\in D(Y^-)\setminus \big(U^-([-h_-, 0]\times \mathbb{R}) \big)$, $C^{l_0-3}$ in $k$, and locally $C^\alpha$ in $(c, k)\in D(Y^-)\cap \{(c, k) \mid c_I\geq 0\}$ for any $\alpha\in [0, 1)$. And the following hold. \begin{enumerate}
    \item For any $\beta\in (0, \frac{1}{2})$, there exists $C>0$ depending only on $\beta, |(U^-)'|_{C^2}$, and $|1/(U^-)'|_{C^0}$ such that, for any $c\in \mathbb{C}$, 
    \begin{equation}\label{E:y--estimate}
         |y^-(x_2)/\sqrt{k^2+1}-sinh\big(\frac{x_2+h_-}{\sqrt{k^2+1}}\big)|\leq C(k^2+1)^{-\frac{\beta}{2}}sinh\big(\frac{x_2+h_-}{\sqrt{k^2+1}}\big).
    \end{equation}
    \item For any $k\in \mathbb{R}$, 
    \begin{equation}\label{E:y->0}
    y^-(c, k, x_2)>0, \; \forall x_2\in (-h, 0], \; c\in \mathbb{R}\setminus  U^-\big((-h_-, 0) \big),\quad  y^-(x_c^-)> 0 \; \text{if}\; c\in U^-\big((-h_-, 0]\big) 
    \end{equation}
    \item There exists $C>0$ depending only on $U^-$ such that for any $k\in \mathbb{R}$, 
    \begin{equation}\label{E:yI->0}
    |y^-_I(c, k, 0)|\geq \frac{C}{k^2+1}|(U^-)''(x_c^-)|sinh \frac{(x_c^-+ h_-)}{\sqrt{k^2+1}}sinh\frac{|x_c^-|}{\sqrt{k^2+1}}, \quad c\in U^-\big([-h_-, 0]\big),
    \end{equation}
    %\item $Y^\pm\big(U^\pm(\pm h_\pm)\big)\in \mathbb{R}$ and $Y^\pm\big(U^\pm(\pm h_\pm), 0\big)=\frac{(U^\pm)'(0)}{U^\pm(0)-U^\pm(\pm h_\pm)}$. 
    \item There exists $C, \rho>0$ depending only on $U^-$ such that 
    \begin{equation}\label{E:Y--log}
        |Y^-(c, k)|\leq C\big(\sqrt{1+k^2}+\big|\log\min\{1, |U^-(0)-c|\}\big| \big), \quad \forall k\in \mathbb{R}, \; |c-U^-(0)|\leq \rho. 
    \end{equation}
    For any $\beta\in (0, \frac{1}{2})$, there exist $k_0>0$ and $C>0$ depending only on $\beta$, $|(U^-)'|_{C^2}$, and $|\frac{1}{(U^-)'}|_{C^0}$ such that, 
    \begin{equation}\label{E:Y-estimate}
        |Y^-(c, k)- kcoth(kh_-)|\leq C\big((k^2+1)^{\frac{1-\beta}{2}}+|\log \min \{1, |U^-(0)-c|\}| \big), \quad \forall |k|\geq k_0, c\neq U^-(0). 
    \end{equation}
    \item $Y_I^-(c, k)=0$ for $c\in \mathbb{R}\setminus U^-\big((-h_-, 0] \big)$. If $y^-(c, k, 0)\neq 0$, 
    \begin{equation}\label{E:YI-}
        Y_I^-(c, k)=\frac{\pi (U^-)''(x_c^-)y^-(c, k, x_c^-)^2}{|(U^-)'(x_c^-)||y^-(c, k, 0)|^2}, \quad c\in U^-\big((-h_-, 0) \big).
    \end{equation}
    \begin{equation}\label{E:Y-cauchy integral-}
    Y^-(c, k)=
        \begin{cases}
         \frac{1}{\pi}\int_{U^-([-h_-, 0])}\frac{Y^-_I(c', k)}{c-c'}dc'+kcothkh_-, \quad c\notin U^-\big([-h_-, 0]\big), \\
        -\mathcal{H}\big(Y^-_I(\cdot, k) \big)(c)+iY^-_I(c, k)+kcothkh_-, \quad c\in U^-\big([-h_-, 0]\big). 
        \end{cases}
    \end{equation}
    Here, $\mathcal{H}$ means Hilbert transform in $c\in \mathbb{R}$. 
\begin{align*}
    \mathcal{H}(f)(c)=\frac{1}{\pi}P.V.\int_{\mathbb{R}}\frac{f'(c')}{c-c'}dc',
\end{align*}
where $P.V.\int$ denotes the principle value of the singular integral. 
\end{enumerate}
\end{lemma}
    
We also have $y^+$ and $Y^+$ satisfy the analogous properties as following.

\begin{lemma}\label{L:yY+}
Assume that $U^+$ satisfies \eqref{E:U monotone}. Let $x_c^+$ be defined in Definition \ref{D:xc}. Then $Y^+$ is $C^{l_0-3}$ in $k$ and locally $C^\alpha$ in $(c, k)\in D(Y^+)\cap \{(c, k) \mid c_I\geq 0\}$ for any $\alpha\in [0, 1)$ and analytic in $(c, k)\in D(Y^+)\setminus \big(U^+([0, h_+])\times \mathbb{R} \big)$. And the following hold. 
\begin{enumerate}
    \item For any $\beta\in (0, \frac{1}{2})$, there exists $C>0$ depending only on $\beta, |(U^+)'|_{C^2}$, and $|1/(U^+)'|_{C^0}$ such that, for any $c\in \mathbb{C}\setminus U^+([0, h_+])$, 
    \begin{equation}\label{E:y+-estimate}
        |y^+(x_2)/\sqrt{k^2+1}-sinh\big(\frac{x_2-h_+}{\sqrt{k^2+1}}\big)|\leq C(k^2+1)^{-\frac{\beta}{2}}sinh\big(\frac{h_+-x_2}{\sqrt{k^2+1}}\big).
    \end{equation}
    \item For any $k\in \mathbb{R}$, 
    \begin{equation}\label{E:y+<0}
    y^+(c, k, x_2)<0, \; \forall x_2\in [0, h_+), \; c\in \mathbb{R}\setminus  U^+\big((0, h_+) \big), \quad y^+(x_c^+)< 0 \; \text{if} \; c\in U^+\big([0, h_+)\big).    
    \end{equation} 
    \item There exists $C>0$ depending only on $U^+$ such that for any $k\in \mathbb{R}$,
     \begin{equation}\label{E:yI+>0}
    |y^+_I(c, k, 0)|\geq \frac{C}{k^2+1}|(U^+)''(x_c^+)|sinh \frac{( h_+-x_c)}{\sqrt{k^2+1}}sinh\frac{|x_c^+|}{\sqrt{k^2+1}}, \quad c\in U^+\big([0, h_+]\big).
    \end{equation}
    \item There exists $C, \rho>0$ depending only on $U^+$ such that 
    \begin{equation}\label{E:Y-+log}
        |Y^+(c, k)|\leq C\big(\sqrt{1+k^2}+\big|\log\min\{1, |U^+(0)-c|\}\big| \big), \quad \forall k\in \mathbb{R}, \; |c-U^+(0)|\leq \rho. 
    \end{equation}
    For any $\beta\in (0, \frac{1}{2})$, there exist $k_0>0$ and $C>0$ depending only on $\beta$, $|(U^+)'|_{C^2}$, and $|\frac{1}{(U^+)'}|_{C^0}$ such that, 
    \begin{equation}\label{E:Y+estimate}
        |Y^+(c, k)+ kcoth(kh_+)|\leq C\big((k^2+1)^{\frac{1-\beta}{2}}+|\log \min \{1, |U^+(0)-c|\}| \big), \quad \forall |k|\geq k_0, c\neq U^+(0). 
    \end{equation}
    \item For $c\in \mathbb{R}\setminus U^+\big([0, h_+) \big)$, $Y_I^+(c, k)=0$. If $y^+(c, k, 0)\neq 0$, 
    \begin{equation}\label{E:YI+}
        Y_I^+(c, k)=-\frac{\pi (U^+)''(x_c^+)y^+(c, k, x_c^+)^2}{|(U^+)'(x_c^+)||y^+(c, k, 0)|^2}, \quad c\in U^+\big((0, h_+) \big).
    \end{equation}
    \begin{equation}\label{E:Y-cauchy integral+}
    Y^+(c, k)=
        \begin{cases}
         \frac{1}{\pi}\int_{U^+([0, h_+])}\frac{Y^+_I(c', k)}{c'-c}dc'-kcothkh_+, \quad c\notin U^+\big([0, h_+]\big), \\
        -\mathcal{H}\big(Y^+_I(\cdot, k) \big)(c)+iY^+_I(c, k)-kcothkh_+, \quad c\in U^+\big([0, h_+]\big).
        \end{cases}
    \end{equation}
\end{enumerate}
\end{lemma}
\begin{remark}\label{R:Y-defined-cI-sign}
1). \eqref{E:y->0} and \eqref{E:y+<0} still hold without assuming $(U^\pm)'\neq 0$.
2). For $c\in U^-\big((-h_-, 0)\big)$, if $Y^-(c, k)$ is defined as $\lim_{\beta\rightarrow 0+}Y^-(c-i\beta, k)$, then $sgn(Y^-_I)$ is different, i.e. $-Y^-_I$ satisfies \eqref{E:YI-}. $Y_I^+$ has an analogous property. In this paper, we shall only use the properties of $Y^\pm_I$ for $c_I\geq 0$. Based on \eqref{E:yI->0}, $Y^-$ is well defined for $c\in U^-\big((-h_-, 0)\big)$ under the assumption that $(U^-)''\neq 0$. According to \eqref{E:yI+>0}, $Y^+$ is well defined for $c\in U^+\big((0, h_+)\big)$ if $(U^+)''\neq 0$.
\end{remark}
\end{section}

\begin{section}{Distribution of Eigenvalues}\label{S:distribution of eigenvalues}
In this section, we shall consider the situation where the upper layer fluid is lighter than the lower layer fluid. In other words, we shall discuss the linear instability of shear flows \eqref{E:Shear flow} which satisfy \eqref{E:U monotone} for $0\leq \rho^+\leq \rho^-$. Some of the results are also true for large $\epsilon=\frac{\rho^+}{\rho^-}$. Since the linear system \eqref{E:LEuler} preserves Fourier mode $e^{ikx_1}$ for any $k\in \mathbb{R}$, we will treat the wave number $k\in \mathbb{R}$ as a parameter. According to Lemma \ref{L:E-value form}, $-ikc$ with $c\in \mathbb{C}\setminus \Big(U^-\big([-h_-, 0] \big)\cup U^+\big([0, h_+] \big)\Big)$ is an eigenvalue of \eqref{E:LEuler} with parameter $k$ if  
\begin{subequations}\label{E:F-definition}
\begin{equation}\label{E:F-definition-up}
\begin{split}
F(c, k, \epsilon):= & \epsilon\Big((U^+)'(0)\big(U^+(0)-c\big) -Y^+\big(U^+(0)-c\big)^2\Big) -(U^-)'(0)\big(U^-(0)-c\big)\\&+Y^-\big(U^-(0)-c\big)^2
-\big(g(1-\epsilon)+\frac{\sigma}{\rho^-}k^2\big)=0, 
\end{split}
\end{equation}
where $Y^\pm$ are defined in \eqref{E:Y-definition}.  According to \eqref{E:Y--log}, \eqref{E:YI-}, \eqref{E:Y-+log},  and \eqref{E:YI+}, we also define that 
\begin{equation}\label{E:F-definition-0}
    F(c, k, \epsilon):=\lim_{\beta\rightarrow 0+}F(c+i\beta, k, \epsilon)=\overline{\lim_{\beta\rightarrow 0+}F(c-i\beta, k, \epsilon)}, \quad \forall c\in U^-\big([-h_-, 0)\big)\cup U^+\big((0, h_+]\big). 
\end{equation}
\end{subequations}
Then, the zeros of $F(c, k, \epsilon)$ correspond to singular or non-singular mode $(c, k)$ which is defined in Definition \ref{D:mode}.
Based on Lemma 4.1 in \cite{LZ}, Lemma \ref{L:cite-yY-}, and Lemma \ref{L:yY+}, we give some basic properties of $F(c, k, \epsilon)$ for each $\epsilon\geq 0$. 

\begin{lemma}\label{L:F-regularity}
Assume $U^\pm$ satisfy \eqref{E:U monotone}. Then the following hold. 
\begin{enumerate}
    \item $F(c, k, \epsilon)$ is well defined for all $k\in \mathbb{R}$,  $c\in \mathbb{C}$, and $\epsilon\geq 0$. 
    \item When restricted to $c_I\geq 0$,  $F(c, k, \epsilon)$ is $C^{l_0-3}$ in $k$,  $C^{1, \alpha}$ in $c$, for any $\alpha\in [0, 1)$, and $C^\infty$ in $\epsilon\geq 0$.
\end{enumerate}

\end{lemma}

To study the eigenvalue distribution, we consider wave number $k$ as a parameter. We address the location of wave speed $c\in \mathbb{C}$ by starting from large $|k|$. And then we study what happens to the eigenvalues as $|k|$ decreases from infinity. 

\begin{lemma}\label{L:c-large k}
Assume that $U^\pm$ satisfy \eqref{E:U monotone}. Then there exist some $k_0>0$ and $C>0$ depending only on $|(U^\pm)'|_{C^2}$, and $|1/(U^\pm)'|_{C^0}$, such that for each $\epsilon\in [0, 1]$ and all $|k|>k_0$, $F(c, k, \epsilon)$ defined by \eqref{E:F-definition} has exactly two solutions $c^\pm(k)$ depending on $k$ analytically. Moreover, 
\begin{equation}\label{E:c real even-large k}
    c^\pm(k)\in \mathbb{R}, \; c^\pm(k)=c^\pm(-k),
\end{equation}
\begin{equation}\label{E:estimate for c-large k}
    \Big|c^\pm(k)-\frac{U^+(0)\epsilon+U^-(0)}{1+\epsilon}\mp\sqrt{-\frac{\epsilon\big(U^+(0)-U^-(0)\big)^2}{(1+\epsilon)^2}+\frac{\sigma |k|}{\rho^-(1+\epsilon)}}\Big|\leq C,
\end{equation} 
\begin{equation}\label{E:est-large c-pcF}
    \big|\partial_c F\big(c^\pm(k), k, \epsilon\big)\mp 2|k|^{\frac{3}{2}}\sqrt{\sigma(1+\epsilon)/\rho^-}\big|\leq C|k|. 
\end{equation}
\end{lemma}
\begin{proof}
Since $\epsilon\in [0, 1]$, \eqref{E:semicircle assumption} holds for all $k\in \mathbb{R}$. According to \eqref{E:Y-estimate}, and \eqref{E:Y+estimate}, $Y^\pm(c, k)$ are comparable to $\mp |k|$ as $|k|$ tends to $\infty$. Then, for large $|k|$, $F(c, k, \epsilon)$ behaves more or less like a quadratic function in $k$ with a negative leading coefficient. Hence, there exist $N_1, \gamma>0$ such that for all $|k|>N_1$, if $ y^\pm(c, k, x_2)$ with $(c_R-\frac{a+b}{2})^2+c_I^2\leq b+2\gamma-a$, 
solves the system \eqref{E:y+} and \eqref{E:y-}, then $F<-\frac{\sigma}{2\rho^-}k^2<0$. Theorem \ref{T:semicircle} implies that if $-ikc$ is an eigenvalue with $c$ lying outside the semicircle \eqref{E:semicircle}, then $c\in \mathbb{R}\setminus [a, b]$. Therefore, if $\big(c, k, y^\pm(c, k, x_2)\big)$ with $|k|>N_1$ 
solves the system \eqref{E:y+}, \eqref{E:y-}, and \eqref{E:F-definition-up}, then $c_R\in [a-\gamma, b+\gamma]$, where $a$ and $b$ are defined in \eqref{E:ab}. Let
\begin{equation}\label{E:S+-}
    S_+:=\{c>b+\gamma\}, \quad S_-:=\{c<a-\gamma\}.
\end{equation}
By the definition of $F(c, k, \epsilon)$, $F$ is $C^\infty$ in $(c, k)\in S_\pm\times \mathbb{R}$. For each $\epsilon\in [0, 1]$, we consider $F(c, k, \epsilon)=0$ as a quadratic equation of $c$. Its roots satisfy 
\begin{align*}
    c=f^\pm(c, k)= \frac{-B\pm\sqrt{B^2-4\big(Y^-(c, k)-\epsilon Y^+(c, k)\big)A}}{2\big(Y^-(c, k)-\epsilon Y^+(c, k)\big)},
\end{align*}
where
\begin{align*}
A:= & \epsilon\big((U^+)'(0)U^+(0)-Y^+U^+(0)^2\big)-(U^-)'(0)U^-(0)+Y^-U^-(0)^2-g(1-\epsilon)-\frac{\sigma}{\rho^-}k^2,\\
B:= & (U^-)'(0)-2U^-(0)Y^--(U^+)'(0)\epsilon+2Y^+U^+(0)\epsilon.
\end{align*}
Using formulas \eqref{E:Y-cauchy integral-} and \eqref{E:YI-} with estimate \eqref{E:y--estimate}, we obtain that there exist $C>0$ and large $N_1>0$ such that for all $c\in S_+\cup S_-$ and $|k|>N_1$,
\begin{align*}
    |Y^-(c, k)-kcothkh|=\frac{1}{\pi}\Big|\int_{U^-([-h_-, 0])} \frac{Y_I^-(c', k)}{c-c'}dc'\Big| \leq & C \int_{U^-[-h_-, 0]}\frac{y^-(c, k, x_c)^2}{|y^-(c, k, 0)|^2}dc \\
     \leq & C \int_{-h_-}^0 e^{\sqrt{k^2+1}x_c}dc\leq C(|k|+1)^{-1}.  
\end{align*}
A similar computation based on \eqref{E:Y-cauchy integral+},
\eqref{E:y+-estimate}
,   and \eqref{E:YI+} leads to 
\begin{align*}    
|Y^+(c, k)+kcothkh|\leq C(|k|+1)^{-1}. 
\end{align*}
According to these two estimates and the fact that $coth(x)=1+2(e^{2x}-1)^{-1}$, there exists $k_0>N_1$ and $C>0$ depending only on $|(U^\pm)'|_{C^2}$, and $|1/(U^\pm)'|_{C^0}$, such that for all $|k|>k_0$, 
\begin{align*}
    \Big|f^\pm(c, k)-\frac{U^+(0)\epsilon+U^-(0)}{1+\epsilon}\mp\sqrt{-\frac{\epsilon\big(U^+(0)-U^-(0)\big)^2}{(1+\epsilon)^2}+\frac{\sigma |k|}{\rho^-(1+\epsilon)}}\Big|\leq C.
\end{align*}
Then, $f^\pm(c, k):S_\pm\rightarrow S_\pm$ for $|k|>k_0$. It remains to evaluate $\partial_c f^\pm(c, k)$ and show $f^\pm(c, k)$ is a contraction acting on $S_\pm$. We first compute $\partial_c Y^\pm$. While $\partial_c Y^-$ was obtained and evaluated in \cite{LZ}, we compute $\partial_c Y^\pm$ here  for self-completeness. 

We differentiate \eqref{E:E-problem-1}  in $c\in \mathbb{R}\setminus U^-([-h_-, 0])$ and $c\in \mathbb{R}\setminus U^+([0, h_+])$ respectively, 
\begin{equation}\label{E:Rayleigh-partial-c}
    \begin{cases}
    -\partial_c (y^\pm)''(x_2)+\big(k^2+\frac{(U^\pm)''}{U^\pm-c}\big)\partial_c y^\pm(x_2)=-\frac{(U^\pm)''}{(U^\pm-c)^2}y^\pm(x_2), \qquad \pm x_2\in (0, h_\pm),\\
    \partial_c y^\pm(\pm h_\pm)=\partial_c (y^\pm)'(\pm h_\pm)=0. 
    \end{cases} 
\end{equation}
We compute $\eqref{E:E-problem-1}\partial_c y^\pm(x_2)-\eqref{E:Rayleigh-partial-c} y^\pm(x_2)$ and obtain that,
\begin{align*}
    -(y^\pm)'' \partial_c y^\pm+\partial_c (y^\pm)'' y^\pm=\frac{(U^\pm)''}{(U^\pm-c)^2}(y^\pm)^2, \qquad \pm x_2\in (0,  h_\pm).
\end{align*}
Using the boundary conditions in \eqref{E:Rayleigh-partial-c}, we integrate the above equation on $[-h_-, 0]$ and $[0, h_+]$ respectively, and obtain that 
\begin{align*}
    -(y^+)'(0)\partial_c y^+(0)+\partial_c (y^+)'(0)y^+(0)=-\int_0^{h_+} \frac{(U^+)''}{(U^+-c)^2}(y^+)^2 dx_2,\\
    -(y^-)'(0)\partial_c y^-(0)+\partial_c (y^-)'(0)y^-(0)=\int_{-h_-}^0 \frac{(U^-)''}{(U^--c)^2}(y^-)^2 dx_2.
\end{align*}
This implies that 
\begin{subequations}\label{E:partial-c-Y}
\begin{equation}\label{E:partial-c-Y+}
    \partial_c Y^+(c,  k)=-\frac{1}{y^+(0)^2}\int_0^{h_+}\frac{(U^+)''}{(U^+-c)^2}(y^+)^2 dx_2,
\end{equation} 
\begin{equation}\label{E:partial-c-Y-}
    \partial_c Y^-(c,  k)=\frac{1}{y^-(0)^2}\int_{-h_-}^0\frac{(U^-)''}{(U^--c)^2}(y^-)^2 dx_2.
\end{equation} 
\end{subequations}
Using \eqref{E:y--estimate} and \eqref{E:y+-estimate}, we obtain that there exist $C_1>0$ such that for all $|k|>k_0$ and $c\in S_+\cup S_-$, 
\begin{align*}
|\partial_c Y^\pm|\leq \frac{C_1}{|k|}.
\end{align*}
We compute $\frac{df^\pm(\cdot, k)}{dc}$. Based on \eqref{E:Y-estimate}, \eqref{E:Y+estimate}, and the estimate of $|\partial_c Y^\pm|$, we can choose $k_0$ large such that the term involving $\frac{dA}{dc}$ dominates in $\frac{df^\pm(\cdot, k)}{dc}$. Then one may check that 
\begin{align*}
    |\frac{df^\pm(\cdot, k)}{dc}|\leq \frac{C}{\sqrt{|k|}}<1, \quad  \forall |k|\geq k_0, 
\end{align*}
where $C$ is independent of $k$, $c$, and $\epsilon$. Therefore, $f^\pm(c, k)$ are contractions acting on $S^\pm$ respectively. Then their fixed points $c^\pm(k)$ are the only solutions to system \eqref{E:E-problem-1}, \eqref{E:E-problem-2}, and \eqref{E:E-problem-n} on $S_\pm$. Moreover, they are analytic in $k$. Finally, one may compute
\begin{equation}\label{E:partial-c-F}
\begin{split}
    \partial_c F = & \epsilon\Big(-(U^+)'(0)+2\big(U^+(0)-c\big)Y^+-\partial_c Y^+\big(U^+(0)-c)^2\Big)\\&+\partial_cY^-\big(U^-(0)-c)^2-2\big(U^-(0)-c\big)Y^-+(U^-)'(0). 
\end{split}
\end{equation}
Then by using the above estimates of $|Y^\pm\pm kcothkh|$, $\partial_c Y^\pm$, and $c^\pm(k)$, the last desired estimate of $\partial_c F\big(c^\pm(k), k, \epsilon\big)$ can be obtained. This completes the proof of the lemma. 
\end{proof}
\begin{remark}\label{R:c-pm-extend}
1). When $\epsilon=0$, $c^\pm(k)$ obtained in Lemma \ref{L:c-large k} coincide with those obtained in \cite{LZ}. Later in Section \ref{S:Ocean-Air Model}, this fact will be used. 

2). Recall that $a$ and $b$ are defined in \eqref{E:ab}. Suppose that $k_+\geq 0$ is the biggest $|k|$ such that $F(b, k_+, \epsilon)=0$, and $k_-\geq 0$ is the biggest $|k|$ such that $F(a, k_-, \epsilon)=0$. By Lemma 4.3 in \cite{LZ} and Theorem \ref{T:semicircle}, for each $\epsilon\in [0, 1]$, $c^+(k)$ (resp. $c^-(k)$) obtained in Lemma \ref{L:c-large k} can be extended to be an even and analytic function of $k$ for all $|k|\in [k_+, \infty)$ (resp. $|k|\in [k_-, \infty)$) such that
\begin{equation}\label{E:cpm-k-ep=0}
\begin{split}
    F\big(c^+(k), k, \epsilon\big)=0, \; \forall |k|\in [k_+, \infty), \quad  c^+(k)>b, \; \forall |k|\in (k_+, \infty), \\
    F\big(c^-(k), k, \epsilon\big)=0,\; \forall |k|\in [k_-, \infty), \quad  c^-(k)<a,\; \forall |k|\in (k_-, \infty). 
\end{split}
\end{equation}
Moreover,  $c^+(k)$ (resp. $c^-(k)$) can be extended to be an analytic function for all $|k|\in [0, \infty)$ if for any $k\in \mathbb{R}$, $(b, k)$ (resp. $(a, k)$) is not a neutral mode. In addition, for $c_R>b$ and $c_I\in \mathbb{R}$, $c^+(k)$ is the unique root of $F(\cdot, k, \epsilon)=0$, (resp. for $c_R<a$ and $c_I\in \mathbb{R}$, $c^-(k)$ is the unique root of $F(\cdot, k, \epsilon)=0$.)

Since $F(c, k, \epsilon)\in \mathbb{R}$ for all $c\in \mathbb{R}\setminus [a, b]$, $\partial_c F(c^\pm(k), k, \epsilon)$ does not change sign along these simple roots. Hence, the signs of $\partial_c F$ in \eqref{E:est-large c-pcF}, Lemma \ref{L:cite-yY-}, Lemma \ref{L:yY+}, and Lemma \ref{L:F-regularity} imply that 
\begin{equation}\label{E:ext-c-pm-pcF}
    \partial_c F\big(c^+(k), k, \epsilon\big)>0, \; \forall |k|\in [k_+, \infty), \quad \partial_c F\big(c^-(k), k, \epsilon\big)<0, \; \forall |k|\in [k_-, \infty). 
\end{equation}
Particularly,  since $Y^-(c, k)$ has logarithmic singularity $log|U^-(0)-c|$ near $c=U^-(0)$,  
\begin{equation}\label{E:pcF-esp-0-c+}
    F\big(U^-(0), k, 0\big)=-g-\frac{\sigma}{\rho^-}k^2<0, \quad \partial_c F\big(c^+(k), k, 0\big)>0, \; \forall k\in \mathbb{R}. 
\end{equation}
\end{remark}

According to Remark \ref{R:c-pm-extend}(2), $c^\pm(k)$ can keep continuity in $k$ as simple roots of the analytic function $F(\cdot, k, \epsilon)$ as $|k|$ gets smaller. We shall track $c^\pm(k)$ as $|k|$ decreases to study possible bifurcation and seek instability. 
Before studying the eigenvalue distribution as $|k|$ decreases, we prove some properties of $F(c, k, \epsilon)$ in the following lemma. Even though some properties of $Y^-(c, k)$ were obtained in \cite{LZ}, we present the proof for self-completeness.

\begin{lemma}\label{L:p-KK-F}
Assume that $U^\pm$ satisfy \eqref{E:U monotone}. Suppose that $K=k^2$ and $\epsilon\geq 0$. Let $a$ and $b$ be defined in \eqref{E:ab}. Then for $c\in \mathbb{R}\setminus \Big( U^-\big((-h_-, 0]\big)\cup U^+\big([0, h_+)\big)\Big)$, the following hold.
\begin{enumerate}
    \item \begin{equation}\label{E:p-KK-F}
  \partial_{KK}F(c, k, \epsilon)<0. 
\end{equation}
\item If both $(U^+)''>0$ and  $(U^-)''>0$, then $\partial_{Kc}F(c, k, \epsilon)<0$ for all $c\leq a$ and $k\in \mathbb{R}$.
\item If both $(U^+)''<0$ and $(U^-)''<0$, then $\partial_{Kc}F(c, k, \epsilon)>0$ for all $c\geq b$ and $k\in \mathbb{R}$. 
\end{enumerate}
\end{lemma}
\begin{proof}
Let $K:=k^2$. We notice that for $c\in \mathbb{C}\setminus U^-([-h_-, 0])$, $\partial_K y^-$ satisfies the following equations.
\begin{subequations}\label{E:R-p-k-y}
\begin{equation}\label{E:R-partial-k-y-}
\begin{cases}
-\partial_K (y^-)''(x_2)+\big(k^2+\frac{(U^-)''}{U^--c}\big)\partial_K y^-(x_2)=-y^-(x_2), \qquad x_2\in (-h_-, 0),\\
    \partial_K y^-(- h_-)=\partial_K (y^-)'(- h_-)=0. 
\end{cases}    
\end{equation}
For $c\in \mathbb{C}\setminus U^+([0, h_+])$, $\partial_K y^+$ satisfies the following system.
\begin{equation}\label{E:R-p-k-y+}
    \begin{cases}
-\partial_K (y^+)''(x_2)+\big(k^2+\frac{(U^+)''}{U^+-c}\big)\partial_K y^+(x_2)=-y^+(x_2), \qquad x_2\in (0, h_+),\\
    \partial_K y^+( h_+)=\partial_K (y^+)'( h_+)=0. 
\end{cases} 
\end{equation}
\end{subequations}
We consider $c\in \mathbb{R}\setminus \big(U^-([-h_-, 0])\cup U^+([0, h_+])\big)$ and do $\eqref{E:E-problem-1}\partial_K y^\pm(x_2)-\eqref{E:R-p-k-y}y^\pm(x_2)$ respectively and obtain that
\begin{subequations}\label{E:partial-k}
\begin{equation}\label{E:partial-k-}
    -(y^-)''\partial_K y^-+\partial_K(y^-)'' y^-=(y^-)^2, \quad  x_2\in (-h_-, 0).
\end{equation} 
\begin{equation}\label{E:partial-k+}
    -(y^+)''\partial_K y^++\partial_K(y^+)'' y^-=(y^+)^2, \quad  x_2\in (0, h_+).
\end{equation} 
\end{subequations}
Using the boundary conditions in \eqref{E:R-p-k-y}, we integrate \eqref{E:partial-k-} on $[-h_-, 0]$ and integrate \eqref{E:partial-k+} on $[0, h_+]$, and then divide it by $y^\pm(0)^2$ respectively.
\begin{subequations}\label{E:partial-k-Y}
\begin{equation}\label{E:partial-k-Y-}
  \partial_K Y^-(c, k)=\frac{1}{y^-(0)^2}\int_{-h_-}^0 y^-(x_2)^2dx_2>0,  
\end{equation}
\begin{equation}\label{E:partial-k-Y+}
  \partial_K Y^+(c, k)=-
  \frac{1}{y^+(0)^2}\int_0^{h_+} y^+(x_2)^2dx_2<0. 
\end{equation}
\end{subequations}
Similarly, we integrate $\eqref{E:partial-k-}$ on $[-h_-, x_2]$ and \eqref{E:partial-k+} on $[x_2, h_+]$, and then divide it by $y^\pm(x_2)$ respectively. We obtain that
\begin{subequations}\label{E:d-p-x2-0}%not cite
\begin{equation}\label{E:d-p-x2-0-}
    \frac{d}{dx_2}\Big(\frac{\partial_K y^-(x_2)}{y^-(x_2)}\Big)=\frac{1}{y^-(x_2)^2}\int_{-h_-}^{x_2}y^-(x_2')^2dx_2' >0, \qquad x_2\in [-h_-, 0].
\end{equation}
\begin{equation}\label{E:d-p-x2-0+}
    \frac{d}{dx_2}\Big(\frac{\partial_K y^+(x_2)}{y^+(x_2)}\Big)=-\frac{1}{y^+(x_2)^2}\int_{x_2}^{h_+}y^+(x_2')^2dx_2' <0, \qquad x_2\in [0, h_+].
\end{equation}
\end{subequations}
These imply that 
\begin{subequations}\label{E:p-KK-Y}%not cite
\begin{equation}\label{E:p-KK-Y-}
    \partial_{KK}Y^-(c, k)=-2\int_{-h_-}^0 \frac{y^-(x_2)^2}{y^-(0)^2}\Big(\frac{\partial_Ky^-(0)}{y^-(0)}-\frac{\partial_Ky^-(x_2)}{y^-(x_2)} \Big)dx_2<0. 
\end{equation}
\begin{equation}\label{E:p-KK-Y+}
    \partial_{KK}Y^+(c, k)=2\int_0^{h_+} \frac{y^+(x_2)^2}{y^+(0)^2}\Big(\frac{\partial_Ky^+(0)}{y^+(0)}-\frac{\partial_Ky^+(x_2)}{y^+(x_2)} \Big)dx_2>0. 
\end{equation}
\end{subequations}
Therefore, for $c\in \mathbb{R}\setminus \Big( U^-\big((-h_-, 0])\cup U^+([0, h_+)\big)\Big)$, 
\begin{align*}
  \partial_{KK}F(c, k, \epsilon)=-\epsilon\partial_{KK} Y^+(c, k)\big(U^+(0)-c\big)^2+\partial_{KK} Y^-(c, k)\big(U^-(0)-c\big)^2<0. 
\end{align*}

Suppose that $c\in \mathbb{R}\setminus \Big( U^-\big((-h_-, 0])\cup U^+([0, h_+)\big)\Big)$.  We differentiate \eqref{E:partial-c-Y+} in $K$ to obtain that
\begin{align*}
    \partial_{Kc} Y^+(c, k)=\frac{2}{y^+(0)^2}\int_0^{h_+}\frac{(U^+)''}{(U^+-c)^2}(y^+)^2\big(\frac{\partial_K y^+(0)}{y^+(0)}-\frac{\partial_K y^+(x_2)}{y^+(x_2}\big)dx_2.
\end{align*}
Using \eqref{E:d-p-x2-0+}, we conclude that 
\begin{subequations}\label{E:sgn-pKc-Y}
\begin{equation}\label{E:sgn-pKcY+}
    sgn(\partial_{Kc} Y^+)=sgn\big((U^+)''\big).
\end{equation}
A similar computation leads to
\begin{equation}\label{E:sgn-pKcY-}
    sgn(\partial_{Kc} Y^-)=-sgn\big((U^-)''\big).
\end{equation}
\end{subequations}
We consider the following function. \begin{equation}\label{E:partial-Kc-F}
  \partial_{Kc} F=\partial_{Kc}Y^-\big(U^-(0)-c\big)^2-2\big(U^-(0)-c\big)\partial_KY^-+\epsilon\Big(-\partial_{Kc}Y^+\big(U^+(0)-c \big)^2+2\big(U^+(0)-c \big)\partial_KY^+ \Big). 
\end{equation}
Suppose that $c\leq a$,  $(U^+)''>0$, and $(U^-)''>0$. Then, based on \eqref{E:partial-Kc-F}, we use \eqref{E:partial-k-Y} and \eqref{E:sgn-pKc-Y} to obtain that $\partial_{Kc}F(c, k, \epsilon)<0$. Lemma \ref{L:cite-yY-}, Lemma \ref{L:yY+}, and Lemma \ref{L:F-regularity} imply that $\partial_{Kc}F$ is well defined at $c=a$ and then $\partial_{Kc}F(a, k, \epsilon)<0$.  A similar proof of statement (3) for the case of $c\geq b$ and both $(U^+)''<0$ and $(U^-)''<0$ is similar.  
\end{proof}

As $|k|$ gets smaller from infinity, there might exist singular modes.  Lemma \ref{L:p-KK-F} implies that we can use the monotonicity of $\partial_K F(c, k, \epsilon)$ and $\partial_c F(c, k, \epsilon)$  in $K$ and reduce some computations to the case of $k=0$.  Hence, it is worth paying closer attention to the special case of $k=0$. When $k=0$, $y^\pm(x_2)$ which are solutions to \eqref{E:E-problem-1} and \eqref{E:fundamental solution} have explicit representations. A direction computation shows that if $c=U^+(h_+)$, 
\begin{equation}\label{E:y+Uh+}
 y^+(x_2)=\frac{U^+(x_2)-U^+(h_+)}{(U^+)'(h_+)}.   
\end{equation}
If $c\in \mathbb{R}\setminus U^+\big([0, h_+]\big)$,
\begin{equation}\label{E:y+out}
    y^+(x_2)=\big(U^+(h_+)-c\big)\big(U^+(x_2)-c\big)\int_{x_2}^{h_+}\frac{-1}{\big(U^+(x_2')-c \big)^2}dx_2'.
\end{equation}
If $c=U^-(-h_-)$,
\begin{equation}\label{E:U-h-y-}
y^-(x_2)=\frac{U^-(x_2)-U^-(-h_-)}{(U^-)'(-h_-)}.
\end{equation}
If $c\in \mathbb{R}\setminus U^-([-h_-, 0])$, 
\begin{equation}\label{E: y-c out}
    y^-(x_2)=\big(U^-(x_2)-c \big)\big(U^-(-h_-)-c \big)\int_{-h_-}^{x_2}\frac{1}{(U^--c)^2}dx_2. 
\end{equation}

%By Lemma 4.3 in \cite{LZ}, when $\epsilon=0$, $c^+(k)$ obtained in Lemma \ref{L:c-large k} can be extended to all $k\in \mathbb{R}$ and $c^+(k)>U^-(0)$. We will use the following lemma from \cite{LZ}.
%\begin{lemma}{Lemma 4.7 \cite{LZ}}\label{L:cite-4.7}Assume $U^-\in C^3$ and $\epsilon=0$. then $\partial_c F(c, c^+(k), 0)>0$ and $(c^+)'(k)>0$ for all $k\in \mathbb{R}$. \end{lemma}

In the next lemma, we address the possible locations of neutral limiting modes for given $\epsilon\in (0, 1)$ and small $\epsilon$. 
\begin{lemma}\label{L:nlm-location-fixed-ep}
Assume that $U^\pm$ satisfy \eqref{E:U monotone}, and  $(U^\pm)''\neq 0$. Let 
\begin{equation}\label{E:EI-definition}
    \mathcal{E}:=\{U^+(h_+), U^+(0), U^-(-h_-), U^-(0)\}, \quad \mathcal{I}:=U^+\big((0, h_+)\big)\cap U^-\big((-h_-, 0)\big),
\end{equation}
Then the following holds. 
\begin{enumerate}
\item Suppose that $\epsilon\in (0, 1)$ and there is a sequence of unstable modes $(c_n, k_n)$ converges to $(c_\infty, k_\infty)\in \mathbb{R}^2$ as $n\rightarrow\infty$, and
\begin{equation}\label{E:ep-k-infty>0}
    g(1-\epsilon)+\frac{\sigma}{\rho^-}k_\infty^2>0.
\end{equation}
Then the following holds.
\begin{enumerate}
    \item If $\mathcal{I}=\emptyset$, then $c_\infty\in [a, b]\setminus \Big(U^-\big((-h_-, 0)\big)\cup U^+\big((0, h_+) \big)\Big)$. Under additional assumption that $U^+(0)=U^-(0)$,  $c_\infty\in \mathcal{E}\setminus \{U^\pm(0)\}$.  
    \item Assume $\mathcal{I}\neq \emptyset$. Then $c_\infty\in \mathcal{E}\cup \mathcal{I}$ if  $U^+(0)\neq U^-(0)$. And if  $U^+(0)= U^-(0)$, $c_\infty\in \mathcal{E}\cup \mathcal{I}\setminus \{U^\pm(0)\}$. In addition, if $(U^+)'' (U^-)''>0$, then $c_\infty\in \{a, b\}$, where $a, b$ are defined in \eqref{E:ab}.
\end{enumerate}
\item Suppose that $\mathcal{I}= \emptyset$. There exists $\epsilon_0>0$ such that for any $\epsilon\in (0, \epsilon_0)$, if an unstable sequence $(c_n, k_n)\rightarrow (c_\infty, k_\infty)\in \mathbb{R}^2$ as $n\rightarrow \infty$, \eqref{E:ep-k-infty>0} holds,  and we further assume one of the following holds, 
\begin{enumerate}
    \item $(U^-)', (U^-)''>0$, and $\max_{[0, h_+]} U^+(x_2)<U^-(-h_-)$; or
    \item $(U^-)', (U^-)''>0$, and $\min_{[0, h_+]}U^+(x_2)>U^-(0)$,
\end{enumerate}
then $c_\infty\in \mathcal{E}$. 
\end{enumerate}
\end{lemma}

\begin{proof}
Suppose that $\epsilon\in (0, 1)$ is given. Lemma \ref{L:F-regularity} implies that
\begin{align*}
    F(c_\infty, k_\infty, \epsilon)=\lim_{n\rightarrow \infty}F(c_n, k_n, \epsilon)=0.
\end{align*}
To address the possible locations of the neutral limiting mode for each $\epsilon>0$, we need to consider the imaginary part of $F(c_\infty, k_\infty, \epsilon)$ . 
\begin{equation}\label{E:FI}
    F_I(c_\infty, k_\infty, \epsilon)=-\epsilon Y^+_I(c_\infty, k_\infty)\big(U^+(0)-c_\infty\big)^2+Y^-_I(c_\infty, k_\infty)\big(U^-(0)-c_\infty\big)^2. 
\end{equation}
Since \eqref{E:ep-k-infty>0} holds, Theorem \ref{T:semicircle} implies that $c_\infty\in [a, b]$, where $a$ and $b$ are defined in \eqref{E:ab}. We shall discuss the cases as following. 

\noindent * \textit{Case (1a). $\mathcal{I}=\emptyset$.}  Since $(U^\pm)''\neq 0$, according to Lemma \ref{L:cite-yY-} and Lemma \ref{L:yY+}, $Y^-_I\neq 0$ if $c\in U^-\big((-h_-, 0] \big)$. And $Y_I^-\equiv 0$ if $c\in \mathbb{R}\setminus U^-\big((-h_-, 0]\big)$, $Y^+_I\neq 0$ if $c\in U^+\big([0, h_+) \big)$, and $Y^+_I\equiv 0$ if $c\in \mathbb{R}\setminus U^+\big([0, h_+)\big)$. Hence,  $F_I(c_\infty, k_\infty, \epsilon)\neq 0$ for $c_\infty\in U^-\big((-h_-, 0)\big)\cup U^+\big((0, h_+)\big)$. Particularly, if we further assume that $U^+(0)=U^-(0)$, then \eqref{E:Y-estimate} and \eqref{E:Y+estimate} imply that  $F(c_\infty, k, \epsilon)=-g(1-\epsilon)-\frac{\sigma}{\rho^-}k^2< 0$ for all $k\in \mathbb{R}$ and $\epsilon\in (0, 1)$. Hence, in this case, $c\notin U^\pm(0)$.

\noindent * \textit{Case (1b). $\mathcal{I}\neq \emptyset$.} It is clear that $c_\infty \in  \mathcal{E}\cup \mathcal{I}$ by considering $F_I(c_\infty, k_\infty, \epsilon)=0$. We know that $c_\infty\neq U^\pm(0)$ if $U^+(0)=U^-(0)$. Let $x_{c_\infty}^-:=(U^-)^{-1}(c_\infty)$ and $x_{c_\infty}^+:=(U^+)^{-1}(c_\infty)$. Based on  \eqref{E:YI-} and \eqref{E:YI+}, if $c_\infty\in \mathcal{I}$,  \eqref{E:FI} can be rewritten as 
\begin{align*}
    F_I(c_\infty, k_\infty, \epsilon)=\epsilon\frac{\pi (U^+)''(x_{c_\infty}^+)y^+(x_{c_\infty})^2}{|(U^+)'(x_{c_\infty}^+)||y^+(0)|^2}\big(U^+(0)-c_\infty\big)^2+\frac{\pi (U^-)''(x_{c_\infty}^-)y^-(x_{c_\infty})^2}{|(U^-)'(x_{c_\infty}^-)||y^-(0)|^2}\big(U^-(0)-c_\infty\big)^2.
\end{align*}
Hence, if $sgn\big((U^-)''\big)=sgn\big((U^+)''\big)$, $F_I(c_\infty, k_\infty, \epsilon)\neq 0$. Statement (1) is proved. 

2. Now we consider $\mathcal{I}=\emptyset$ and will prove $c_\infty\in \mathcal{E}$ under the assumption of (2a) or (2b). 
To prove (2), we will show that in the following cases, there exists $\epsilon_0>0$ such that for any $\epsilon\in (0, \epsilon_0)$, if $F(c, k, \epsilon)=0$ then $\partial_c F(c, k, \epsilon)\neq 0$ under some conditions. Let $K:=k^2$ and $K_\infty=k_\infty^2$. \\
\textit{Case (a).} $(U^-)'>0,  (U^-)''>0$, and $\max_{[0, h_+]} U^+(x_2)<U^-(-h_-)$. Assume that \begin{equation}\label{E:case-a}
    c\in \big(\max_{[0, h_+]} U^+(x_2), U^-(-h_-) \big).
\end{equation}
Since $(U^-)''>0$, \eqref{E:sgn-pKc-Y} implies that $\partial_{Kc} Y^-<0$. By \eqref{E:partial-k-Y-}, $\partial_K Y^->0$. Hence, 
\begin{align*}
    \partial_{Kc} F(c, k, 0)=\big((U^-(0)-c)^2\big)\partial_{kc} Y^-(c, k)-2\big(U^-(0)-c\big)\partial_K Y^-(c, k)<0. 
\end{align*}
This implies that $\partial_c F(c, k, 0)<\partial_c F(c, 0, 0)$ for all $k>0$. If $c=U^-(-h_-)$, $y^-(x_2)$ is in the form of \eqref{E:U-h-y-}. Using \eqref{E:partial-c-Y-}, we compute that 
\begin{equation}\label{E:-h-pcY-}
    \big(U^-(0)-U^-(-h_-)\big)^2\partial_c Y^-\big(U^-(-h_-), 0\big)=\int_{-h_-}^0 (U^-)'' dx_2=(U^-)'(0)-(U^-)'(-h_-),
\end{equation} 
\begin{equation}\label{E:Y-h-k0}
    Y^-\big(U^-(-h_-), 0\big)=\frac{(U^-)'(0)}{U^-(0)-U^-(-h_-)}. 
\end{equation}
Hence, we obtain that 
\begin{equation}\label{E:p-c-F-U-h-}
    \partial_c F\big(U^-(-h_-), 0, 0 \big)=-(U^-)'(-h_-)<0.  
\end{equation}
Similarly, if $c\in \mathbb{R}\setminus U^-([-h_-, 0])$,  $y^-(x_2)$ is in the form of  \eqref{E: y-c out}. Then we compute 
\begin{equation}\label{E:p-c-1-0outU-h-}
\begin{split}
    \big(U^-(0)-c\big)^2 \partial_c Y^-(c, 0)=&\frac{1}{\big( \int_{-h_-}^0\frac{1}{(U^--c)^2}dx_2\big)^2}\int_{-h_-}^0 (U^-)'' \big(\int_{-h_-}^{x_2}\frac{1}{(U^--c)^2}dx_2'\big)^2dx_2\\
    \leq & \int_{-h_-}^0 (U^-)''dx_2 =(U^-)'(0)-(U^-)'(-h_-),
    \end{split}
\end{equation}
\begin{equation}\label{E:Y-k-0outU-}
    Y^-(c, 0)= \frac{(U^-)'(0)}{U^-(0)-c}+\frac{1}{\big(U^-(0)-c\big)^2\int_{-h_-}^0 \frac{1}{(U^--c)^2}dx_2}. 
\end{equation}
Then, we obtain that 
\begin{equation}\label{E:pc-F-0outU-}
    \partial_c F(c, 0, 0)\leq -(U^-)'(-h_-)-\frac{2}{\big(U^-(0)-c\big)\int_{-h_-}^0 \frac{1}{(U^--c)^2}dx_2}<0. 
\end{equation}
Hence, $\partial_c F(c, 0, 0)<0$ if $c\leq U^-(h_-)$. 
By \eqref{E:p-KK-F}, there exists $\epsilon_0>0$  such that $\partial_c F(c, k, \epsilon)=\partial_c F(c, k, 0)+O(\epsilon)< \partial_c F(c, 0, 0)+O(\epsilon)<0$ for any $\epsilon\in (0, \epsilon_0)$ and $c$ satisfies \eqref{E:case-a}. In fact, this result does not require that $F(c, k, \epsilon)=0$.  \\
\textit{Case (b).} $(U^-)'>0$, $(U^-)''>0$, and $U^-(0)<\min_{[0, h_+]}U^+(x_2)$. Assume that 
\begin{equation}\label{E:case-b}
    c\in \big(U^-(0), \min_{[0, h_+]}U^+(x_2)\big). 
\end{equation}
Since $F(c, k, \epsilon)=0$, by the smoothness of $F$ in $\epsilon$, we obtain that 
\begin{align*}
    F(c, k, \epsilon)=F(c, k, 0)+O(\epsilon). 
\end{align*}
Hence, a direction computation leads to 
\begin{align*}
    Y^-(c, k)=\frac{g+\frac{\sigma}{\rho^-}k^2+O(\epsilon)}{\big(U^-(0)-c\big)^2}+\frac{(U^-)'(0)}{U^-(0)-c}.
\end{align*}
Then we compute that 
\begin{align*}
    \partial_c F(c, k, \epsilon)=\big(U^-(0)-c\big)^2\partial_c Y^-(c, k)+(U^-)'(0)-2\frac{g+\frac{\sigma}{\rho^-}k^2+O(\epsilon)}{U^-(0)-c}+O(\epsilon).
\end{align*}
Using the assumption that $(U^-)'>0$, the fact that $\partial_c Y^->0$ if $(U^-)''>0$ (see \eqref{E:partial-c-Y-}), and the smallness of $\epsilon$, we have $\partial_c F(c, k, \epsilon)>0$ in this case. 

We let \begin{align*}
    \mathcal{N}:=[a, b]\setminus \Big(U^-\big([-h_-, 0] \big)\cup U^+\big([0, h_+] \big)\Big). 
\end{align*} 
By the proof in (1), to prove (2), it remains to consider $c_\infty\in \mathcal{N}$. We suppose that for some $\epsilon\in (0, \epsilon_0)$, there is a unstable sequence $(c_n, k_n)\rightarrow (c_\infty, k_\infty)\in \mathbb{R}^2$.  
$\epsilon_0$ is chosen by the previous discussion. Lemma \ref{L:c-large k} implies that $k_\infty$ can not be sufficiently large. Hence, $\epsilon_0$ can be chosen independent of $k_\infty$ by the smoothness of $F$ in $(c, k, \epsilon)\in \mathcal{N}\times \mathbb{R}\times (0, 1)$. Since $F(c, k, \epsilon)$ is analytic near $c_\infty\in \mathcal{N}$, we use the Cauchy-Riemann equation to compute the $2\times 2$ Jacobian matrix of $D_c F$. 
\begin{align*}
    D_c F\big(c_\infty, k_\infty, \epsilon \big)=\begin{pmatrix}
    \partial_{c_R} F_R & \partial_{c_I} F_R\\ \partial_{c_R} F_I & \partial_{c_I} F_I
    \end{pmatrix}\Big| _{(c_\infty, k_\infty, \epsilon)}= \partial_c F(c_\infty, k_\infty, \epsilon) I_{2\times 2}. 
\end{align*}
For both case (a) and (b), $\partial_c F(c_\infty, k_\infty, \epsilon)\neq 0$. Hence, by the Implicit Function Theorem, there exists a smooth complex-valued function $\mathcal{C}(k)$ such that all the roots of $F(\cdot, \cdot, \epsilon)$ near $(c_\infty, k_\infty)$ are in the form of $\big(\mathcal{C}(k), k\big)$ and $\mathcal{C}(k_\infty)=c_\infty$. We will show that $\mathcal{C}(k)\in \mathbb{R}$ to complete the proof. Since $F_R$ is smooth for $c\in \mathcal{N}$ and $\partial_{c_R} F_R(c_\infty, k_\infty, \epsilon)=\partial_c F(c_\infty, k_\infty, \epsilon)\neq 0$, we apply the Implicit Function Theorem on $F_R$. Then there exists a smooth real-valued function $\mathcal{C}_1(k)$ for $k$ near $k_\infty$ such that $\mathcal{C}_1(k_\infty)=c_\infty$ and $F_R\big(\mathcal{C}_1(k), k, \epsilon\big)=0$. Since $c_\infty\in \mathcal{N}$, $F_I(c, k, \epsilon)=0$ near $c_\infty$. Hence, by the uniqueness of solutions obtained by the Implicit Function Theorem, $\mathcal{C}(k)=\mathcal{C}_1(k)\in \mathbb{R}$ near $c_\infty$. Hence, there is no unstable mode near $(c_\infty, k_\infty)$. The lemma is proved. 
\end{proof}

According to Lemma \ref{L:nlm-location-fixed-ep}, for each $\epsilon\in (0, 1)$, neutral limiting modes might happen at the endpoints of the range of $U^\pm$ or the intersection of the range of $U^\pm$. In either case, $(c, k)$ is a singular mode. By Remark \ref{R:c-pm-extend}, we notice that the number of eigenvalues $\lambda=-ikc$ may be changed if $c^\pm(k)$ obtained in Lemma \ref{L:c-large k} touches the range of $U^\pm$ as $|k|$ decreases. On the other hand, if we can rule out the existence of any singular modes, then $c^\pm(k)$ can be extended to all $k\in \mathbb{R}$. We would like to study how the eigenvalue distributes when $|k|$ gets smaller. The rest of this section includes two main directions. One is to prove the stability of shear flows by ruling out all singular modes under some certain conditions. The other one is to seek instability occurring near singular modes for general $\epsilon\geq 0$. In the following lemma, we provide some sufficient conditions for the non-existence of singular modes.

\begin{lemma}\label{L:no singular mode}
Assume that $U^\pm$ satisfies \eqref{E:U monotone}, $(U^\pm)''\neq 0$, and $\epsilon\in (0, 1)$. Let $a, b$ be defined in \eqref{E:ab},  $\mathcal{E}, \mathcal{I}$ be defined in \eqref{E:EI-definition}, and
\begin{equation}\label{D:m(c)}
m(c):= \epsilon\int_0^{h_+}\big(U^+(x_2)-c\big)^2dx_2+\int_{-h_-}^0\big(U^-(x_2)-c\big)^2dx_2.
\end{equation}
If 
\begin{equation}\label{E:A-F<0}
    \min_{c\in\{U^-(0), U^-(-h_-)\}} \int_{-h_-}^0 \frac{1}{(U^--c)^2}dx_2>\frac{1}{g}, \quad \min_{c\in\{U^+(0), U^+(h_+)\}}\int_0^{h_+}\frac{1}{(U^+-c)^2}dx_2>\frac{\epsilon}{g},
\end{equation}
and one of the following holds,
\begin{enumerate}
    \item $\mathcal{I}=\emptyset$ and $\max_{c\in \mathcal{E}} m(c)<\frac{\sigma}{\rho^-}$;
    \item $\mathcal{I}\neq \emptyset$, $\max_{c\in \{a, b\}}m(c)<\frac{\sigma}{\rho^-}$, and $(U^+)''(U^-)''>0$,
\end{enumerate}
then for each $k>0$, $c^\pm(k)$ obtained in Lemma \ref{L:c-large k} can be extended to be even and analytic functions for all $k\in \mathbb{R}$ and the following hold 
\begin{equation}\label{E:two branches-no singular mode}
    F(c^\pm(k), k, \epsilon)=0, \quad c^+(k)>b, \quad c^-(k)<a, \quad k\partial_c F(k, c^\pm(k))>0. 
\end{equation}
\end{lemma}
\begin{proof}
According to \eqref{E:Y-estimate}-\eqref{E:Y-cauchy integral-}, $\big(U^-(0)-c\big)^2\partial_k^j Y^-(c, k), \;j\in \mathbb{N}$ is well defined and $C^1$ near $c=U^-(0)$, $U^-(-h_-)$. Similarly,  \eqref{E:Y+estimate}-\eqref{E:Y-cauchy integral+} imply that $\big(U^+(0)-c\big)^2\partial_k^j Y^+(c, k), \; j\in \mathbb{N}$ is well defined and $C^1$ near $c=U^+(0)$ and $U^+(h_+)$.  In the following, we prove the non-existence of singular neutral mode under the assumption in the statement. \\
\noindent \textit{* Case 1.} $\mathcal{I}= \emptyset$ and $\max_{c\in \mathcal{E}}m(c)<\frac{\sigma}{\rho^-}$. 
We will show that there is no singular mode at $c\in \mathcal{N}:=[a, b]\setminus \Big(U^-\big((-h_-, 0)\big)\cup U^+\big((0, h_+)\big) \big)$. \\
\textit{Step 1.} We will show that $\partial_K F(c, 0, \epsilon)<0$ for $c\in \mathcal{E}$ if $\max_{c\in \mathcal{E}} m(c)<\frac{\sigma}{\rho^-}$.
A direct computation leads to 
\begin{equation}\label{E:pKF}
    \partial_K F(c, k, \epsilon)= \big(U^-(0)-c\big)^2\partial_K Y^-(c, k)-\epsilon \big(U^+(0)-c\big)^2\partial_K Y^+(c, k)-\frac{\sigma}{\rho^-}.
\end{equation}
Suppose that $K=0$. If $c=U^+(h_+)$, $y^+(x_2)$ is in the form of \eqref{E:y+Uh+}. We use \eqref{E:partial-k-Y+} to compute
\begin{equation}\label{E:pkY+-0-c-h+}
    -\epsilon \partial_K Y^+\big(U^+(h_+), 0\big) \big(U^+(0)-c\big)^2= \epsilon\int_0^{h_+}\big(U^+(x_2)-U^+(h_+)\big)^2dx_2.
\end{equation}
For $c\in \mathbb{R}\setminus U^+\big([0, h_+]\big)$, $y^+(x_2)$ has the representation of \eqref{E:y+out}. We obtain that
\begin{equation}\label{E:pkY+-0-c-out}
\begin{split}
    -\epsilon \partial_K Y^+(c, 0) \big(U^+(0)-c\big)^2=&\epsilon \frac{1}{\big(\int_0^{h_+} \frac{-1}{(U^+-c)^2}dx_2\big)^2}\int_0^{h_+}\big(U^+(x_2)-c\big)^2\Big(\int_{x_2}^{h_+}\frac{-1}{(U^+-c)^2}dx_2'\Big)^2dx_2,\\
    \leq & \epsilon\int_0^{h_+}\big(U^+(x_2)-c\big)^2dx_2.
\end{split}
\end{equation}
Similarly, if $c=U^-(-h_-)$, $y^-(x_2)$ is in the form of \eqref{E:U-h-y-}. Using \eqref{E:partial-k-Y-}, we compute
\begin{equation}\label{E:pkY-0-c-h_-}
    \partial_K Y^-\big(U^-(-h_-), 0\big)\big(U^-(0)-U^-(-h_-)\big)^2=\int_{-h_-}^0 \big(U^-(x_2)-U^-(-h_-)\big)^2dx_2.
\end{equation}
For any $c\in \mathbb{R}\setminus U^-\big([-h_-, 0]\big)$, using  \eqref{E: y-c out}, we obtain that
\begin{equation}\label{E:pkY-0-c-out}
\begin{split}
    \partial_K Y^-(c, 0)\big(U^-(0)-c\big)^2=& \frac{1}{\big(\int_{-h_-}^0 \frac{1}{(U^--c)^2}dx_2\big)^2}\int_{-h_-}^0\big(U^-(x_2)-c\big)^2\big(\int_{-h_-}^{x_2}\frac{1}{(U^--c)^2}dx_2'\big)^2dx_2 \\
    \leq & \int_{-h_-}^0 \big(U^-(x_2)-c\big)^2dx_2.
\end{split}
\end{equation}
Since $Y^-(c, k)$ has a logarithmic singularity  near $c=U^-(0)$, we have 
\begin{equation}\label{E:pkY-0-c-0-}
    \lim_{c\rightarrow U^-(0)}\partial_K Y^-(c, 0)\big(U^-(0)-c\big)^2=0.
\end{equation}
Similarly, if $c$ is near $U^+(0)$, we obtain that 
\begin{equation}\label{E:pkY-0-c-0+}
    \lim_{c\rightarrow U^+(0)}-\epsilon\partial_K Y^+(c, 0)\big(U^+(0)-c\big)^2=0.
\end{equation}
Now, we compute $\partial_K F(c, 0, \epsilon)$ for each $c\in \mathcal{E}$.
If $c=U^-(-h_-)\notin U^+([0, h_+])$, we use the  formula \eqref{E:pKF} and estimates \eqref{E:pkY-0-c-h_-} and \eqref{E:pkY+-0-c-out} to compute
\begin{align*}
    \partial_K F(c, 0, \epsilon)\leq & \int_{-h_-}^0\big( U^-(x_2)-U^-(-h_-)\big)^2dx_2+\epsilon \int_0^{h_+}\big(U^+(x_2)-U^-(-h_-)\big)^2dx_2-\frac{\sigma}{\rho^-},\\
    = & m\big(U^-(-h_-)\big)-\frac{\sigma}{\rho^-}<0.
\end{align*}
According to Lemma \ref{L:p-KK-F}, we have $\partial_K F(c, k, \epsilon)<\partial_K F(c, 0, \epsilon)$ for $c\in \mathbb{R}\setminus \big(U^-((-h_-, 0))\cup U^+((0, h_+))\big)$. Hence, $\partial_K F\big(U^-(-h_-), k, \epsilon\big)<\partial_K F\big(U^-(-h_-), 0, \epsilon\big)<0$. 

If $c=U^-(0)\notin U^+([0, h_+])$, we use \eqref{E:p-KK-Y+}, \eqref{E:pkY-0-c-0-}, and \eqref{E:pkY+-0-c-out} to compute that 
\begin{align*}
    \partial_K F\big(U^-(0), k, \epsilon\big)=&-\epsilon\partial_K Y^+\big(U^-(0), k\big)\big(U^+(0)-c\big)^2\leq -\epsilon \partial_K Y^+\big(U^-(0), 0\big)\big(U^+(0)-c\big)^2\\
    \leq & \epsilon\int_0^{h_+}\big(U^+(x_2)-U^-(0)\big)^2dx_2<m\big(U^-(0)\big)<0. 
\end{align*}
Similarly, using the above estimates and logarithmic singularity of $\partial_K Y^\pm$ near $U^\pm(0)$ respectively, one may check that for all $c\in \mathcal{E}$,
\begin{align*}
    \partial_K F(c, k, \epsilon)\leq m(c)<0.
\end{align*}
\textit{Step 2.} We will show that $F(c, 0, \epsilon)<0$ if $c\in \mathcal{E}$. %It suffices to consider the following cases. 
If $c\notin U^+([0, h_+])$, using \eqref{E:y+out}, we compute
\begin{align*}
    Y^+(c, 0)=\frac{(U^+)'(0)}{U^+(0)-c}-\frac{1}{\big(U^+(0)-c\big)^2\int_0^{h_+}\frac{1}{(U^+-c)^2}dx_2}.
\end{align*}
If $c=U^-(-h_-)\notin U^+([0, h_+])$, we use \eqref{E:U-h-y-}, \eqref{E:Y-k-0outU-}, and \eqref{E:A-F<0} to obtain that
\begin{equation}\label{E:F-h-0-e}
    F\big(U^-(-h_-), 0, \epsilon\big)=-g+\epsilon\frac{1}{\int_0^{h_+}\frac{1}{(U^+-c)^2}dx_2}<0.
\end{equation}
If $c=U^-(-h_-)=U^+(0)$ or $c=U^-(-h_-)=U^+(h_+)$, then $F(c, 0, \epsilon)=-g<0$. Following a similar argument for $c=U^\pm(0)$ and $U^+(h_+)$, we have $F(c, 0, \epsilon)<0$ for $c\in \mathcal{E}$ and $\mathcal{I}=\emptyset$.

Combining the results from step 1 and step 2, we obtain that   $F(c, k, \epsilon)<0$ for $c\in \mathcal{E}$ and all $k\in \mathbb{R}$ if $\mathcal{I}=\emptyset$ by Lemma \ref{L:p-KK-F} and the logarithmic singularity of $Y^+(c, k)$ near $c=U^+(0)$(resp. $Y^-(c, k)$ near $c=U^-(0)$). In other words, in this case, there is no singular mode at $c\in \mathcal{E}$. 

Now, we first consider the case of $\max_{[-h_-, 0]} U^-<\min_{[0, h_+]} U^+$. Let $k_0$ be defined in Lemma \ref{L:c-large k}. There are exactly two wave speed $c^\pm(k_0)\in \mathbb{C}$ such that $F\big(c^\pm(k_0), k_0, \epsilon\big)=0$. According to Lemma \ref{L:cite-yY-} and Lemma \ref{L:yY+}, $F_I(c, k, \epsilon)\neq 0$ if $c\in U^-\big((-h_-, 0)\big)\cup U^+\big((0, h_+)\big)$. Because of the non-existence of singular modes at $c\in \mathcal{E}$, $F(c, k, \epsilon)\neq 0$ for all $c\in U^-\big([-h_-, 0]\big)\cup U^+\big([0, h_+]\big)$. By the compactness and continuity of $F$, there exist open sets $B_{1, 2}\subset \mathbb
C$ satisfying $U^-([-h_-, 0])\subset B_1$ and $U^+([0, h_+])\subset B_2$ such that $F(c, k, \epsilon)\neq 0$ whenever $c\in B_1\cup B_2$ and $k\in \mathbb{R}$. By \eqref{E:YI-}, \eqref{E:Y-cauchy integral-}, and \eqref{E:y--estimate}, we notice that if $c$ is real and $|c|$ is large, the leading term of $\big(U^-(0)-c \big)^2 Y^-(c, k)$ is $\big(U^-(0)-c\big)^2kcoth(kh_-)$. Similarly, we use \eqref{E:YI+}, \eqref{E:Y-cauchy integral+}, and \eqref{E:y+-estimate} to obtain the leading term of $\big(U^+(0)-c\big)^2Y^+(c, k)$ is $-\big(U^+(0)-c\big)^2kcoth(kh_+)$. Hence, for $|k|\in [0, k_0]$, $F(c, k, \epsilon)$ behaves like a quadratic function in $c$ with a positive uniformly bounded leading coefficient if $c$ is real and $|c|$ is large. Therefore, there exists $B_c>0$, such that if $F(c, k, \epsilon)=0$ and $|k|\in [0, k_0]$, then $|c|<B_c$. 
We choose $B_3\subset \mathbb{C}$ be a disk centered at the origin with radius larger than $B_c$ such that $c^\pm(k_0)\in B_3$ and $B_{1, 2}\subset B_3$. Then $F(c, k, \epsilon)\neq 0$ for all $c\in \partial B_3$ and $|k|\in [0, k_0]$.  Let $\Omega\subset B_3\setminus (B_1\cup B_2)$ containing $(\mathbb{R}\cap B_3)\setminus (B_1\cup B_2)$ such that $\partial\Omega$ is sufficiently close to $\partial B_3$, $\partial B_2$, and $\partial B_1$. Then $\Omega$ can be chosen such that $F(c, k, \epsilon)\neq 0$ for all $c\in \partial \Omega$ and $|k|\in [0, k_0]$. We consider 
\begin{equation}\label{E:degree}
    n(k):=\int_{\partial \Omega} \frac{\partial_c F}{F}(c, k, \epsilon)dc. 
\end{equation}
Since $c^\pm(k_0)$ are the unique roots of $F(\cdot, k_0, \epsilon)=0$, $n(k_0)=2$. Since $n(k)$ is continuous in $k$, $n(k)=2$ for all $|k|\in [0, k_0]$. Hence, for all $k\in \mathbb{R}$, $n(k)\equiv 2$. Moreover, by Remark \ref{R:c-pm-extend}, we obtain that $c^\pm(k)$ obtained in Lemma \ref{L:c-large k} can be extended to be even and analytic functions  for all $k\in \mathbb{R}$ and they are the only roots of $F(\cdot, k, \epsilon)=0$. 
Since $F(c, k, \epsilon)\in \mathbb{R}$ for all $c\in \mathbb{R}\setminus \big(U^-([-h_-, 0])\cup U^+([0, h_+])\big)$, $\partial_c F(c^\pm(k), k, \epsilon)$ does not change sign and theirs signs are the same as the ones given in Lemma \ref{L:c-large k}. 

Let us consider the remaining case of $\max_{[-h_-, 0]} U^-(x_2)=\min_{[0, h_+]} U^+(x_2)$. Since $F(c, k, \epsilon)\neq 0$ for all $k\in \mathbb{R}$ and $c\in [\min U^-, \max U^+]$, there exists $B_4\subset \mathbb{C}$ containing $[\min U^-, \max U^+]$ such that $F(c, k, \epsilon)\neq 0$ for all $k\in \mathbb{R}$ and $c\in B_4$. Then, there exists a big enough $\Omega_+\subset\{c\geq \max U^+\}\setminus B_4$ containing $c^+(k_0)$ which is a bounded region such that $F(c, k, \epsilon)\neq 0$ for all $c\in \partial \Omega_+$ and $|k|\in [0, k_0]$. We can apply \eqref{E:degree} and a similar argument as above on such $\Omega_+$ to obtain that $c^+(k)$ which is obtained in Lemma \ref{L:c-large k} can be extended to be an even and analytic function for all $k\in \mathbb{R}$. Similarly, we can choose $\Omega_-$ big enough containing $c^-(k_0)$ such that $F(c, k, \epsilon)\neq 0$ on $\partial \Omega_-$ and obtain that $c^-(k)$ can be extended to all $k\in \mathbb{R}$. Moreover, let $\tilde{\Omega}$ be a large bounded region containing $c^\pm(k_0)$ and apply \eqref{E:degree} on $(\Tilde{\Omega}\setminus B_4)\cap \{c_I>-\alpha\}$ where $0<\alpha\ll 1$ and $(\Tilde{\Omega}\setminus B_4)\cap \{c_I<\alpha\}$ respectively. We can obtain that $c^\pm(k)$ are the only roots of $F(\cdot, k, \epsilon)=0$ for all $k\in \mathbb{R}$. 

\noindent \textit{* Case 2.} $\mathcal{I}\neq \emptyset$, $\max_{c\in \{a, b\}}m(c)<\frac{\sigma}{\rho^-}$, and $(U^+)''(U^-)''>0$. Following the same argument of Case 1, it remains to consider $c\in \mathcal{I}$. We notice that $F_I(c, k, \epsilon)\neq 0$ for any $c\in \mathcal{I}$ if $sgn(Y^-_I)=-sgn(Y_I^+)$. Lemma \ref{L:cite-yY-} and Lemma \ref{L:yY+} imply that if $(U^-)''(U^+)''>0$, then $F(c, k, \epsilon)\neq 0$ for $c\in \mathcal{I}$. The rest of the proof can be completed by using the same continuation argument of $n(k)$ as the one in Case 1. 
\end{proof}

We notice that $c=a$ is a point on the boundary of the domain of analyticity of $F(\cdot, k, \epsilon)$. Even more, it is possible that it corresponds to a neutral limiting mode. In the following lemma, we shall study the behavior of $F$ near $c=a$ for any $\epsilon\geq 0$.

\begin{lemma}\label{L:p-c-F-a}
Assume that $U^\pm\in C^6$, $(U^\pm)'$, and $(U^\pm)''>0$. For each $\epsilon\geq 0$, if one of the following holds,
\begin{enumerate}
    \item $a=U^+(0)\leq U^-(-h_-)$, or 
    \item  $a=U^-(-h_-)<U^+(0)$ and 
    \begin{equation}\label{A:p-c-a-}
        -(U^-)'(-h_-)+\epsilon (U^+)'(h_+)<2\epsilon\Big(\big(U^+(0)-U^-(-h_-)\big)\int_0^{h_+}\frac{1}{(U^+-c)^2}dx_2 \Big)^{-1},
    \end{equation}
\end{enumerate}
then $\partial_c F(a, k, \epsilon)<0$ for all $k\in \mathbb{R}$.
\end{lemma}
\begin{proof}
Let $K:=k^2$. Lemma \ref{L:p-KK-F} implies that $\partial_c F(a, K, \epsilon)<\partial_c F(a, 0, \epsilon)$ for all $K\geq 0$.  We shall compute $\partial_c F(a, 0, \epsilon)$ in the following cases.\\
\textit{* Case 1.} $a=U^-(-h_-)$. We apply \eqref{E:-h-pcY-} and \eqref{E:Y-h-k0} to compute $\partial_c F(a, 0, \epsilon)$. 
If $U^-(-h_-)=U^+(0)$, Lemma \ref{L:yY+} implies that both $\big(U^+(0)-c\big)Y^+$ and $\partial_c Y^+\big(U^+(0)-c\big)^2$ vanish at $c=U^+(0)$. Since $(U^-)'>0$, we compute \eqref{E:partial-c-F} to obtain that
\begin{align*}
   \partial_c F(a, 0, \epsilon)=  -(U^-)'(-h_-)-\epsilon (U^+)'(0)<0, \quad \forall \epsilon\geq 0.  
\end{align*}
If $U^-(-h_-)\notin U^+([0, h_+])$, we use \eqref{E:y+out} and \eqref{E:partial-c-Y+} to compute that 
\begin{align*}
    -\big(U^+(0)-a\big)^2\partial_c Y^+(a, 0)= & \frac{1}{\big(\int_0^{h_+} \frac{1}{(U^+-a)^2}dx_2 \big)^2}\int_0^{h_+} (U^+)'' \big(\int_{x_2}^{h_+} \frac{1}{(U^+-a)^2}dx_2' \big)^2 dx_2\\
    \leq & \int_0^{h_+} (U^+)''(x_2) dx_2 =(U^+)'(h_+)-(U^+)'(0),\\
    Y^+(a, 0)= \frac{(U^+)'(0)}{U^+(0)-a}&-\frac{1}{\big(U^+(0)-a\big)^2\int_0^{h_+}\frac{1}{(U^+-a)^2}dx_2}. 
\end{align*}
Using \eqref{A:p-c-a-}, this implies that 
\begin{align*}
    \partial_c F(a, 0, \epsilon)\leq -(U^-)'(-h_-)+\epsilon\Big((U^+)'(h_+)-\frac{2}{\big(U^+(0)-U^-(-h_-) \big)\int_0^{h_+}\frac{1}{(U^+-U^-(-h_-))^2}dx_2} \Big)<0. 
\end{align*}
\textit{* Case 2.} $a=U^+(0)$. It remains to consider $U^+(0)\notin U^-([-h_-, 0])$. Using \eqref{E:p-c-1-0outU-h-} and \eqref{E:Y-k-0outU-}, we compute 
\begin{align*}
    \partial_c F(a, 0, \epsilon)\leq -(U^-)'(-h_-) -\frac{2}{\big(U^-(0)-U^+(0)\big)\int_{-h_-}^0\frac{1}{\big(U^--U^+(0) \big)^2}dx_2}-\epsilon (U^+)'(0)<0.  
\end{align*}
\end{proof}

\begin{remark}
When $(U^\pm)'>0$, $(U^\pm)''\neq 0$, and $a=U^-(-h_-)<U^+(0)$, there exists $\epsilon_0>0$ such that   \eqref{A:p-c-a-} holds for all $\epsilon\in [0, \epsilon_0)$. Particularly, when $\epsilon=0$, the result in Lemma \ref{L:p-c-F-a} coincides the result for the capillary gravity water wave linearized at monotone convex shear flows.
\end{remark}

According to Lemma \ref{L:p-c-F-a}, we can apply bifurcation analysis near $c=a$ and seek instability if $F(a, k, \epsilon)=0$.  We introduce $g_*$ which tells us when a bifurcation may happen at $c=a$. In the following lemma, we first study the roots of $F(a, \cdot, \epsilon)=0$ based on the value of $g$ relative to $g_*$.
\begin{lemma}\label{L:g*}
Assume that $U^\pm$ satisfies \eqref{E:U monotone} and $(U^\pm)''\neq 0$. Let $a$ be defined in \eqref{E:ab}.
Let 
\begin{equation}\label{E:g*}
    g_*=\max\{F(a, k, \epsilon)+g\; |\; k\in \mathbb{R}\},
\end{equation}
then the following holds.
\begin{enumerate}
    \item $g_*> 0$ for any $\epsilon>0$.
    \item If $g>g_*$, then $F(a, k, \epsilon)<0$ for all $k\in \mathbb{R}$ and $\epsilon\geq 0$.
    \item If one of the following cases holds,
    \begin{enumerate}
        \item $a=U^+(0)\notin U^-([-h_-, 0])$ and 
       \begin{equation}\label{A:1/g-+0}
            0<\frac{1}{g(1-\epsilon)}<\int_{-h_-}^0 \frac{1}{\big( U^-(x_2)-U^+(0)\big)^2}dx_2,
        \end{equation}
        \item $a=U^+(h_+)\notin U^-([-h_-, 0])$ and 
        \begin{equation}\label{A:1/g-h+}
            0<\frac{1}{g(1-\epsilon)}<\int_{-h_-}^0 \frac{1}{\big( U^-(x_2)-U^+(h_+)\big)^2}dx_2
        \end{equation}
        \item $a=U^+(0)\in U^-([-h_-, 0])$ or $U^+(h_+)\in U^-([-h_-, 0])$,
        \item $a=U^-(-h_-)$ or $U^-(0)$,
    \end{enumerate}
then the following holds.
\begin{enumerate}
    \item[(i)] If $0<g=g_*$, then there exists a unique $k_*>0$ such that $F(a, \pm k_*, \epsilon)=0, \pm k_*\partial_k F(a, \pm k_*, \epsilon)<0$, and $F(a, k, \epsilon)<0$ for all $|k|\neq k_*$. 
    \item[(ii)] If $g\in (0, g_*)$, then there exist $k_*^+>k_*^->0$ such that
    \begin{align*}
        F(a, k, \epsilon)<0, \; |k|\notin (k_*^-, k_*^+);&\quad F(a, k, \epsilon)>0,\; |k|\in (k_*^-, k_*^+);\\  F(a, \pm k_*^\pm, \epsilon)=0,&\quad \mp \partial_k F(a, k_*^\pm, \epsilon)>0. 
    \end{align*} 
\end{enumerate}
\end{enumerate}
\end{lemma}
\begin{proof}
Since $Y^\pm(c, k)$ are comparable to $\mp k$ as $|k|$ becomes large, $F(a, k, \epsilon)$ behaves like a quadratic function of $k$ with a negative leading coefficient. Hence, $g_*$ exists. If $a=U^-(-h_-)$, $Y^-(a, 0)$ is in the form of \eqref{E:Y-k-0outU-}. If $a=U^-(0)$, \eqref{E:Y--log} implies that $\big(U^-(0)-a\big)^2Y^-(a, 0)=0$. If $a\notin U^-([-h_-, 0])$, then one can use \eqref{E:Y-k-0outU-} to compute $F(a, 0, \epsilon)$. If $a=U^+(0)$, \eqref{E:Y-+log} implies that $\big(U^+(0)-a\big)^2Y^+(a, 0)=0$. If $a=U^+(h_+)$, we apply \eqref{E:y+Uh+} to compute
\begin{equation}\label{E:Y+h+}
    Y^+(a, 0)=\frac{(U^+)'(0)}{U^+(0)-a}.
\end{equation}
If $a\notin U^+([0, h_+])$, we apply \eqref{E:y+out} to obtain that 
\begin{equation}\label{E:Y+out}
    Y^+(a, 0)=\frac{(U^+)'(0)}{U^+(0)-a}+\frac{1}{\big(U^+(0)-a\big)^2\int_0^{h_+}\frac{1}{(U^+-a)^2}dx_2}. 
\end{equation}
Using these formulas, one can obtain that $g_*\geq F(a, 0, \epsilon)+g\geq \epsilon g>0$ since $\epsilon>0$. If $a=U^+(0)\notin U^-([-h_-, 0])$, 
\begin{equation}\label{E:F-a-0-ep}
    F(a, 0, \epsilon)=g(\epsilon-1)+\frac{1}{\int_{-h_-}^0 \frac{1}{(U^--a)^2}dx_2}. 
\end{equation}
Hence, $F(a, 0, \epsilon)<0$ if \eqref{A:1/g-+0} holds. If $a=U^+(h_+)\notin U^-([-h_-, 0])$, $F(a, 0, \epsilon)$ is in the form of \eqref{E:F-a-0-ep} and $F(a, 0, \epsilon)<0$ if \eqref{A:1/g-h+} holds. If $a=U^-(-h_-)\notin U^+([0, h_+])$, 
\begin{equation}\label{E:F-a-0-ep-h-}
    F(a, 0, \epsilon)=g(\epsilon-1)-\epsilon \Big(\int_0^{h_+}\frac{1}{(U^+-a)^2}dx_2\Big)^{-1}<0.
\end{equation}
If $a=U^-(-h_-)=U^+(0)$ or $U^-(-h_-)=U^+(h_+)$, then $F(a, 0, \epsilon)=g(\epsilon-1)<0$. If $a=U^-(0)\notin U^+([0, h_+])$, then $F(a, 0, \epsilon)$ is in the form of \eqref{E:F-a-0-ep-h-} and $F(a, 0, \epsilon)<0$. If $a=U^-(0)=U^+(0)$ or $U^-(0)=U^+(h_+)$, then $F(a, 0, \epsilon)=g(\epsilon-1)<0$. 
The rest of the proof for (3) follows directly from the concavity of $F(a, k, \epsilon)$ in $K=k^2$ by \eqref{E:p-KK-F} and the definition of $g_*$.
\end{proof}

In the next lemma, we analyze the roots of $F(\cdot, k, \epsilon)=0$ near $c=a$ if $F(a, k_*, \epsilon)=0$ for some $k_*>0$ based on the different properties of $F$ at $(a, k_*)$. This is a local result near $(a, k_*)$ for all $\epsilon\geq 0$.

\begin{lemma}\label{L:bifurcation-a}
Assume that $U^\pm$ satisfies \eqref{E:U monotone} and $(U^\pm)''>0$, and $F(a, k_*, \epsilon)=0$ for some $k_*>0, \epsilon\geq 0$. If $\partial_c F(a, k_*, \epsilon)<0$, then there exist $\delta>0$ and a unique $\mathcal{C}^-\in C^{1, \alpha}\big([k_*-\delta, k_*+\delta], \mathbb{C} \big)$ for any $\alpha\in [0, 1)$ such that 
\begin{align*}
    \mathcal{C}^-(k_*)=a, \quad    F(\mathcal{C}^-(k), k, \epsilon)=0 \quad  \forall  |k|\in[k_*-\delta, k_*+\delta]. 
\end{align*}
And $\mathcal{C}^-(k)$ satisfies the following.
\begin{enumerate}
    \item If $\partial_k F(a, k_*, \epsilon)=0$, then $(\mathcal{C}^-)'(k_*)=0$, $\mathcal{C}^-_I\equiv 0$ and $\mathcal{C}^-(k)<a$ for all $0<|k-k_*|<\delta$. Moreover, $\pm k_*$ are the only roots of $F(a, \cdot, \epsilon)=0$.
    \item If $\pm \partial_k F(a, k_*, \epsilon)>0$, then $\pm (\mathcal{C}^-_R)'(k_*)>0$. Moreover,
    \begin{align*}
        \mathcal{C}_R^-(k)<a,\; \mathcal{C}_I^-(k)=0, \; \forall 0<\pm(k_*-k)\leq \delta,\\
        \mathcal{C}_R^-(k)>a, \; \mathcal{C}_I^-(k)>0, \; \forall 0<\pm(k-k_*)\leq \delta. 
    \end{align*}
\end{enumerate}
\end{lemma}
\begin{proof}
Recall that $\mathcal{I}$ is defined in \eqref{E:EI-definition}. 
According to Lemma \ref{L:cite-yY-} and Lemma \ref{L:yY+}, $F_I$ is not continuous at $c\in U^-\big( (-h_-, 0]\big)\cup U^+\big([0, h_+) \big)\subset\mathbb{C}$ in general. Hence, we extend $F(c, k, \epsilon)$ by letting 
\begin{equation}\label{E:extend-F}
    \tilde{F}(c, k, \epsilon)=\tilde{F}_R(c, k, \epsilon)+i\tilde{F}_I(c, k, \epsilon)\in \mathbb{C}, \quad \tilde{F}=F \; \text{for}\; c_I\geq 0,
\end{equation}
be a $C^{1, \alpha}$ extension of $F$ into a neighborhood of $(a, k)\in \mathbb{C}\times\mathbb{R}$ for any $\alpha\in [0, 1)$ and each $\epsilon\geq 0$.   By Lemma \ref{L:F-regularity}, since Cauchy-Riemann equations hold and $F(c, k, \epsilon)\in \mathbb{R}$ for all $c<a$, the $2\times 2$ Jacobian matrix of $D_c \tilde{F}$ satisfies
\begin{align*}
    D_c \tilde{F}(a, k_*, \epsilon)=\begin{pmatrix} \partial_{c_R} \tilde{F}_R
 & \partial_{c_I}\tilde{F}_R \\ \partial_{c_R} \tilde{F}_I & \partial_{c_I} \tilde{F}_I    \end{pmatrix}\Big|_{\big(a, k_*, \epsilon \big)}=\partial_c F(a, k_*, \epsilon)I_{2\times 2}. 
\end{align*}
Since $\partial_c F(a, k_*, \epsilon)<0$, we apply the Implicit Function Theorem and obtain a $C^{1, \alpha}$ complex-valued function $\mathcal{C}^-(k)$ and $\delta>0$ for any $\alpha\in [0, 1)$ such that $\mathcal{C}^-(k_*)=a$ and $\tilde{F}(\mathcal{C}^-(k), k, \epsilon)=0$  on $[k_*-\delta, k_*+\delta]$.  We know that $F_R\in C^1$ and $\partial_{c_R} F_R(a, k_*, \epsilon)=\partial_c F(a, k_*, \epsilon)<0$. According to the Implicit Function Theorem, there exist  $\Tilde{\delta}>0$ and a $C^{1, \alpha}$ real-valued function $\tilde{\mathcal{C}}(k)$ for any $\alpha\in [0, 1)$  on $[k_*-\Tilde{\delta}, k_*+\tilde{\delta}]$ such that $\Tilde{\mathcal{C}}(k_*)=a$ and $F_R(\Tilde{\mathcal{C}}(k), k, \epsilon)=0$. Since $F$ is real-valued if $c\leq a$, the uniqueness of solutions by the Implicit Function Theorem ensures that $\mathcal{C}^-(k)=\tilde{\mathcal{C}}(k)\in \mathbb{R}$ if $\Tilde{\mathcal{C}}(k)\leq a$.  Moreover, $F\big(\mathcal{C}^-(k), k, \epsilon\big)=\tilde{F}\big(\mathcal{C}^-(k), k, \epsilon\big)=0$. 

We first consider the case of  $\partial_k F(a, k_*, \epsilon)=0$ with  $k_*>0$. The uniqueness of $k_*$ is ensured by \eqref{E:p-KK-F}. Since $\partial_{c_R} F_R(a, k_*, \epsilon)=\partial_c F(a, k_*, \epsilon)<0$ and $F(a, k, \epsilon)$ is concave in $K=k^2$, we obtain that 
\begin{align*}
    F_R(c, k_*, \epsilon)>0, \; \forall 0<a-c\ll 1, \quad F_R(a, k, \epsilon)<0, \; \forall 
    k\in \mathbb{R}^+\setminus\{k_*\}. 
\end{align*}
Then there exist $k$ near $k_*$ and $c$ near $a$ with $c<a$ such that $F_R(c, k, \epsilon)=0$. Moreover, the uniqueness of solutions by the Implicit Function Theorem ensures that $c=\tilde{\mathcal{C}}(k)$. Therefore, $\mathcal{C}^-(k)=\Tilde{\mathcal{C}}(k)<a$ for $k\neq k_*$ close to $k_*$. 

Now, we consider the case of $\partial_k F(a, k_*, \epsilon)>0$. The proof for the case of $\partial_k F(a, k_*, \epsilon)<0$ is similar. Since $\partial_c F(a, k_*, \epsilon)<0$, 
\begin{align*}
    (\mathcal{C}^-)'(k_*)=-\frac{\partial_k F(a, k_*, \epsilon)}{\partial_c F(a, k_*, \epsilon)}>0. 
\end{align*}
Hence, $\mathcal{C}_R^-(k)>a$ for $k\in (k_*, k_*+\delta]$. By the Mean Value Theorem, there exists $\gamma$ between $0$ and $\mathcal{C}_I^-(k)$ such that 
\begin{align*}
    0=F_I(\mathcal{C}^-(k), k, \epsilon)=F_I(\mathcal{C}^-_R(k), k, \epsilon)+\mathcal{C}^-_I(k)\partial_{\mathcal{C}_I}F_I\big(\mathcal{C}_R^-(k)+i\gamma, k, \epsilon\big).
\end{align*}
By the $C^{1, \alpha}$ regularity of $F$ and $\mathcal{C}^-(k)$, we obtain that 
\begin{align*}
    \mathcal{C}^-_I(k)&=-\frac{F_I\big(\mathcal{C}_R^-(k), k, \epsilon\big)}{\partial_{\mathcal{C}_I} F_I\big(\mathcal{C}_R^-(k)+i\gamma, k, \epsilon \big)}\\
    &=-\frac{Y^-_I\big(\mathcal{C}_R^-(k), k, \big)\big(U^-(0)-\mathcal{C}_R^-(k) \big)^2-\epsilon Y_I^+\big(\mathcal{C}_R^-(k), k, \big)\big(U^+(0)-\mathcal{C}_R^-(k)\big)^2}{\partial_{\mathcal{C}_I}F_I\big(\mathcal{C}_R^-(k), k, \epsilon\big)+O(|\mathcal{C}_I^-(k)|^\alpha)}\\
    &=-\frac{Y^-_I\big(\mathcal{C}_R^-(k), k, \big)\big(U^-(0)-a+O(|k-k_*|)\big)^2-\epsilon Y_I^+\big(\mathcal{C}_R^-(k), k, \big)\big(U^+(0)-a+O|k-k_*|\big)^2}{\partial_cF(a, k_*, \epsilon)+O(|k-k_*|^\alpha)}.
\end{align*}
Since $(U^\pm)''>0$, using \eqref{E:YI-} and \eqref{E:YI+}, we have that the numerator is always positive for small $\delta$. This implies that $\mathcal{C}_I^-(k)>0$ when $0<k-k_*\ll 1$.   
\end{proof}

\begin{remark}
1). If $\partial_c F(a, k_*, \epsilon)>0$, to obtain the same results above, the assumption in Lemma \ref{L:bifurcation-a}(2) should be changed to $\pm \partial_k F(a, k_*, \epsilon)<0$. 2). If $(U^\pm)''<0$, the signs in the result of Lemma \ref{L:bifurcation-a}(2) should be reversed. 
\end{remark}

We are in the position to prove Theorem \ref{T:instability fixed e}.

\noindent \textit{Proof of Theorem \ref{T:instability fixed e}}. It remains to prove the last part of Theorem \ref{T:instability fixed e}(2)(b)(ii). The rest follows directly from Lemma \ref{L:g*},  Lemma \ref{L:bifurcation-a}, and Remark \ref{R:c-pm-extend}. Suppose that $a=U^-(-h_-)\notin U^+([0, h_+])$. Then \eqref{E:F-h-0-e} shows that $g_*\geq F(U^-(-h_-), 0, \epsilon)+g>0$ if $\epsilon>0$. Similarly, one may check that $F(U^+(0), 0, \epsilon)+g>0$ by using \eqref{E: y-c out} if $a=U^+(0)\notin U^-([-h_-, 0])$ and $\epsilon>0$. In the case of $a=U^-(-h_-)=U^+(0)$, since $F(a, 0, \epsilon)+g=0$, $g_*=0$ if and only if $\partial_K F(a, 0, \epsilon)\leq 0$, where $K:=k^2$. The proof of the theorem is completed. \hfill$\square$ 
\end{section}

\begin{section}{Ocean-Air Model}\label{S:Ocean-Air Model}
In this section, we study the spectral distribution for the case of   $0\leq \epsilon \ll 1$. We shall treat the interface problem as a perturbation to the capillary-gravity water wave problem. According to Theorem \ref{T:semicircle}, all the unstable modes must have wave speed lie in a big upper semicircle \eqref{E:semicircle} if \eqref{E:semicircle assumption} is satisfied. We first locate all the possible locations of neutral limiting modes inside $[a, b]$ where $a, b$ are defined in \eqref{E:ab} as $\epsilon$ tends to zero, and then seek unstable modes nearby when $\epsilon$ is perturbed. We give a complete picture of the eigenvalue distribution for a special case of $(U^\pm)'>0, (U^\pm)''\neq 0$, and $U^-(0)<U^+(0)$.  In the following lemma, we first address the possible locations of neutral limiting mode as $\epsilon$ tends to $0$.

\begin{lemma}\label{L:nlm-e-0}
Assume that $U^\pm$ satisfy \eqref{E:U monotone}. Let $\mathcal{I}$ be defined in $\eqref{E:EI-definition}$ and $\mathcal{I}=\emptyset$.
If a sequence of unstable modes $(c_n, k_n, \epsilon_n)$ with $\epsilon_n>0$ converges to $(c_*, k_*, 0)$ as $n \rightarrow \infty$, then $c_*=\lim_{n\rightarrow \infty}c_n \in \{U^-(-h_-)\}\cup U^+([0, h_+])\setminus U^-(0)$ or
\begin{align*}
\Omega:=\{c_R+ic_I, c_R\in \mathbb{R}, c_I\in \mathbb{R}^+| (c_R-\frac{U^-(-h_-)+U^-(0)}{2})^2+c_I^2<(\frac{U^-(-h_-)-U^-(0)}{2})^2\}.
\end{align*}
\end{lemma}
\begin{proof}
Recall that $F(c, k, \epsilon)$ is defined in \eqref{E:F-definition}.  Then, Lemma \ref{L:F-regularity} implies that  $\lim_{n\rightarrow \infty}F(c_n, k_n, \epsilon_n)=F(c_*, k_*, 0)=0$. According to Theorem 1.1 in \cite{LZ}, $c_*\in (\mathbb{R}\setminus U^-((-h_-, 0]))\cup \Omega$. Since $F(c_n, k_n, \epsilon_n)=0$, Theorem \ref{T:semicircle} implies that $c_n$ stays in the upper semicircle \eqref{E:semicircle}. By the definition of $(c_n, k_n, \epsilon_n)$, for any $0<\alpha\ll 1$, there exists $N_0>0$ such that whenever $n>N_0$, $\epsilon_n, |c_n-c_*|, |k_n-k_*|<\alpha$. Hence, $c_*\in \Big([a, b]\setminus U^-\big((-h_-, 0]\big) \Big)\cup \Omega$, where $a$ and $b$ are defined in \eqref{E:ab}. We shall discuss the following possible cases to shrink this region. 

\noindent\textit{Case 1. $U^-(0)<U^+(0)$.}  Suppose that $c_*\in \big(U^-(0), U^+(0)\big)$. By Remark \ref{R:c-pm-extend}, when $\epsilon=0$, $c^+(k)$ obtained in Lemma \ref{L:c-large k} can be extended to all $k\in \mathbb{R}$.
This shows that $c_*$ is a simple root of $F(c, k_*, 0)$ and according to \eqref{E:ext-c-pm-pcF}, $\partial_c F(c_*, k_*, 0)>0$. Since $F(c, k, \epsilon)$ is analytic near $c=c*$, there exist $\delta, \epsilon_0>0$ and a smooth function $\mathcal{C}(k, \epsilon): [k_*-\delta, k_*+\delta]\times[-\epsilon_0, \epsilon_0]\rightarrow \mathbb{C}$ such that 
\begin{align*}
    F\big(\mathcal{C}(k, \epsilon), k, \epsilon\big)=0, \quad \mathcal{C}(k_*, 0)=c_*. 
\end{align*}
By the uniqueness of solutions ensured by the Implicit Function Theorem and the definition of the sequence $\{(c_n, k_n, \epsilon_n)\}$, there exists $N>0$ such that for all $n\geq N$, $\mathcal{C}(k_n, \epsilon_n)=c_n$. This implies that $\mathcal{C}_I(k_n, \epsilon_n)>0$. 

Since $c_*\notin U^-([-h_-, 0])\cup U^+([0, h_+])$, $F_R\in C^\infty$ and $\partial_{c_R}F_R(c_*, k_*, 0)=\partial_c F(c_*, k_*, 0)>0$. Then the Implicit Function Theorem implies that there exists a smooth real-valued function $\Tilde{C}(k, \epsilon)$ for $k$ near $k_*$ and $\epsilon$ near $0$ such that 
\begin{align*}
    \Tilde{C}(k_*, 0)=c_*, \quad F_R\big(\Tilde{C}(k, \epsilon), k, \epsilon\big)=0.
\end{align*}
Since $F_I(c, k, \epsilon)=0$ if $c\in \big(U^-(0), U^+(0) \big)$, the uniqueness of solutions by the Implicit Function Theorem implies that $\mathcal{C}(k, \epsilon)=\Tilde{C}(k, \epsilon)\in \mathbb{R}$ if $\Tilde{C}\in \big(U^-(0), U^+(0)\big)$. Hence, $\mathcal{C}(k_n, \epsilon_n)\in \mathbb{R}$ for all $n>N$, which contradicts to the previous argument that $\mathcal{C}_I(k_n, \epsilon_n)>0$. Therefore, $c_*\notin \big(U^-(0), U^+(0)\big)$.  \\
\textit{Case 2. $U^+(h_+)\leq U^-(-h_-)$.} It suffices to consider $U^+(h_+)<U^-(-h_-)$. According to Remark \ref{R:c-pm-extend}, $\partial_c F(c_*, k_*, 0)<0$. By a similar argument as above, the lemma is proved. 
\end{proof}

This lemma implies that instability might occur near $U^-(-h_-)$ and $U^+([0, h_+])$ for $0<\epsilon\ll 1$. We shall analyze the bifurcation near $U^+(h_+)$ and $U^+(0)$  and give a complete picture of the eigenvalue distribution when $0<\epsilon\ll 1$. The strategy can be adapted to analyze some other cases including all the cases of $(U^\pm)', (U^\pm)''\neq 0$ and $\mathcal{I}\neq \emptyset$, where $\mathcal{I}$ is defined in \eqref{E:EI-definition}. 

\begin{lemma}\label{L:bifurcation-U+h}
Suppose that $U^\pm$ satisfy \eqref{E:U monotone}, $(U^\pm)'> 0,  U^-(0)<U^+(0)$, $(U^\pm)''\neq 0$, and $F\big(U^+(h_+), k_0, 0\big)=0$ for some $k_0> 0$. Then the following hold. 
\begin{enumerate}
    \item Suppose that $\partial_k F\big(U^+(h_+), k_0, 0 \big)\neq 0$. Then there exist $\epsilon_0>0$ such that for each $\epsilon\in (0, \epsilon_0)$ there exist $\delta_0>0$, $k(\epsilon)\in C^{1, \beta}\big([0, \epsilon_0], \mathbb{R}^+\big)$, and $\mathcal{C}^+(k)\in C^{1, \alpha}\big([k(\epsilon)-\delta_0, k(\epsilon)+\delta_0] \big)$ for any $\alpha, \beta\in [0, 1)$ such that  
\begin{align*}
    F\big(U^+(h_+),\pm k(\epsilon), \epsilon \big)=0,& \quad k(0)=k_0, \\
    F\big(\mathcal{C}^+(k), \pm k, \epsilon\big)=0, \; \forall|k|\in [k(\epsilon)-\delta_0,\;& k(\epsilon)+\delta_0], \quad \mathcal{C}^+\big(\pm k(\epsilon)\big)=U^+(h_+). 
\end{align*}
And for $c_I\in [0, U^+(h_+)-U^+(0)]$ and $|c_R-U^+(h_+)|\leq U^+(h_+)-U^+(0)$, $\mathcal{C}^+(k)$ is the unique root of $F(\cdot, k, \epsilon)=0$ with $|k|\in [k(\epsilon)-\delta_0, k(\epsilon)+\delta_0]$. 
Moreover, the following hold.
\begin{enumerate}
    \item If  $\partial_k F\big( U^+(h_+), k_0, 0\big)<0$, then 
    \begin{enumerate}
        \item  $k'(0)>0$ and $k(\epsilon)>k_0$;
        \item $\mathcal{C}^+(k)>U^+(h_+)$ for any $|k|\in \big(k(\epsilon), k(\epsilon)+\delta_0]$ and 
        \begin{align*}
         (U^+)''\mathcal{C}^+_I(k)<0,\quad  \mathcal{C}_R^+(k)<U^+(h_+),\;  \forall |k|\in [k(\epsilon)-\delta_0, k(\epsilon)\big).
\end{align*}
\end{enumerate}
    \item If  $\partial_k F\big(U^+(h_+), k_0, 0\big)>0$, then 
    \begin{enumerate}
        \item $k'(0)<0$ and $k(\epsilon)<k_0$;
        \item $\mathcal{C}^+(k)>U^+(h_+)$ for any $ |k|\in [k(\epsilon)-\delta_0, k(\epsilon)\big)$ and 
        \begin{align*}
         (U^+)''\mathcal{C}^+_I(k)<0,  \quad \mathcal{C}_R^+(k)<U^+(h_+),\; |k|\in \big( k(\epsilon), k(\epsilon)+\delta_0].
        \end{align*}
    \end{enumerate}
\end{enumerate}
\item Suppose that $\partial_k F\big(U^+(h_+), k_0, 0\big)=0$. Then there exist $\epsilon_1>0$ such that for each $\epsilon\in (0, \epsilon_1]$ there exist $k_2(\epsilon)>k_0>k_1(\epsilon)\geq 0$, $\delta>0$, and $\mathcal{C}_{1, 2}^+\in C^{1, \alpha}\big([k_{1, 2}(\epsilon)-\delta, k_{1, 2}(\epsilon)+\delta], \mathbb{C}\big)$ for any $\alpha\in [0, 1)$ such that 
    \begin{align*}
        &F\big(U^+(h_+), \pm k_{1, 2}(\epsilon), \epsilon\big)=0,\\ 
        F\big(\mathcal{C}_{1, 2}^+(k), \pm k, \epsilon\big)=0, \; \forall |k|\in &[k_{1, 2}(\epsilon)-\delta, k_{1, 2}(\epsilon)+\delta], \quad \mathcal{C}^+_{1, 2}\big(\pm k_{1, 2}(\epsilon)\big)=U^+(h_+). 
    \end{align*}
Moreover, $\mathcal{C}_{1, 2}^+(k)$ satisfy that 
\begin{align*}
    \mathcal{C}_1^+(k)>U^+(h_+), \; \forall |k|\in [k_1(\epsilon)-\delta,  k_1(\epsilon)\big),& \quad \mathcal{C}_2^+(k)>U^+(h_+), \; \forall |k|\in \big(k_2(\epsilon), k_2(\epsilon)+\delta], \\
    \mathcal{C}_{1 _R}^+(k)<U^+(h_+), \quad (U^+)''&\mathcal{C}_{1 _I}^+(k)<0, \; \forall |k|\in \big(k_1(\epsilon), k_1(\epsilon)+\delta],\\
    \mathcal{C}_{2 _R}^+(k)<U^+(h_+), \quad (U^+)''&\mathcal{C}_{2 _I}^+(k)<0, \; \forall |k|\in [k_2(\epsilon)-\delta, k_2(\epsilon)\big).
\end{align*}
In addition, for $|c_R-U^+(h_+)|<U^+(h_+)-U^+(0)$ and $c_I\in [0, U^+(h_+)-U^+(0)]$, $\mathcal{C}^+_1$ is the unique root of $F(\cdot, k, \epsilon)=0$ with $|k|\in [k_1(\epsilon)-\delta, k_1(\epsilon)+\delta]$, $\mathcal{C}^+_2$ is the unique root of $F(\cdot, k, \epsilon)=0$ with $|k|\in [k_2(\epsilon)-\delta, k_2(\epsilon)+\delta]$. 
\end{enumerate}
\end{lemma}

\begin{proof}
Let $K:=k^2$.  We shall prove the case of $\partial_K F\big(U^+(h_+), k_0, 0\big)<0$. The proof for the case of  $\partial_K F\big(U^+(h_+), k_0, 0\big)>0$ is similar. Since $\partial_k F\big(U^+(h_+), k_0, 0\big)<0$ and $k_0>0$, $\partial_K F\big(U^+(h_+), k_0, 0\big)<0$. And there exist $0<\epsilon_0\ll 1$ and a real-valued smooth function $K(\epsilon)$ such that $F\big(U^+(h_+), K(\epsilon), \epsilon\big)=0$ for all $\epsilon\in [-\epsilon_0, \epsilon_0]$ and $K(0)=k_0^2$. We compute that 
\begin{equation}\label{E:p-ep-F}
    \partial_\epsilon F(c, k, \epsilon)=-\big(U^+(0)-c\big)^2Y^+(c, k)+(U^+)'(0)\big(U^+(0)-c\big)+g. 
\end{equation}
We apply \eqref{E:Y+h+} to obtain that 
\begin{equation}\label{E:p-ep-F-h+-0}
    \partial_\epsilon F\big(U^+(h_+), 0, 0\big)=g>0
\end{equation}
According to \eqref{E:partial-k-Y+}, \begin{align*}
    \partial_{\epsilon K}F\big(U^+(h_+), k, \epsilon\big)=-\big(U^+(0)-U^+(h_+) \big)^2\partial_K Y^+\big( U^+(h_+), k\big)>0,
\end{align*} 
for all $k\in \mathbb{R}$ and $\epsilon\geq 0$. 
Hence, $\partial_\epsilon F\big( U^+(h_+), k_0, 0\big)>\partial_\epsilon F\big(U^+(h_+), 0, 0\big)>0$. 
Using the assumption that $\partial_K F\big(U^+(h_+), k_0, 0\big)<0$, we obtain that
\begin{align*}
    K'(0)=-\frac{\partial_\epsilon F}{\partial_K F}\big(U^+(h_+), k_0, 0 \big)>0.
\end{align*}
Then $\epsilon_0$ can be chosen small enough such that  $K'(\epsilon)>0$ for all $\epsilon\in [0, \epsilon_0]$.  

For each $\epsilon\in (0, \epsilon_0]$, we consider $F\big(U^+(h_+), k(\epsilon), \epsilon \big)$, where $k(\epsilon)=\sqrt{K(\epsilon)}$. We extend $F(c, k, \epsilon)$ by letting $\Tilde{F}$ be defined in \eqref{E:extend-F}. Then $\Tilde{F}$ is a $C^{1, \alpha}$ extension of $F$ into a neighborhood of $\big(U^+(h_+), k(\epsilon)\big)\in \mathbb{C}\times\mathbb{R}$ for any $\alpha\in [0, 1)$. Since $F(c, k, \epsilon)\in \mathbb{R}$ for all $c\geq U^+(h_+)$, we use Cauchy-Riemann equations to compute 
\begin{equation}\label{E:Dc-tilde-F}
    D_c \tilde{F}\big(U^+(h_+), k(\epsilon), \epsilon)=\begin{pmatrix} \partial_{c_R} \tilde{F}_R
 & \partial_{c_I}\tilde{F}_R \\ \partial_{c_R} \tilde{F}_I & \partial_{c_I} \tilde{F}_I    \end{pmatrix}\Big|_{\big(U^+(h_+), k(\epsilon), \epsilon \big)}=\partial_c F\big(U^+(h_+), k(\epsilon), \epsilon \big)I_{2\times 2}. 
\end{equation}
We apply the Mean Value Theorem and the smoothness of $K(\epsilon)$ to obtain that 
\begin{align*}
    \partial_cF\big(U^+(h_+), k(\epsilon), \epsilon\big)=& \partial_c F\big(U^+(h_+), k(\epsilon), 0\big)+\partial_{c\epsilon} F\big(U^+(h_+), k(\epsilon), \theta_1\big)\epsilon\\
    =& \partial_c F\big(U^+(h_+), k_0, 0\big) +\partial_{Kc} F\big(U^+(h_+), k_0+\theta_2, 0\big)\big(K(\epsilon)-k_0^2\big)+O(\epsilon)\\
    =& \partial_c F\big(U^+(h_+), k_0, 0\big) +O(\epsilon), 
\end{align*}
where $\theta_1\in (0, \epsilon)$ and $\theta_2\in \big(0, k(\epsilon)-k_0\big)$. By \eqref{E:pcF-esp-0-c+}, $\partial_c F\big(U^+(h_+), k_0, 0\big)>0$. Since $\epsilon_0 \ll 1$, 
\begin{equation}\label{E:pcF-h+-ep}
    \partial_c F\big(U^+(h_+), k(\epsilon), \epsilon\big)>0,  \quad \forall \epsilon\in [0, \epsilon_0]. 
\end{equation}
Therefore, there exist $0<\beta\ll 1$ and a $C^{1, \alpha}$ function $\mathcal{C}^+(k):[k(\epsilon)-\beta, k(\epsilon)+\beta]\rightarrow \mathbb{C}$ for any $\alpha\in [0, 1)$ such that 
\begin{align*}
    \Tilde{F}\big(\mathcal{C}^+(k), k, \epsilon\big)=0, \quad  \mathcal{C}^+\big(k(\epsilon)\big)=U^+(h_+), 
\end{align*}
\begin{equation}\label{E:cR+-l-U-}
     \mathcal{C}_R^+(k)> U^-(0).
\end{equation}
Since $F_R\in C^1$ and $\partial_{c_R}F_R\big(U^+(h_+), k(\epsilon), \epsilon\big)=\partial_c F\big(U^+(h_+), k(\epsilon), \epsilon \big)>0$, there is a $C^1$ real-valued function $\Tilde{C}(k)$ for $k$ near $k(\epsilon)$ such that 
\begin{align*}
    \Tilde{C}\big(k(\epsilon)\big)=U^+(h_+), \quad F_R\big(\Tilde{C}(k), k, \epsilon\big)=0.
\end{align*}
Since $F_I(c, k, \epsilon)=0$ for all $c\geq U^+(h_+)$, the uniqueness of solutions by the Implicit Function Theorem ensure that $\Tilde{C}(k)=\mathcal{C}^+(k)\in \mathbb{R}$ if $\Tilde{C}\geq U^+(h_+)$. 

By the Mean Value Theorem and smoothness of $K(\epsilon)$, we use the smallness of $\epsilon_0$ to obtain that 
\begin{equation}\label{E:pKF-h+kep-ep}
    \partial_K F\big( U^+(h_+), k(\epsilon), \epsilon\big) =\partial_K F\big(U^+(h_+), k_0, 0\big)+O(\epsilon)<0.
\end{equation}
Hence, we compute 
\begin{align*}
    \Tilde{C}'\big(k(\epsilon)\big)=(\mathcal{C}^+)'\big(k(\epsilon) \big)=-\frac{2k(\epsilon)\partial_K F}{\partial_c F}\big(U^+(h_+), k(\epsilon), \epsilon\big)>0. 
\end{align*}
Hence, $\Tilde{C}(k)=\mathcal{C}^+(k)>U^+(h_+)$ for $k$ slightly larger than $k(\epsilon)$. When $k\in [k(\epsilon)-\beta, k(\epsilon)\big)$, $\mathcal{C}^+_R(k)<U^+(h_+)$. In this case, we consider 
\begin{align*}
    0=\tilde{F}_I\big(\mathcal{C}^+(k), k, \epsilon \big)=F_I\big(\mathcal{C}_R^+(k), k, \epsilon \big)+\mathcal{C}_I^+(k)\partial_{c_I}\Tilde{F}_I\big(\mathcal{C}_R^+(k)+i\theta, k, \epsilon \big),
\end{align*}
where $|\theta|\in \big(0, \mathcal{C}_I^+(k)\big)$. then 
\begin{align*}
    \mathcal{C}_I^+(k)=&-\frac{F_I\big(\mathcal{C}_R^+(k), k, \epsilon\big)}{\partial_{c_I}\Tilde{F}_I\big(\mathcal{C}_R^+(k)+i\theta, k, \epsilon\big)}\\
    =& - \frac{\big(U^-(0)-\mathcal{C}_R^+(k)\big)^2Y_I^-\big(C_R^+(k), k \big)-\epsilon\big(U^+(0)-\mathcal{C}_R^+(k)\big)^2Y_I^+\big(\mathcal{C}_R^+(k), k\big)}{\partial_{c_I}\Tilde{F}_I\big(\mathcal{C}_R^+(k), k, \epsilon\big)+O\big(|\mathcal{C}_I^+(k)|^\alpha\big)+O(\epsilon)}.
\end{align*}
By \eqref{E:YI-}, \eqref{E:cR+-l-U-}, and \eqref{E:YI+} $Y_I^-\big(C_R^+(k), k \big)=0$. Hence, 
\begin{align*}
    \mathcal{C}_I^+(k)=&\frac{\epsilon\big(U^+(0)-\mathcal{C}_R^+(k)\big)^2Y_I^+\big(\mathcal{C}_R^+(k), k\big)}{\partial_{c_I}F_I\big(U^+(h_+), k(\epsilon), \epsilon\big)+O\big( |k-k(\epsilon)|^\alpha\big)}\\
    =& \frac{\epsilon\big(U^+(0)-U^+(h_+)+O(|k-k(\epsilon)|)\big)^2Y_I^+\big(\mathcal{C}_R^+(k), k\big)}{\partial_c F\big(U^+(h_+), k(\epsilon), \epsilon\big)+O\big( |k-k(\epsilon)|^\alpha\big)}.
\end{align*}
By \eqref{E:YI+}, the statement (1) is proved. 

Now we consider the case of $\partial_k F\big(U^+(h_+), k_0, \epsilon\big)=0$. Since $\partial_\epsilon F\big(U^+(h_+), k_0, \epsilon\big)>0$, there exist $0<\delta\ll 1$ and $\Gamma(k)\in C^\infty([k_0-\delta, k_0+\delta], \mathbb{R})$ such that 
\begin{align*}
    F\big(U^+(h_+), k, \Gamma(k)\big)=0, \; \Gamma(k_0)=0. 
\end{align*}
Since $\partial_K F\big(U^+(h_+), k_0, \epsilon\big)=0$, we use \eqref{E:p-KK-F} to obtain that, 
\begin{align*}
    \partial_{KK}\Gamma(k_0)=-\frac{\partial_{KK}F}{\partial_\epsilon F}\big(U^+(h_+), k_0, \epsilon\big)>0. 
\end{align*}
Hence, $\Gamma(k)\geq 0$ on $[k_0-\delta, k_0+\delta]$. Let $\epsilon_1:=\sup_{|k|\in [k_0-\delta, k_0+\delta]}\Gamma(k)$.  Then for each $\epsilon\in (0, \epsilon_1)$, there exist $k_2(\epsilon)>k_0>k_1(\epsilon)> 0$ such that 
\begin{align*}
    F\big(U^+(h_+), k_2(\epsilon), \epsilon\big)=F\big(U^+(h_+), k_1(\epsilon), \epsilon\big)=0. 
\end{align*}
Thanks to \eqref{E:p-KK-F}, we have 
\begin{align*}
    \partial_K F\big(U^+(h_+), k_2(\epsilon), \epsilon\big)<0, \quad \partial_K F\big(U^+(h_+), k_1(\epsilon), \epsilon\big)>0. 
\end{align*}
For such $\epsilon$, we apply the  bifurcation analysis  near $(U^+(h_+), k_1(\epsilon), \epsilon\big)$ and $(U^+(h_+), k_2(\epsilon), \epsilon\big)$ respectively, which is similar as the proof for the case of $\partial_K F\big(U^+(h_+), k_0, \epsilon \big)<0$. The proof is completed. 
\end{proof}

When $U^+(h_+)\notin U^-([[-h_-, 0])$ is the maximum speed of both fluids, we also discuss about the possible bifurcation near $c=U^+(h_+)$ if $\big(U^+(h_+), 0\big)$ is a neutral mode in the capillary gravity water wave problem.  

\begin{lemma}\label{L:bifurcation-U+h-k=0}
Suppose that $U^\pm$ satisfy \eqref{E:U monotone}, $U^-(0)<U^+(0)$,  $(U^\pm)'> 0$, $(U^+)''\neq 0$, and $F\big(U^+(h_+), 0, 0 \big)=0$. Then the following hold. 
\begin{enumerate}
    \item There exist $\epsilon_0>0$  and  $\Bar{C}(\epsilon)\in C^{1, \alpha}([0, \epsilon_0], \mathbb{C})$ for any $\alpha\in [0, 1)$ such that
\begin{align*}
    F\big(\Bar{C}(\epsilon), 0, \epsilon\big) =0, \; \forall \epsilon\in [0, \epsilon_0], \quad \Bar{C}(0)=U^+(h_+). 
\end{align*}
Moreover, $\Bar{C}$ satisfies that 
\begin{align*}
    \Bar{C}_R(\epsilon)<U^+(h_+),  \quad (U^+)''\Bar{C}_I(\epsilon)<0,\; \forall \epsilon\in (0, \epsilon_0]. 
\end{align*}
\item Let $K:=k^2$. If $\partial_K F\big(U^+(h_+), 0, 0\big)\leq 0$, then there exist $\epsilon_1>0$ and $k(\epsilon)\in C^\infty([0, \epsilon_1], \mathbb{R}^+\cup \{0\})$ such that for each $\epsilon\in (0, \epsilon_1]$, 
\begin{align*}
    F\big(U^+(h_+), \pm k(\epsilon), \epsilon\big)=0, \quad k(\epsilon)>0,\; \text{and}\; k(0)=0.
\end{align*}
Moreover, there exist $\delta_0>0$ and $\mathcal{C}^+\in C^{1, \beta}([k(\epsilon)-\delta_0, k(\epsilon)+\delta_0], \mathbb{C})$ for any $\beta\in [0, 1)$ such that 
\begin{align*}
    F\big(\mathcal{C}^+(k), \pm k, \epsilon\big)=0, \quad \mathcal{C}^+\big( k(\epsilon)\big)=U^+(h_+),&\quad
    \mathcal{C}^+(k)>U^+(h_+), \; \forall |k|\in \big(k(\epsilon), k(\epsilon)+\delta_0], \\
    (U^+)''\mathcal{C}_I^+(k)<0, \quad  \mathcal{C}_R^+(k)<U^+&(h_+), \; \forall |k|\in [k(\epsilon)-\delta_0, k(\epsilon)\big).
\end{align*}
In addition, for $c_I\in [0, U^+(h_+)-U^+(0)]$, $|c_R-U^+(h_+)|<U^+(h_+)-U^+(0)$, $\mathcal{C}^+(k)$ is the only root of $F(\cdot, k, \epsilon)=0$ with $|k|\in [k(\epsilon)-\delta_0, k(\epsilon)+\delta_0]$. 
\end{enumerate}
\end{lemma}

\begin{proof}
We extend $F(c, k, \epsilon)$ into a neighborhood of $\big(U^+(h_+), 0\big)\in \mathbb{C}\times \mathbb{R}$ by letting $\Tilde{F}$ be defined in \eqref{E:extend-F}. Using Cauchy-Riemann equations, $D_c\Tilde{F}\big(U^+(h_+), 0, 0\big)$ is in the form of \eqref{E:Dc-tilde-F}. By \eqref{E:pcF-esp-0-c+}, $\partial_c F\big(U^+(h_+), 0, 0\big)>0$. Hence, there exist $0<\epsilon_0\ll 1$ and a $C^{1, \alpha}$ function $\Bar{C}(\epsilon):[-\epsilon_0, \epsilon_0]\rightarrow \mathbb{C}$ for any $\alpha\in [0, 1)$ such that 
\begin{align*}
    \Tilde{F}\big(\Bar{C}(\epsilon), 0, \epsilon\big)=0, \quad \Bar{C}(0)=U^+(h_+).
\end{align*}
By a similar argument as the proof in Lemma \ref{L:bifurcation-U+h}, we obtain that, $F\big( \Bar{C}(\epsilon), 0,\epsilon\big)=\Tilde{F}\big(\Bar{C}(\epsilon), 0, \epsilon\big)$. Using \eqref{E:p-ep-F-h+-0}, we obtain that $(\Bar{C})'(0)=-\frac{\partial_\epsilon F}{\partial_c F}\big(U^+(h_+), 0, 0\big)<0$. Then, for each $\epsilon\in (0, \epsilon_0)$, $\Bar{C}_R(\epsilon)<U^+(h_+)$. Moreover, from the Mean Value Theorem, for such $\epsilon$, there exists $|\theta|\in (0, \epsilon)$ such that 
\begin{align*}
    0=\Tilde{F}_I\big(\Bar{C}(\epsilon), 0, \epsilon\big)=F_I\big(\Bar{C}_R(\epsilon), 0, \epsilon\big)+\Bar{C}_I(\epsilon)\partial_{c_I}\Tilde{F}_I\big(\Bar{C}_R(\epsilon)+i\theta, 0, \epsilon\big). 
\end{align*}
Similar as the argument in the proof of Lemma \ref{L:bifurcation-U+h}, using the $C^{1, \alpha}$ regularity of $F$ and $\Bar{C}(\epsilon)$, we obtain that 
\begin{align*}
    \Bar{C}_I(\epsilon)=  -\frac{F_I\big(\Bar{C}_R(\epsilon), 0, \epsilon\big)}{\partial_{c_I}\Tilde{F}_I\big(\Bar{C}_R(\epsilon)+i\theta, 0, \epsilon\big)}
    = \frac{\epsilon Y_I^+\big(\Bar{C}_R(\epsilon), 0\big)\big(U^+(0)-U^+(h_+)+O(\epsilon)\big)^2}{\partial_c F\big(U^+(h_+), 0, 0\big)+O(\epsilon)}. 
\end{align*}
The proof of statement (1) is completed. 

Let $K:=k^2$. To prove the statement (2), we notice that the proof for the case of $\partial_K F\big(U^+(h_+), 0, 0\big)<0$ is similar as the one in Lemma \ref{L:bifurcation-U+h}. It remains to consider the case of $\partial_K F\big(U^+(h_+), 0, 0\big)=0$. By \eqref{E:p-ep-F-h+-0}, there exist $0<K_0\ll 1$ and $\Gamma\in C^\infty([-K_0, K_0], \mathbb{R})$  such that
\begin{align*}
    F\big(U^+(0), k, \Gamma(K)\big)=0, \quad \Gamma(0)=0.
\end{align*}
Moreover, by a direct computation, $\partial_{KK}\Gamma(0)=-\frac{\partial_{KK}F}{\partial_\epsilon F}\big(U^+(h_+), 0, 0 \big)>0$. Let $k_0:=\sqrt{K_0}$. For any $k\in (0, k_0)$, we have $\Gamma(K)>0$ and $\Gamma'(K)>0$. Let $\epsilon_0:=\sup_{(0, k_0)}\Gamma (k)$,  $K(\epsilon):=\Gamma^{-1}(\epsilon)>0$, and $k(\epsilon):=\sqrt{K(\epsilon)}$.  Then for each $\epsilon\in (0, \epsilon_0)$, 
\begin{align*}
    &F\big(U^+(h_+), k(\epsilon), \epsilon\big)=0,\\
    \partial_c F\big(U^+(h_+), k(\epsilon),& \epsilon\big)=\partial_c F\big(U^+(h_+), 0, 0\big)+O(\epsilon)>0. 
\end{align*}
Moreover, we use \eqref{E:partial-k-Y+} to compute that
\begin{align*}
    \partial_K F\big(U^+(h_+), k(\epsilon), \epsilon\big)=&\partial_K F\big(U^+(h_+), k(\epsilon), 0\big)-\epsilon\partial_K Y^+\big(U^+(0)-U^+(h_+)\big)^2\\
    < & \partial_K F\big(U^+(h_+), 0, 0\big)+O(\epsilon)<0. 
\end{align*}
The rest of the proof is similar as the proof for Lemma \ref{L:bifurcation-U+h}(1a). \end{proof}

Using a similar strategy, we can also prove the following result for small $\epsilon$ if the one fluid free boundary problem has a neutral mode at $\big(U^+(0), k_0\big)$ with $k_0>0$ or $\big(U^+(0), 0\big)$. 

\begin{lemma}\label{L:bifurcation-U+0}
Suppose that $U^\pm$ satisfy \eqref{E:U monotone},  $(U^\pm)'>0$, $U^-(0)<U^+(0)$, $(U^\pm)''\neq 0$, and $F\big(U^+(0), k_0, 0\big)=0$ for some $k_0>0$. Then the following hold. 
\begin{enumerate}
    \item Suppose that $\partial_k F\big(U^+(0), k_0, 0\big)\neq 0$. Then there exist $\epsilon_0>0$ such that for each $\epsilon\in (0, \epsilon_0)$, there exist $\delta_0>0$, $k(\epsilon)\in C^{1, \alpha}([0, \epsilon_0], \mathbb{R}^+)$, and $\mathcal{C}^+(k)\in C^{1, \alpha}\big([k(\epsilon)-\delta_0, k(\epsilon)+\delta_0]\big)$ for any $\alpha\in [0, 1)$ such that  
\begin{align*}
    F\big(U^+(0), \pm k(\epsilon), \epsilon\big)=0,& \quad k(0)=k_0, \\
    F\big(\mathcal{C}^+(k), \pm k, \epsilon\big)=0, \; \forall |k|\in [k(\epsilon)-\delta_0,\;& k(\epsilon)+\delta_0], \text{and} \; \mathcal{C}^+\big(k(\epsilon)\big)=U^+(0). 
\end{align*}
Let $\rho:=\min\{U^+(h_+)-U^+(0), U^+(0)-U^-(0)\}$. For $c_I\in [0, \rho]$ and $|c_R-U^+(0)|\leq \rho$, $\mathcal{C}^+(k)$ is the only root of $F(\cdot, k, \epsilon)=0$ with $|k|\in [k(\epsilon)-\delta_0, k(\epsilon)+\delta_0]$. Moreover, the following hold. 
\begin{enumerate}
    \item If $\partial_k F\big(U^+(0), k_0, 0\big)<0$, then 
    \begin{enumerate}
        \item $k'(\epsilon)>0$ and $k(\epsilon)>k_0$;
        \item $\mathcal{C}^+(k)<U^+(0)$ for any $|k|\in [k(\epsilon)-\delta_0, k(\epsilon)\big)$ and 
        \begin{align*}
            (U^+)''\mathcal{C}_I^+(k)<0,\quad \mathcal{C}_R^+(k)>U^+(0), \; \forall |k|\in \big(k(\epsilon), k(\epsilon)+\delta_0].
        \end{align*}
    \end{enumerate}
    \item $If \partial_k F\big(U^+(0), k_0, 0\big)>0$, then 
    \begin{enumerate}
        \item $k'(\epsilon)<0$ and $k(\epsilon)<k_0$;
        \item $\mathcal{C}^+(k)<U^+(0)$ for any $|k|\in \big(k(\epsilon), k(\epsilon)+\delta_0]$ and 
        \begin{align*}
           (U^+)''\mathcal{C}_I^+(k)<0 , \quad \mathcal{C}_R^+(k)>U^+(0), \; \forall |k|\in [k(\epsilon)-\delta_0, k(\epsilon)\big). 
        \end{align*}
    \end{enumerate}
\end{enumerate}
\item Suppose that $\partial_k F\big(U^+(0), k_0, 0\big)=0$. Then there exist $\epsilon_1>0$ such that for each $\epsilon\in (0, \epsilon_1]$, there exist $k_2(\epsilon)>k_0>k_1(\epsilon)\geq 0$, $\delta>0$, and $C^{1, \alpha}$ functions $\mathcal{C}_{1, 2}:[k_{1, 2}(\epsilon)-\delta, k_{1, 2}(\epsilon)+\delta]\rightarrow \mathbb{C}$ for any $\alpha\in [0, 1)$ such that 
\begin{align*}
    F\big(U^+(0), \pm k_{1, 2}(\epsilon), \epsilon\big)=0, \quad \mathcal{C}_{1, 2}^+\big(\pm k_{1, 2}(\epsilon)\big)=U^+(0), \\
    F\big(\mathcal{C}_{1, 2}^+(k), \pm k, \epsilon\big)=0, \quad \forall |k|\in [k_{1, 2}(\epsilon)-\delta, k_{1, 2}(\epsilon)+\delta]. 
\end{align*}
In addition, $\mathcal{C}_{1, 2}^+(k)$ satisfy that 
\begin{align*}
    \mathcal{C}^+_1(k)<U^+(0),\; \forall |k|\in \big(k_1(\epsilon), k_1(\epsilon)+\delta],& \quad \mathcal{C}^+_2(k)<U^+(0), \;\forall |k|\in [k_2(\epsilon)-\delta, k_2(\epsilon)\big), \\
    \mathcal{C}^+_{1_ R}(k)>U^+(0), \quad (U^+)''&\mathcal{C}^+_{1_I}(k)<0, \; \forall |k|\in [k_1(\epsilon)-\delta, k_1(\epsilon)\big),\\
    \mathcal{C}^+_{2_ R}(k)>U^+(0), \quad (U^+)''&\mathcal{C}^+_{2_I}(k)<0, \; \forall |k|\in \big(k_2(\epsilon), k_2(\epsilon)+\delta].
\end{align*}
Let $\rho:=\min \{U^+(h_+)-U^+(0), U^+(0)-U^-(0)\}$. Then for $|c_R-U^+(0)|<\rho$ and $c_I\in [0, \rho]$, $\mathcal{C}_1^+(k)$ is the unique root of $F(\cdot, k, \epsilon)=0$ with $|k|\in [k_1(\epsilon)-\delta, k_1(\epsilon)+\delta]$ and $\mathcal{C}_2^+(k)$ is the unique root of $F(\cdot, k, \epsilon)=0$ with $|k|\in [k_2(\epsilon)-\delta, k_2(\epsilon)+\delta]$. 
\end{enumerate}
\end{lemma}

\begin{lemma}
Suppose that $U^\pm$ satisfy \eqref{E:U monotone}, $U^-(0)<U^+(0)$, $(U^\pm)'>0$, $(U^+)''<0$, and $F\big(U^+(0), 0, 0\big)=0$. Then the following hold. 
\begin{enumerate}
    \item There exist $\epsilon_0>0$ and  $C^{1, \alpha}$ function $\Bar{C}(\epsilon):[0, \epsilon_0]\rightarrow \mathbb{C}$ for any $\alpha\in [0, 1)$ such that 
    \begin{align*}
        F\big(\Bar{C}(\epsilon), 0, \epsilon\big)=0,\quad \Bar{C}(0)=U^+(0),\\
        \Bar{C}(\epsilon)\in \big(U^-(0), U^+(0)\big), \; \forall
        \epsilon\in (0, \epsilon_0].
    \end{align*}
    \item Let $K:=k^2$. If $\partial_K F\big(U^+(0), 0, 0\big)\leq 0$, then there exist $\epsilon_1>0$ and smooth function $k(\epsilon):[0, \epsilon_1]\rightarrow \mathbb{R}$ such that for each $\epsilon\in (0, \epsilon_1]$, 
    \begin{align*}
        F\big(U^+(0), \pm k(\epsilon), \epsilon\big)=0,\quad  k(\epsilon)>0,\; \text{and}\; k(0)=0.
    \end{align*}
    In addition, there exist $\delta_0>0$ and $C^{1, \beta}$ function $\mathcal{C}^+(k):[k(\epsilon)-\delta_0, k(\epsilon)+\delta_0]\rightarrow \mathbb{C}$ for any $\beta\in [0, 1)$ such that 
    \begin{align*}
        F\big(\mathcal{C}^+(k), \pm k, \epsilon\big)=0,\quad  \mathcal{C}^+\big(k(\epsilon)\big)=U^+(0),& \quad 
        \mathcal{C}^+(k)<U^+(0), \; \forall |k|\in [k(\epsilon)-\delta_0, k(\epsilon)\big), \\
        (U^+)''\mathcal{C}_I^+(k)<0, \quad  \mathcal{C}_R^+(k)&>U^+(0), \; \forall |k|\in \big(k(\epsilon), k(\epsilon)+\delta_0].
        \end{align*}
    Let $\rho:=\min\{U^+(0)-U^-(0), U^+(h_+)-U^+(0)\}$. For $c_I\in [0, \rho)$, $|c_R-U^+(0)|<\rho$, $\mathcal{C}^+(k)$ is the unique root of $F(\cdot, k, \epsilon)=0$ with $|k|\in [k(\epsilon)-\delta_0, k(\epsilon)+\delta_0]$. 
\end{enumerate}
\end{lemma}

In the following, we prove Theorem \ref{T:e-small-g*}.

\noindent\textit{Proof of Theorem \ref{T:e-small-g*}}. 
When $k=0, \epsilon=0$ and $c>U^-(0)$, $y^-$ is in the form of \eqref{E:Y-k-0outU-}. Then $F(c, 0, 0)=0$ corresponds to 
\begin{align*}
    \big(\int_{-h_-}^0 \frac{1}{(U^--c)^2}dx_2\big)^{-1}=g.
\end{align*}
Hence, $c_0$ is the only root of $F(\cdot, 0, 0)=0$ in $\{c_R>U^-(0)\}$. 

We first consider the case of $c_0>U^+(h_+)$. By \eqref{E:pcF-esp-0-c+}, 
\begin{equation}\label{E:F-0h+-ep0<0}
    F(c, 0, 0)<F\big(U^+(h_+), 0, 0\big)<F(c_0, 0, 0)=0, \; \forall c\in \big(U^-(0), U^+(h_+)\big).
\end{equation}
There exists $0<\epsilon_0\ll 1$ such that for each $\epsilon\in (0, \epsilon_0)$, $|F(c, 0, \epsilon)|>0$ where $c\in \{U^+(0), U^+(h_+) \}$. Let $m(c)$ be defined in \eqref{D:m(c)}. Then \eqref{E:A-F<0} holds for small $\epsilon$. Let $c^+(k)$ be the one obtained in Lemma \ref{L:c-large k}. Then, Lemma \ref{L:no singular mode} implies that $c^+(k)$ can be extended for all $k\in \mathbb{R}$ and it is the only singular and non-singular mode in $\{c_R>U^-(0)\}$. Statement (1a) is proved. 

Next, we shall prove statement (1b).  Let $K:=k^2$. Suppose that $c_0\in \big(U^+(0), U^+(h_+)]$. Let $c^+_0(k)$ be the $c^+(k)$ obtained in Lemma \ref{L:c-large k} when $\epsilon=0$. By Remark \ref{R:c-pm-extend}, $c^+_0(k)$ can be extended for all $k\in \mathbb{R}$. Then there exists $k_0>0$ such that $c^+_0(k_0)=U^+(h_+)$. We use \eqref{E:p-KK-F}  and compute that \begin{align*}
    \partial_K F\big(U^+(h_+), k_0, 0\big)<\partial_K F\big(U^+(h_+), 0, 0\big)=\big(U^-(0)-U^+(h_+) \big)\partial_K Y^-\big(U^+(h_+), 0\big)-\frac{\sigma}{\rho^-}.
\end{align*}
Using \eqref{E:pkY-0-c-out} and the assumption \eqref{A:u+h+-bifurcation}, we obtain that
\begin{align*}
    \partial_K F\big(U^+(h_+), 0, 0\big)\leq \int_{-h_-}^0 \big(U^-(x_2)-U^+(h_+)\big)^2dx_2-\frac{\sigma}{\rho^-}<0.
\end{align*}
This implies that  
\begin{equation}\label{E:pKF-h+-ep0}
    \partial_K F\big(U^+(h_+), k_0, 0\big)<0.
\end{equation}
For $\epsilon=0$, we use \eqref{E:pcF-esp-0-c+} to obtain that 
\begin{align*}
    (c^+_0)'(k_0)=-\frac{2k \partial_K F}{\partial_c F}\big( U^+(h_+), k_0, 0\big)>0. 
\end{align*}
Since $c^+_0(k)$ is analytic in $k\in \mathbb{R}$, this implies that $k_0$ is the unique root of $F\big(U^+(h_+), \cdot, 0\big)=0$. By Lemma \ref{L:bifurcation-U+h}, there exist $\delta_0>0$, $k(\epsilon)\in C^{1, \beta}\big([0, \epsilon_1], \mathbb{R}^+\big)$, and $\mathcal{C}^+(k)\in C^{1, \alpha}\big([k(\epsilon)-\delta_0, k(\epsilon)+\delta_0], \mathbb{C} \big)$ for any $\alpha, \beta\in [0, 1)$ such that $k(\epsilon)>k_0$, $F\big(U^+(h_+), k(\epsilon), \epsilon\big)=0$, $F\big(\mathcal{C}^+(k), k, \epsilon \big)=0$, and $\mathcal{C}^+(k(\epsilon))=U^+(h_+)$. Since $(U^+)''<0$, $\mathcal{C}_I^+(k)>0$, $\mathcal{C}_R^+(k)<U^+(h_+)$ for any $k\in [k(\epsilon)-\delta_0, k(\epsilon)\big)$ and $\mathcal{C}^+(k)>U^+(h_+)$ for any $k\in \big(k(\epsilon), k(\epsilon)+\delta_0]$. By Remark \ref{R:c-pm-extend}, $c^+(k)$ obtained in Lemma \ref{L:c-large k} can be extended to be an analytic function for all $|k|>k(\epsilon)$ and $C^{1, \alpha}$ for any $\alpha\in [0, 1)$ in $k\geq k(\epsilon)$ such that $F\big(c^+(k), k, \epsilon\big)=0$. The uniqueness of solutions by the Implicit Function Theorem ensures that $\mathcal{C}^+(k)=c^+(k)$ for any $k\in \big(k(\epsilon), k(\epsilon)+\delta_0]$. Moreover, there exists an open $B_0$ satisfying $\big(U^+(h_+), k(\epsilon)\big)\in B_0\subset \mathbb{C}$ such that $F(c, k, \epsilon)\neq 0$ for any $c\in \overline{B_0}$ and $|k|\in [k(\epsilon)-\delta_0, k(\epsilon)-\frac{\delta_0}{4}]$. By the smallness of $\epsilon_1$, \eqref{E:pKF-h+kep-ep} holds. Then \eqref{E:p-KK-F} implies that $k(\epsilon)$ is the unique root of $F\big(U^+(h_+), \cdot, \epsilon\big)=0$ for all $k\in \mathbb{R}$. We use \eqref{E:pkY-0-c-out},  the assumption \eqref{A:u+h+-bifurcation},  $(U^-)'>0$, and $U^-(0)<U^+(0)$ to compute that 
\begin{align*}
    \partial_K F(c, 0, 0)=&\big(U^-(0)-c\big)^2\partial_K Y^-(c, 0)-\frac{\sigma}{\rho^-}\\ \leq & \int_{-h_-}^0 \big(U^-(x_2)-c\big)^2dx_2-\frac{\sigma}{\rho^-}\\
    \leq& \int_{-h_-}^0 \big(U^-(x_2)-U^+(h_+)\big)^2dx_2-\frac{\sigma}{\rho^-}<0, \; \forall c\in \big(U^-(0), U^+(h_+)\big). 
\end{align*}
Hence, by \eqref{E:p-KK-F}, for any $k\in \mathbb{R}$ and $c\in \big(U^-(0), U^+(h_+)\big)$, $\partial_K F(c, k, 0)<\partial_K F(c, 0, 0)<0$. Since $0<\epsilon_1\ll 1$, for any $\epsilon\in [0, \epsilon_1)$, $\partial_K F(c, k, \epsilon)<0$ for all $k\in \mathbb{R}$ and $c\in \big(U^-(0), U^+(h_+)\big)$. Since $F(c, k, \epsilon)=F(c, k, 0)+\epsilon g$ and $c_0>U^+(0)$, using the smallness of $\epsilon_1$ and \eqref{E:F-0h+-ep0<0}, we obtain that 
\begin{equation}\label{E:F-0-k-ep-neq0}
    F(c, k, \epsilon)<0, \; \forall\epsilon\in [0, \epsilon_1), \; \forall k\in \mathbb{R}, \; \forall c\in \big(U^-(0), U^+(0)].
\end{equation}
Since $F_I(c, k, \epsilon)\neq 0$ for any $c\in U^+\big((0, h_+)\big)$, $F(c, k, \epsilon)\neq 0$ for any $c\in U^+\big([0, h_+]\big)$, $|k|\leq  k(\epsilon)-\frac{\delta_0}{4}$. Hence, there exists an open set $B_1$ satisfying that $U^+([0,h_+])\subset B_1\subset \mathbb{C}$ and $B_1\cap B_0\neq \emptyset$ such that $F(c, k, \epsilon)\neq 0$ for any $c\in B_1$ and $|k|\leq  k(\epsilon)-\frac{\delta_0}{4}$.  Because of \eqref{E:F-0-k-ep-neq0} and Lemma \ref{L:nlm-e-0}, there exists $\gamma>0$ such that $|F(c, k, 0)|> 0$ for any $c\in B_2:=\{U^+(0)-\gamma\leq  c_R\leq U^+(0)+\gamma\}$ and $k\in \mathbb{R}$. The smallness of $\epsilon_1$ implies that $|F(c, k, \epsilon)|> 0$ for $c\in B_2$ and $k\in \mathbb{R}$. Let $B_3:=[U^+(0), U^+(h_+)]\times [0, U^+(h_+)-U^-(0)]\subset \mathbb{C}$. Consider $\Omega:=B_3\setminus (B_0\cup B_1\cup B_2)$. Thanks to Theorem \ref{T:semicircle}, we obtain that $F(c, k, \epsilon)\neq 0$ for any $c\in \partial \Omega$ and $|k|\leq k(\epsilon)-\frac{\delta_0}{2}$. We consider $n(k)$ which is defined in \eqref{E:degree}. Lemma \ref{L:bifurcation-U+h} implies that $n\big(k(\epsilon)\big)=1$. Since $n(k)$ is continuous for $|k|\leq k(\epsilon)-\frac{\delta_0}{2}$, $n(k)\equiv 1$ for $|k|\leq k(\epsilon)-\frac{\delta_0}{2}$.  Lemma \ref{L:bifurcation-U+h} shows that $\mathcal{C}^+(k)$ is the unique zero of $F(\cdot, k, \epsilon)$ for $|k|\in [k(\epsilon)-\delta_0, k(\epsilon)+\delta_0]$. Hence, $n(k)\equiv 1$ for all $|k|\leq k(\epsilon)$. This implies that $\mathcal{C}^+(k)$ can be extended to be a $C^{1, \alpha}$ function (for any $\alpha\in [0, 1)$) for all $|k|\leq k(\epsilon)$ and analytic except at $k=k(\epsilon)$. Let $c^+(k)=\mathcal{C}^+(k)$ for $|k|<k(\epsilon)$. Statement (1b) is proved. The proof for statement (1c) is similar. 

To prove statement (2), let
\begin{align*}
    g_\#:=\max\{F(c, k, 0)+g\}. 
\end{align*}
Then statement (2a) directly follows from Theorem 1.1 in \cite{LZ} and a similar proof as the one for Statement (1a). The proof for Statement (2b) is similar as the one for Statement (1b). 
\hfill $\square$

\end{section}

\begin{section}{Acknowledgments}
The author would like to thank  Prof. Chongchun Zeng, for introducing the problem and suggestion. 
\end{section}

\bibliographystyle{abbrv}
\bibliography{reference}

\begin{thebibliography}{10}

\bibitem{Oli}
O.~Bühler, J.~Shatah, S.~Walsh, and C.~Zeng.
\newblock On the wind generation of water waves.
\newblock {\em Arch. Ration. Mech. Anal.}, 222:827 -- 878, 2016.

\bibitem{ChengShkollerCoutand}
C.-h. Cheng, S.~Shkoller, and D.~Coutand.
\newblock On the motion of vortex sheets with surface tension in
  three-dimensional euler equations with vorticity.
\newblock {\em Communications on Pure and Applied Mathematics}, 61:1715 --
  1752, 12 2008.

\bibitem{Dra}
P.~Drazin and W.~Reid.
\newblock Hydrodynamic stability.
\newblock {\em Cambridge: Cambridge University Press.}, 2004.

\bibitem{FH98}
S.~Friedlander and L.~Howard.
\newblock Instability in parallel flows revisited.
\newblock {\em Stud. Appl. Math.}, 101(1):1--21, 1998.

\bibitem{How}
L.~Howard.
\newblock Note on a paper of john w. miles.
\newblock {\em J. Fluid. Mech.}, 10:509--512, 1961.

\bibitem{Hur}
V.~Hur and Z.~Lin.
\newblock Unstable surface waves in running water.
\newblock {\em Comm. Math Phys.}, 282:733 -- 796, 2008.

\bibitem{HL13}
V.~M. Hur and Z.~Lin.
\newblock Erratum to: {U}nstable surface waves in running water [mr2426143].
\newblock {\em Comm. Math. Phys.}, 318(3):857--861, 2013.

\bibitem{Bea}
T.~Y.~H. J.~T.~Beale and J.~S. Lowengrub.
\newblock Growth rates for the linearized motion of fluid interfaces away from
  equilibrium.
\newblock {\em Comm. Pure Appl. Math.}, 46:1269--1301, 1993.

\bibitem{janssen}
P.~Janssen.
\newblock {\em The interaction of ocean waves and wind}.
\newblock Cambridge University Press, 2004.

\bibitem{Lin03}
Z.~Lin.
\newblock Instability of some ideal plane flows.
\newblock {\em SIAM J. Math. Anal.}, 35(2):318--356, 2003.

\bibitem{LZ}
X.~Liu and C.~Zeng.
\newblock Capillary gravity water waves linearized at monotone shear flows:
  eigenvalues and inviscid damping.
\newblock {\em arXiv preprint arXiv:2110.12604.}, 2021.

\bibitem{Mil}
J.~W. Miles.
\newblock On the generation of surface waves by shear flows.
\newblock {\em J. Fluid Mech.}, 3:185--204, 1957.

\bibitem{miles592}
J.~W. Miles.
\newblock On the generation of surface waves by shear flows. part 2.
\newblock {\em J. Fluid Mech.}, 6(4):568–582, 1959.

\bibitem{miles593}
J.~W. Miles.
\newblock On the generation of surface waves by shear flows part 3.
  kelvin-helmholtz instability.
\newblock {\em J. Fluid Mech.}, 6(4):583–598, 1959.

\bibitem{Ray}
L.~Rayleigh.
\newblock On the stability or instability of certain fluid motions.
\newblock {\em Pro. London Math. Soc.}, 9:57--70, 1880.

\bibitem{Ren}
M.~Renardy and Y.~Renardy.
\newblock On the stability of inviscid parallel shear flows with a free
  surface.
\newblock {\em Math. Fluid Mech.}, 15:129--137, 2012.

\bibitem{SZ}
J.~Shatah and C.~Zeng.
\newblock A priori estimates for fluid interface problems.
\newblock {\em Communications on Pure and Applied Mathematics: A Journal Issued
  by the Courant Institute of Mathematical Sciences}, 61(6):848--876, 2008.

\bibitem{Shatah2}
J.~Shatah and C.~Zeng.
\newblock Local well-posedness for fluid interface problems.
\newblock {\em Arch. Rational Mech. Anal.}, 199:653–705, 2011.

\bibitem{Kelvin}
W.~Thomson.
\newblock Xlvi. hydrokinetic solutions and observations.
\newblock {\em The London, Edinburgh, and Dublin Philosophical Magazine and
  Journal of Science}, 42(281):362--377, 1871.

\bibitem{Yih}
C.-S. Yih.
\newblock Surface waves in flowing water.
\newblock {\em J. Fluid. Mech.}, 51:209--220, 1972.

\end{thebibliography}
\end{document}